\documentclass{amsart}[11pt,leqno] \usepackage{amsmath,amssymb,enumerate,amscd,tikz,url}

\theoremstyle{plain}
\newtheorem{theo}[equation]{Theorem}
\newtheorem{lem}[equation]{Lemma}
\newtheorem{cor}[equation]{Corollary}
\newtheorem{prop}[equation]{Proposition}
\theoremstyle{definition}
\newtheorem{Def}[equation]{Definition}
\newtheorem{Ex}[equation]{Example}
\newtheorem{Rem}[equation]{Remark} 
\newtheorem{Not}[equation]{Notation}

\newcommand{\F}{{\mathbb F}}

\newcommand{\Q}{{\mathbb Q}}

\newcommand{\T}{{\mathbb T}}
\newcommand{\W}{{\mathbb W}}
\newcommand{\Z}{{\mathbb Z}}

\newcommand{\bfa}{{\mathbf a}}

\newcommand{\bfd}{{\mathbf d}}

\newcommand{\bfL}{{\bf L}}
\newcommand{\bfs}{{\bf s}}
\newcommand{\bft}{{\bf t}}
\newcommand{\bfu}{{\bf u}}

\newcommand{\Kappa}{\boldsymbol \kappa}
\newcommand{\Mu}{\boldsymbol \mu}

\newcommand{\ga}{{\mathfrak a}}
\newcommand{\gb}{{\mathfrak b}}

\newcommand{\gE}{{\mathfrak E}}

\newcommand{\gf}{{\mathfrak f}}
\newcommand{\gh}{{\mathfrak h}}
\newcommand{\gI}{{\mathfrak I}}

\newcommand{\gN}{{\mathfrak N}}

\newcommand{\gp}{{\mathfrak p}}
\newcommand{\gP}{{\mathfrak P}}
\newcommand{\gR}{{\mathfrak R}}

\newcommand{\cD}{{\mathcal D}}
\newcommand{\cE}{{\mathcal E}}

\newcommand{\cI}{{\mathcal I}}

\newcommand{\cO}{{\mathcal O}}
\newcommand{\cP}{{\mathcal P}}

\newcommand{\cS}{{\mathcal S}}

\newcommand{\cU}{{\mathcal U}}
\newcommand{\cV}{{\mathcal V}}
\newcommand{\cW}{{\mathcal W}}

\newcommand{\rL}{{\rm L}}
\newcommand{\rM}{{\rm M}}
\newcommand{\rF}{{\rm F}}
\newcommand{\rR}{{\rm R}}

\newcommand{\rV}{{\rm V}}

\newcommand{\un}[1]{\underline{#1}}

\newcommand{\Artin}{ {\rm Artin}}
\newcommand{\Aut}{\operatorname{Aut}}

\newcommand{\Cor}{\operatorname{Cor}}
\newcommand{\CW}{\operatorname{CW}}

\newcommand{\End}{\operatorname{End}}
\newcommand{\Ext}{\operatorname{Ext}}
\newcommand{\Image}{\operatorname{Image} \, }

\newcommand{\rk}{\operatorname{rank}}
\newcommand{\Gal}{\operatorname{Gal}}
\newcommand{\Hom}{\operatorname{Hom}}
\newcommand{\len}{\operatorname{length}}
\newcommand{\GL}{{\rm GL}}
\newcommand{\Mat}{{\rm Mat}}
\newcommand{\SL}{{\rm SL}}
\newcommand{\GS}{{\rm GSp}}
\newcommand{\SP}{{\rm Sp}}
\newcommand{\spn}{{\rm span}}
\newcommand{\Sym}{{\rm Sym}}
\newcommand{\sign}{\operatorname{sign}}

\newcommand{\res}{\operatorname{res}}

\newcommand{\ord}{\operatorname{ord}}

\newcommand{\0}{\vec{0}}

\newcommand{\CWplus}{\, \dot{+} \, }

\newcommand{\lexp}[2]{{^{#1}\hspace{-.7 pt}{#2}}}

\newcommand{\bino}[2]{\left( \begin{smallmatrix} {#1} \\ {#2} \end{smallmatrix} \right)}

\newcommand{\ov}[1]{\overline{#1}}

\newcommand{\fdeg}[2]{[{#1}\!:\!{#2}]}

\newcommand{\lr}[1]{\langle{#1}\rangle}

\makeatletter
\renewcommand*\l@subsection{\@tocline{2}{0pt}{30pt}{0pt}{}}
\makeatother

\author[A. Brumer]{Armand Brumer}
\address{Department of Mathematics, Fordham University, Bronx, NY 10458}
\email{brumer@fordham.edu}
\author[K. Kramer]{Kenneth Kramer}
\address{Department of Mathematics, Queens College (CUNY), Flushing, NY 11367;  Department of Mathematics, The Graduate Center of CUNY, New York, NY 10016}
\email{kkramer@qc.cuny.edu}
\subjclass[2010]{Primary 11G10; Secondary 14K15, 11R37, 11S31}
\keywords{semistable abelian variety, group scheme, Honda system, conductor, paramodular conjecture.}

\begin{document}
\title{Certain abelian varieties bad at only one prime}

\begin{abstract}
An abelian surface $A_{/\Q}$ of prime conductor $N$ is {\em favorable} if its 2-division field $F$ is an $\cS_5$-extension with ramification index 5 over $\Q_2$.  Let $A$ be favorable and let $B$ be any semistable abelian variety of dimension $2d$ and conductor $N^d$ such that $B[2]$ is filtered by copies of $A[2]$.  We give a sufficient class field theoretic criterion on $F$ to guarantee that $B$ is isogenous to $A^d$.  

As expected from our paramodular conjecture, we conclude that there is one isogeny class of abelian surfaces for each conductor in $\{277, 349,461,797,971\}$.  The general applicability of our criterion is discussed in the data section.\end{abstract}

\thanks{Research of the second author was partially supported by a PSC-CUNY Award, cycle 44, jointly funded by The Professional Staff Congress and The City University of New York. }

\maketitle 

\tableofcontents

\numberwithin{equation}{section}

\section{Introduction} 
Let $\gI_d(S)$ be the set of  isogeny classes of simple abelian varieties over $\Q$ of dimension $d$ with good reduction outside $S,$ a finite set  of primes.  By \cite{Falt1}, $\gI_d(S)$ is finite and it is empty when $S$ is by \cite{Abr,Fon4}.  All curves of genus 2 with good reduction outside 2 are found in \cite{MeSm, Sma}, yielding 165 isogeny classes of Jacobians.  Factors of $J_0(2^{10})$ and Weil restrictions of elliptic curves over quadratic fields provide an additional 50 members of  $\gI_2(\{2\})$, but the complete determination of $\gI_2(\{2\})$ is still open.

For {\em semistable} abelian varieties, Fontaine's non-existence result has been slightly extended \cite{BK1, BK2, BK4, Cal, Sch2}. It is much more challenging  to find all isogeny classes when some exist. 

 In a beautiful sequence of papers \cite{Sch2,Sch3,Sch4}, Schoof shows that for  $S=\{N\}$ with $N\le 19$ and $N = 23$ (resp.\! $S=\{3,5\}$), the classical modular variety $J_0(N)$ (resp.$\,J_0(15)$) is the only  simple semistable abelian variety of arbitrary dimension, up to isogeny.  To apply Faltings' isogeny theorem on abelian varieties, Schoof introduces a general result on $p$-divisible groups whose  constituents belong to a category $\un{C}$ of finite flat group schemes.  For the reader's convenience, the statement is included here as Theorem \ref{pdiv}.  For a suitable choice of category $\un{D}$, depending on $S$, Schoof determines all  simple objects and their extensions by one another.  Because the Odlyzko bounds are used, the sets $S$ to which  these methods apply are severely limited. 
 
In fact, given a finite set $S$ of primes, it seems challenging to decide whether the dimension of the simple semistable abelian varieties good outside $S$ is bounded.

This paper grew out of the desire  to check the uniqueness of certain isogeny classes for larger conductors.  Another motivation was  to provide additional evidence  for our conjecture.

\vspace{2 pt}

\noindent{\bf  Paramodular Conjecture}(\cite{BK4}).  {\em Let $K(N)$ be the paramodular group of level $N$.  There is a one-to-one correspondence:}  \vspace{  2pt}

\centerline{\fbox{\parbox{150 pt}{\begin{center} isogeny classes of  abelian surfaces $A_{/\Q}$ of conductor $N$ with $\End_{\Q} A=\Z$ \end{center}}}
$\longleftrightarrow$
\fbox{\parbox{140 pt}{\begin{center}   weight 2 {\em non-lifts} $f$ on $K(N)$, with rational eigenvalues, up to scalar multiplication  \end{center} }}  }   

\vspace{2 pt}

\noindent {\em in which the $\ell$-adic representation of  $\T_\ell(A)\otimes\Q_\ell$ and that associated to $f$ are isomorphic for any $\ell$ prime to $N,$ so that the L-series  of $A$ and $f$ agree.}

\vspace{2 pt}

The $L$-series of abelian surfaces  of $\GL_2$-type are understood via classical elliptic modular forms, while our conjecture treats all other abelian surfaces. It is verified in \cite{BDPS,JLR1} for the Weil restrictions of modular elliptic curves over quadratic fields, not isogenous to their conjugates.   It  is also compatible with twists \cite{JLR2}. 

To ensure that we are not in the endoscopic case, we consider prime conductors.  By \cite[Theorem 3.4.11]{BK4}, an abelian surface of prime conductor  is isogenous to a Jacobian.   For each $N$ in $\{277, 349,461,797,971\}$, the space of weight 2 non-lift paramodular forms on $K(N)$ is one-dimensional \cite{PoYu1}, so our conjecture predicts that there should be exactly one isogeny class of abelian surfaces of conductor $N$.  In \cite{BK4}, we proved that $277$ is the smallest prime conductor.   For each $N$ listed above, there is a unique Galois module structure available for $A[2]$.  For those $N$,  $\Q(A[2])$ must be the Galois closure of a {\em favorable} quintic field as defined below.

\begin{Def}  \label{fav}
Let $N$ be an odd prime.  A quintic extension $F_0/\Q$ of discriminant $\pm 16N$ is {\em favorable} if the prime over 2 has ramification index 5.  A {\em favorable polynomial} is any minimal polynomial for a favorable quintic field.   An abelian surface $A$ of prime conductor $N$ is {\em favorable} its 2-division field $\Q(A[2])$ is the Galois closure of a favorable quintic field. 
\end{Def}

We note some pleasant properties of favorable quintic fields.

\begin{prop} \label{Fproperties}
Let $F$ be the Galois closure of a favorable quintic field $F_0$ of discriminant $d_0 =16N^*$ with $N^* = \pm N$.  Then:
\begin{enumerate}[{\rm i)}]
\item $\Gal(F/\Q)$ is isomorphic to the symmetric group $\cS_5$.  At each prime $\gN \vert N$, the inertia group $\cI_\gN = \cI_\gN(F/\Q)$ is generated by a transposition.   \vspace{2 pt}
\item The completion $F_\gP$ of $F$ at each prime $\gP \vert 2$ is isomorphic to $\Q_2(\Mu_5,\sqrt[5]{2})$ and the decomposition group $\cD_\gP = \cD_\gP(F/\Q)$ is the Frobenius group of order $20$.  The sign of $N^*$ is determined by $N^* \equiv 5 \, (8)$.   \vspace{2 pt}
\item There is only one prime over $2$ in the subfield $K_{20}$ of $F$ fixed by $\Sym \{3,4,5\}$. \vspace{2 pt}
\item If $A$ is a favorable abelian surface, then $A[2]_{|\Z_2}$ is absolutely irreducible and biconnected over $\Z_2$.
\end{enumerate}
\end{prop}

\begin{proof}
i) Since $N$ exactly divides $d_0$, only one prime say $\gN_0$ over $N$ ramifies in $F_0/\Q$ and the $\cO_{F_0}$-ideal generated by $N$ factors as $(N) = \gN_0^e \ga$, where $\ga$ is an ideal prime to $\gN_0$ and $e > 1$.  If $f$ is the residue degree of $\gN_0$ then $N^{(e-1)f}$ divides $d_0$, so $e = 2$, $f=1$ and the other primes over $N$ are unramified in $F_0/\Q$.  Thus the completion $F_\gN$ is $\Q_N(\sqrt{d_0})$ and $\cI_\gN$ has order 2.  Since $\cI_\gN$ acts non-trivially on $\sqrt{d_0}$, it is generated by a transposition. A transposition and a 5-cycle generate $\cS_5$.  \vspace{2 pt}

ii) By assumption, $F_\gP/\Q_2$ has tame ramification of degree 5 and thus contains $\Q_2(\Mu_5,\sqrt[5]{2})$.  Since $\cD_\gP$ is solvable, $F = \Q_2(\Mu_5,\sqrt[5]{2})$.  Any Frobenius automorphism at $\gP$ is a 4-cycle, so it acts non-trivially on $\sqrt{d_0}$ and therefore $N^* \equiv 5 \bmod{8}$.   

\vspace{2 pt}

iii) There are no transpositions in $\cD_\gP$, so $D_\gP \cap \Sym\{3,4,5\}$ is trivial.  Since $\fdeg{K_{20}}{\Q} = 20$, there is only one prime over $2$ in $K_{20}$.   \vspace{2 pt}

iv) Since $\cD_\gP$ acts on $A[2]$ via its unique 4-dimensional absolutely irreducible $\F_2$-representation, $A[2]_{|\Z_2}$ has no \'{etale} or multiplicative constituents.
\end{proof}

A {\em favorable} $\cS_5$-field is the Galois closure of a favorable quintic field.    The Jacobian of  a genus 2 curve $C$ is favorable only if $C$ has a model $y^2=f(x)$ with $f$ favorable, but $C$ might have bad reduction outside $N$. 

In general, $L$ is a {\em stem field} for $M$ if $M$ is the Galois closure of $L/\Q$.  A {\em pair-resolvent} for an $\cS_5$-field $F$ is a subfield $K$ fixed by the centralizer of a transposition in $\cS_5$.  Then $K$ is well-defined up to isomorphism and is a stem field for $F$.  If $r_1$ and $r_2$ are distinct roots of a quintic polynomial $f$ with splitting field $F$, we can take $K = \Q(r_1+r_2)$, the fixed field of $\Sym\{1,2\} \times \Sym\{3,4,5\}$.  There is only one prime $\gp$ over 2 in $K$ by Proposition \ref{Fproperties}(iii).  Let $\Omega_K^{(a)}$ be the maximal elementary $2$-extension of $K$ of modulus $\gp^a\!\cdot\!\infty$, i.e., the compositum of all quadratic extensions of $K$ with that modulus.  Write $rk_a$ for the rank of $\Gal(\Omega_K^{(a)}/K)$.

The following is a restatement of Theorem \ref{main}. 

\begin{theo} \label{IntroThm}
Let $A$ be a favorable abelian surface of conductor $N$ and let $K$ be a pair-resolvent field for $F = \Q(A[2])$.  Suppose that $B$ is a semistable abelian variety of dimension $2d$ and conductor $N^d,$ with $B[2]$  filtered by copies of $A[2]$.  If $rk_2=0$ and $rk_4\le 1,$ then $B$ is isogenous to $A^d$.   If $B$ is a surface, it is isogenous to $A$.
\end{theo}

For the proof, we first construct suitable categories $\un{E}$, chosen so that extensions of the simple objects $\cE$ in $\un{E}$ can be identified.  A description of their extensions as group schemes over $\Z_p$ is obtained via Honda systems.   For global applications, assume that $p=2$ and $\Q(\cE)$ is a favorable $\cS_5$-field.  Monodromy at $N$ restricts the extensions $\cW$ of $\cE$ by $\cE$ as group schemes over $\Z[\frac{1}{2N}]$.  A comparison with local data determines when $\cW$ prolongs to a group scheme over $\Z[\frac{1}{N}]$ and leads to our class field theoretic criterion for the control of $\Ext^1_{\un{E}}(\cE,\cE)$ required by Schoof's theorem.  Ray class field information, difficult to reach over $F$,  becomes accessible over the degree 10 field $K$.  Moreover, we found that Theorem \ref{IntroThm} and  Proposition \ref{MakeGaFields} have no analog for other intermediate fields of $F/\Q$.  A more detailed overview of our paper follows.

The category $\un{E}$ of finite flat $p$-group schemes over $\Z[\frac{1}{N}]$ defined in \S3 is motivated by necessary conditions for an abelian variety $B$ to be isogenous to a product of given semistable abelian varieties $A_i$.  It is essential to impose conductor bounds at $N$, without which Theorem \ref{pdiv} does not apply, as indicated in Remark \ref{obstrep}.  Thanks to Proposition \ref{Ext}, we deduce in Theorem \ref{mypdiv} that it suffices to study the subgroup $\Ext^1_{[p],\un{E}}(\cE,\cE)$ consisting of classes of extensions $\cW$ of $\cE$ by $\cE$ such that $p\cW = 0$.

We review group schemes and Honda systems over the ring of Witt vectors $\W$ of a finite field $k$ of characteristic $p$ in \S\ref{GpScheme}.  In \S \ref{Honda}, finite Honda systems are used to classify absolutely simple biconnected finite flat group schemes $\cE$ of rank $p^4$ over $\W$ and describe the classes $[\cW]$ in $\Ext^1_{[p],\Z_p}(\cE,\cE)$.   We give the structure of the associated Galois modules $E$ and $W$ in \S \ref{FP} and obtain a conductor bound for the elementary abelian extension $K(W)/K(E)$ in Proposition \ref{CondExpP}.  The latter improves on Fontaine's bound in our case, cf.\! Remark \ref{FB}.  

In \S \ref{Punchline}, we restrict to $p = 2$ and give a class field theoretic condition equivalent to the vanishing of $\Ext^1_{[2],\un{E}}(\cE,\cE)$ in Proposition \ref{ext2}.   Its proof exploits the following ingredients: (i) monodromy at $N$, to determine the matrix groups available for $\Gal(\Q(W)/\Q)$ as $W$ runs over the extensions of $E$ by $E$ as Galois modules; (ii) conductor bounds at $p=2$, as described above and (iii) rigidification in \S \ref{CornerSection} and \eqref{CorGa} of the cocycles corresponding to local and global extensions of $E$ by $E$, to check whether they are compatibile, as needed for patching.

Appendix \ref{condapp} contains several general facts required for the determination of abelian conductor exponents in our applications.

In Appendix \ref{DataSection}, we apply Theorem \ref{IntroThm} to all the favorable quintic fields with $N$ at most 25000 to obtain Table \ref{Fields}.  In particular, there is a unique isogeny class of abelian surfaces for each conductor $N$ in $\{277, 349,461,797,971\}$.  Curious about the wider applicability of our criterion, we studied the fields corresponding to 276109 favorable abelian surfaces of prime conductor at most $10^{10}$ found by an ad-hoc search.  We were surprised to discover that the uniqueness, up to isogeny, in Theorem \ref{IntroThm} holds uniformly for about 11.8\% of those fields.  The data is summarized in Table  \ref{data}. 

In our companion paper \cite{BK5}, extensions $\cW$ of exponent $p^2$ are studied and new ``full image" results for certain subgroups of $\GS_{2g}(\Z_2)$ generated by transvections are obtained.  As a consequence, if $A$ is a favorable abelian surface,  then  $\Q(A[4])$ is an elementary 2-extension of rank  11  over $\Q(A[2])$ with carefully controlled ramification.  In Table \ref{Fields}, we also indicate the fields for which no favorable abelian surface can exist because there is no candidate for its 4-division field.  

The authors wish to express their gratitude to the anonymous referees for their extremely careful reading of the manuscript.  Their valuable suggestions helped us clarify and improve the exposition.

Write $\ov{K}$ for the algebraic closure  of $K$ and $G_K=\Gal(\ov{K}/K)$.  For any local or global field $K$, let $\cO_K$ be its ring of integers. If $L/K$ is a Galois extension of number fields, let $\cD_v(L/K)$ and $\cI_v(L/K)$ be the decomposition and inertia subgroups of $\Gal(L/K)$ at a place $v$ of $L$.   We also use $v$ for its restriction to each subfield of $L$.   When the local extension $L_v/K_v$ is abelian, $\gf_v(L/K)$ denotes the abelian conductor exponent of $L_v/K_v$.   Write $\gf_v(V)$ for the Artin conductor exponent of a finite $\Z_p[\cD_v]$-module $V$.

\section{Some review of group schemes} \label{GpScheme}
Let $R$ be a Dedekind domain with quotient field $K$.  Calligraphic letters are used for finite flat group schemes $\cV$ over $R$ and the corresponding Roman letter for the Galois module $V=\cV(\ov{K})$. The order of $\cV$ is the rank over $R$ of its affine algebra, or equivalently the order of the finite abelian group $V=\cV(\ov{K})$.

By the following result of Raynaud (\cite{Con1}, \cite{Ray1}), group schemes occurring as subquotients of known group schemes can be treated via their associated Galois modules.  Thus, the generic fiber functor induces an isomorphism between the lattice of finite flat closed $R$-subgroup schemes of $\cV$ and that of finite flat closed $K$-subgroup schemes of $\cV_{|K}$, where $K$ is the field of fractions of $R$.  The following results will be used without explicit reference.

\begin{lem}\label{CoRay} 
Let $R$ be a Dedekind domain with quotient field $K$ and let $\cV$ be a finite flat group scheme over $R$ with generic fiber $V=\cV_{|K}.$  If $W=V_2/V_1$ is a subquotient of $V$, for closed immersions of finite flat $K$-group schemes $V_1 \hookrightarrow V_2 \hookrightarrow V$, there are unique closed immersions of finite flat $R$-group schemes $\cV_1 \hookrightarrow \cV_2\hookrightarrow \cV$, such that $V_i=\cV_{i \, |K}$, and there is a unique isomorphism $\cV_2/\cV_1\simeq \cW$ compatible with $(\cV_2/\cV_1)_{|K}\simeq W.$ 
\end{lem}  

Let $p$ be a prime not dividing $N,$ $R=\Z[\frac{1}{N}]$, $R' = \Z[\frac{1}{pN}]$ and let $\un{Gr}$ be the category of $p$-primary finite flat group schemes over $R.$  Let $\un{C}$ be the category of triples $(\cV_1,\cV_2,\theta)$ where $\cV_1$ is a finite flat $\Z_p$-group scheme, $\cV_2$ a finite flat $R'$-group scheme and $\theta\!: \,\cV_1\otimes_{\Z_p} \Q_p\to \cV_2\otimes_{R'}\Q_p$ an isomorphism of $\Q_p$-group schemes. Then Proposition 2.3 of \cite{Sch1} asserts that the functor 
$\un{Gr}\to \un{C}$ taking the $R$-group scheme $\cV$ to $(\cV\otimes_R\Z_p,\cV\otimes_R R', id\otimes_R\Q_p)$ is an equivalence of categories.   We can identify $\cV \otimes_R R'$ with the Galois module $V$, since $\cV$ is \'{e}tale over $R'$. For  objects $\cV_1,\,\cV_2$ of $\un{Gr}$, the Mayer-Vietoris sequence of \cite[Cor.\! 2.4]{Sch1} specializes to:
\begin{equation} \label{MV}  
\begin{array}{l} 
  \Hom_{\Q_p}(V_1,V_2)  \leftarrow \Hom_{\Z_p}(\cV_1,\cV_2) \times  \Hom_{R'}(\cV_1,\cV_2) \leftarrow \Hom_R(\cV_1,\cV_2) \leftarrow  0  \vspace{2 pt} \\
 \hspace{15 pt}     \delta \, \downarrow    \vspace{2 pt} \\
   \Ext^1_{R}(\cV_1,\,\cV_2)   \rightarrow  \Ext^1_{\Z_p}(\cV_1,\cV_2) \times\Ext^1_{R'}(\cV_1,\cV_2) \rightarrow   \Ext^1_{\Q_{p}}(V_1,V_2).
\end{array}
\end{equation}

\begin{cor}
Let $\cV_1$ and $\cV_2$ be finite flat group schemes over $R = \Z[\frac{1}{N}]$ with $\cV_1,\,\cV_2$ biconnected over $\Z_p$.  The following natural maps are isomorphisms:
$$
\Hom_R(\cV_1,\cV_2)\to \Hom_{\Gal}(V_1,V_2) \quad\text{and}\quad \Ext^1_{R}(\cV_1,\,\cV_2)\to \Ext^1_{\Gal}(V_1,\,V_2).
$$
If $\cV$ is a group scheme over $R$ and $V_{\vert \Q_p}$ is absolutely irreducible, then 
$$
\End_{\Q_p}(\cV) = \End_{R'}(\cV) = \F_p, \quad \text{and} \quad \End_{R}(\cV)=\F_p.  
$$
In addition, $\delta=0$ in \text{\eqref{MV}} with $\cV_1 = \cV_2 = \cV$. 
\end{cor} 

\begin{proof}
The first claim follows from \eqref{MV} and a theorem of Fontaine quoted in  \cite[Thm.~1.4]{Maz}. For the second, use Schur's Lemma and a diagram chase.
\end{proof}

\medskip

We next review some basic material on Honda systems found in \cite{BrCo,Con2,Fon3}.   Let $p$ be a prime, $k$ a perfect field of characteristic $p>0$, $\W = \W(k)$ the Witt vectors and $K$ its field of fractions. Let $\sigma:\W \to \W$  be the Frobenius automorphism characterized by $\sigma(x)\equiv x^p  \pmod{p}$ for $x$ in $\W.$ The Dieudonn\'e ring $D_k=\W[\rF,\rV]$ is generated by the Frobenius operator $\rF$ and Verschiebung operator $\rV$. We have $\rF\rV=\rV\rF=p$, $\rF a=\sigma(a)\rF$ and $\rV a=\sigma^{-1}(a)\rV$ for all $a$ in $\W.$ 

A {\em Honda system} over $\W$ is a pair $(\rM,\rL)$ consisting of a finitely generated free $\W$-module $\rM$, a $\W$-submodule $\rL$ and a Frobenius semi-linear injective endomorphism $\rF\!: \, \rM \to \rM$ with $p\rM\subseteq \rF(\rM)$ and the induced map $\rL/p\rL\to \rM/\rF\rM$ an isomorphism. If $\rF$ is topologically nilpotent, then $(\rM,\rL)$ is {\em connected}.  Since $\rM$ is torsion free, $\rM$ becomes a $D_k$-module with $\rV=p\rF^{-1}.$

 A {\em finite Honda system} over $\W$ is a pair $(\rM,\rL)$ consisting of a left $D_k$-module $\rM$ of finite  $\W$-length and a $\W$-submodule $\rL$ with $\rV\!: \, \rL\to \rM$ injective and the induced map $\rL/p\rL\to \rM/\rF\rM$ an isomorphism. If $\rF$ is nilpotent on $\rM$, then $(\rM,\rL)$ is {\em connected}.  Morphisms are defined in the obvious manner. If $(\rM,\rL)$ is a Honda sytem then $(\rM/p^n\rM,\rL/p^n\rL)$ is a finite Honda sytem.

Honda systems owe their importance to the following fundamental result. 
\begin{theo}[Fontaine]\label{FoHon}{\rm  Let $k$ be a perfect field of characteristic $p>0$.
 \begin{enumerate}[{\rm i)}] 
 \item If $p>2,$ there is a natural anti-equivalence of categories $G\rightsquigarrow ({\bf D}(G_k),\bfL(G))$ from the category of $p$-divisible groups over $\W$ to that of Honda systems  (${\bf D}(G_k)$ is the Dieudonn\'e  module of $G_k).$ The same holds for $p=2$ if we restrict to connected objects on both sides. 
 \item If $p>2$, there is a natural anti-equivalence of categories  from the category of finite flat $p$-primary group schemes over $\W$ to that of finite Honda systems and the same holds for $p=2$ if we restrict to connected objects on both sides. 
 \item The cotangent space of $G_k$ at the origin is ${\bf D}(G_k)/\rF{\bf D}(G_k).$ 
 \item Both anti-equivalences respect extensions of $k$. Moreover, if $G$ is a  $p$-divisible group over $\W,$ then $({\bf D}(G_k)/(p^n),\bfL(G)/(p^n))$ is naturally identified with the finite Honda system associated with $G[p^n]$ for all $n\ge 1$.
 \end{enumerate}}
\end{theo}

\begin{lem} \label{pM}
Let $(\rM,\rL)$ be a Honda system of exponent $p$.  Then $\rM=\rL + \rF\rM$ is a direct sum, $\ker \rF=\rV\rL=\rV\rM$, $\dim \ker \rF = \dim \rL$ and $\ker \rV=\rF\rM$.
\end{lem}

\begin{proof} 
Since $\rL/p\rL \to \rM/\rF\rM$ is an isomorphism, $\rM=\rL + \rF\rM$ is a direct sum and 
$$
\dim \rM=\dim \rF\rM+\dim \rL=\dim \rM-\dim\ker \rF+\dim \rL.
$$
Hence $\dim \ker \rF = \dim \rL$ and equality holds for each inclusion in $\rV\rL \subseteq \rV\rM \subseteq \ker \rF$ because $\rV_{\vert \rL}$ is injective.
In addition, 
$$
\dim \rL =\dim \rV\rL = \dim \rV\rM = \dim \rM-\dim \ker \rV,
$$
so $\rM= \rL +  \ker \rV$ is a direct sum and the inclusion $\rF\rM \subseteq \ker \rV$ is an equality. 
\end{proof}

Let $\widehat{\CW}_k$ denote the formal $k$-group scheme associated to the {\em Witt covector} group functor $\CW_k$, cf.\! \cite{Con2,Fon3}.  When $k'$ is a finite extension of $k$ and $K'$ is the field of fractions of $W(k')$, we have $\CW_k(k') \simeq K'/W(k')$.  For any $k$-algebra $R$ and $\W = W(k)$, let $D_k = \W[\rF,\rV]$ act on elements $\bfa = ( \dots, a_{-n}, \dots, a_{-1},a_{0})$ of $\CW_k(R)$ by $\rF \bfa = ( \dots,a_{-n}^p, \dots, a_{-1}^p,a_{0}^p)$,  $\rV \bfa = ( \dots, a_{-(n+1)}, \dots, a_{-2},a_{-1})$ and $\dot{c} \, \bfa = ( \dots, c^{p^{-n}}a_{-n}, \dots, c^{p^{-1}}a_{-1},ca_{0})$, where $\dot{c}$ in $\W$ is the Teichm\"uller lift of $c$.  Note that such lifts generate $\W$ as a topological ring.

The Hasse-Witt exponential map is a homomorphism of additive groups: 
$$
\xi: \, \widehat{CW}_k(\cO_{\ov{K}}/p\cO_{\ov{K}}) \to \ov{K}/p\cO_{\ov{K}} 
\quad \text{by} \quad 
(\dots,a_{-n},\dots,a_{-1},a_0) \mapsto \sum \, p^{-n} \, \tilde{a}_{-n}^{p^n}
$$
independent of the choice of lifts $\tilde{a}_{-n}$ in $\cO_{\ov{K}}$.  If $\cU$ is the group scheme of a Honda system $(\rM,\rL)$, the points of the Galois module $U$ correspond to $D_k$-homomorphisms $\varphi\!: \, \rM \to  \widehat{CW}_k(\cO_{\ov{K}}/p\cO_{\ov{K}})$ such that $\xi(\varphi(\rL)) = 0$ and we say that $\varphi$ {\em belongs to} $\cU$.  The action of $G_K$ on $U(\ov{K})$ is induced from its action on $\widehat{CW}_k(\cO_{\ov{K}}/p\cO_{\ov{K}})$.  

We write $\CWplus$ for the usual Witt covector addition \cite[p.\! 242]{Con2} and state some related elementary facts.  For $q$ a power of $p$ and $x, y$ in $\ov{k}$, the congruence $\Phi_q(x,y) \equiv ((\tilde{x}+\tilde{y})^q-\tilde{x}^q-\tilde{y}^q)/q \pmod{p\cO_{\ov{K}}}$ defines a unique, possibly non-integral element of $\ov{K}/p\cO_{\ov{K}}$, independent of the choices of lifts $\tilde{x}, \tilde{y}$ in $\cO_{\ov{K}}$.   The binomial theorem yields the following estimate:

\begin{lem} \label{ppower}
$
\ord_p((\tilde{x} + \tilde{y})^q - \tilde{x}^q - \tilde{y}^q) \ge 1 + q \min\{\ord_p(\tilde{x}), \ord_p(\tilde{y}) \}.
$ \qed  
\end{lem}

It is convenient to write  $(\0,x_{-n},\dots,x_0)$ for the element $(\dots,0,0,x_{-n},\dots,x_0)$ in $\widehat{CW}_k(\cO_{\ov{K}}/p\cO_{\ov{K}})$.  A routine calculation using the formulas in \cite{Abr,Con2} gives:

\begin{lem} \label{Wittadd} 
Addition in $\widehat{CW}_k(\cO_{\ov{K}}/p\cO_{\ov{K}})$ specializes to:
$$
(\0,u_4,u_3,u_2,u_1,u_0) \CWplus (\0,v_2,v_1,v_0)  = (\0,u_4,u_3,u_2+v_2,w_1,w_0)
$$
where $w_1 = u_1+v_1-\Phi_p(u_2,v_2)$ and 
$$
w_0=u_0+v_0+ \frac{1}{p} (u_1^p+v_1^p)-\Phi_{p^2}(u_2,v_2)- \frac{1}{p} (u_1+v_1-\Phi_p(u_2,v_2))^p. \hspace{20 pt} \qed
$$
\end{lem}

\section{The new categories} \label{OurCat}
After a review of local conductors, we introduce the categories in which extension classes will be studied.
   
Fix distinct primes $N$ and $p$ and let $K$ be a finite extension of $\Q_N$.  If  $L/K$ is a Galois extension, let $\cD=\cD(L/K)$ be its Galois group and $\cI=\cI(L/K)$ its inertia subgroup.  When $\cI$ acts tamely on the finite $\Z_p[\cD]$-module $V$, its Artin conductor exponent is given by $\gf_N(V) = {\len}_{\Z_p} V/V^{\cI}$.  If $$0 \to V_1 \to V \to V_2 \to 0$$ is an exact sequence of finite $\Z_p[\cD]$-modules, then $\gf_N(V)\ge \gf_N(V_1)+ \gf_N(V_2)$.

\vspace{2 pt}

Let $A$ be an abelian variety over $\Q_N$ with semistable bad reduction and let $\T_p(A)$ denote its $p$-adic Tate module.   We freely use results of Grothendieck \cite{Gro}, reviewed in \cite{BK1}.  The $p^\infty$-division field $\Q_N(A[p^\infty])$ depends only on the isogeny class of $A$, so is shared by the dual variety $\widehat{A}$.  The inertia subgroup $\cI$ of $\Gal(\Q_N(A[p^\infty])/\Q_N)$ is pro-$p$ cyclic and $(\sigma-1)^2(\T_p(A)) = 0$ for any topological generator $\sigma$ of $\cI$.  The fixed space $M_f(A)= \T_p(A)^{\cI}$  is a $\Z_p$-direct summand $\T_p(A)$ and the toric space $M_t(A)$ is the $\Z_p$-submodule of $\T_p(A)$ orthogonal to $M_f(\widehat{A})$ under the natural pairing of $\T_p(A)$ with $\T_p(\widehat{A})$.  Moreover, $(\sigma-1)(\T_p(A))$ has finite index in $M_t(A)$.  The conductor exponent of $A$ at $N$, denoted $\gf_N(A)$, is the $\Z_p$-rank of $\T_p(A)/M_f(A)$.  Equivalently,  we have
$
\gf_N(A) = \rk_{\Z_p} M_t(A) = \rk_{\Z_p} (\sigma-1)(\T_p(A)).
$

\begin{lem}  \label{Level1}
Suppose that $\gf_N(A[p]) = \gf_N(A)$.  Then $\gf_N(A[p^n]) =n \, \gf_N(A[p])$ for all $n \ge 1$ and $(\sigma-1)(\T_p(A)) = M_t(A)$.
\end{lem}

\begin{proof}
In the following diagram
$$
(\sigma-1)(A[p^n]) \xleftarrow{\ov{\pi}} \frac{(\sigma-1)(\T_p(A))}{(\sigma-1)(\T_p(A)) \cap p^n \T_p(A)}  \xrightarrow{\ov{\text{\j}}} M_t(A)/p^n M_t(A), 
$$
$\ov{\pi}$ is  an isomorphism induced by the natural projection $\pi\!: \, \T_p(A) \to A[p^n]$ and $\ov{\text{\j}}$ is an injection induced by the inclusion $j\!: \, (\sigma-1)(\T_p(A)) \to M_t(A)$.  Since $M_t(A)$ is a $\Z_p$-direct summand of $\T_p(A)$, we have $M_t(A)/p^n M_t(A) \simeq (\Z/p^n)^f$, where $f = \gf_N(A)$ and thus
\begin{equation} \label{nf}
nf = \len_{\Z_p} M_t(A)/p^n M_t(A)  \ge \gf_N(A[p^n]) \ge n \, \gf_N(A[p]),
\end{equation}
using super-additivity of conductors for the last inequality.  By assumption, the left and right sides of \eqref{nf} are equal, so $\gf_N(A[p^n]) =n \, \gf_N(A[p])$.  Then $\ov{\text{\j}} \circ \ov{\pi}^{-1}$ is an isomorphism and $(\sigma-1)(\T_p) = M_t(A)$ upon passage to the limit.
\end{proof}

\begin{Def} \label{ourcat} 
Let $\Sigma = \{\cE_i \, \vert \, 1 \le i \le s\}$ be a collection of finite flat group schemes over $\Z[\frac{1}{N}]$ such that:
\begin{enumerate}[i)]
\item $\cE_i$ is biconnected over $\Z_p$ for all $i$ and
\item the Galois modules $E_i$ are absolutely simple and pairwise non-isomorphic.
\end{enumerate}
Given $\Sigma$, a category $\un{E}$ of finite flat group schemes $\cV$ over $\Z[\frac{1}{N}]$ is a $\Sigma$-{\em category} if the following properties are satisfied:
\begin{enumerate}
\item[{\bf E1.}] Each composition factor of  $\cV$ is isomorphic to some $\cE_i$ with $1\le i\le s$.  \vspace{2 pt}
\item[{\bf E2.}] If $\sigma_v$ generates inertia at $v \vert N$, then $(\sigma_v-1)^2$ annihilates $V=\cV(\ov{\Q})$.  \vspace{2 pt}
\item[{\bf E3.}] If $n_i$ is the multiplicity of $E_i$ in the semi-simplification $V^{ss}$ of $V$, then 
$$
\gf_N(V) = f_N(V^{ss}) = \sum n_i \gf_N(E_i).
$$
\end{enumerate}
A collection of semistable abelian varieties $ A_i$, good outside $N$, is $\Sigma$-{\em favorable} if $\End A_i=\Z$, the $\cE_i = A_i[p]$ satisfy (i) and (ii)  and $\gf_N(A_i) = \gf_N(E_i)$ for $1 \le i \le s$.
\end{Def}

In particular, a favorable abelian surface $A$ is $\Sigma$-favorable with $\Sigma = \{A[2]\}$. 

\begin{lem} \label{E3properties}
If $0 \to \cW \to \cV \to \ov{\cV} \to 0$ is an exact sequence of finite flat group schemes and $\cV$ is in $\un{E}$, then $\cW$ and $\ov{\cV}$ also are in $\un{E}$.
\end{lem}

\begin{proof}
By super-additivity of conductors  and {\bf E3} for $V$, we have
$$
\gf_N(V^{ss}) = \gf_N(W^{ss}) + \gf_N(\ov{V}^{ss}) \le \gf_N(W) + \gf_N(\ov{V}) \le \gf_N(V) = \gf_N(V^{ss}).
$$
Hence {\bf E3} is valid for both $W$ and $\ov{V}$.  The rest is clear.
\end{proof}

Lemma \ref{E3properties} implies that $\un{E}$ is a full subcategory of the category of $p$-primary group schemes over $\Z[\frac{1}{N}]$, closed under taking products, closed flat subgroup schemes and quotients by closed flat subgroup schemes.   As in  \cite{Sch2}, this guarantees that $\Ext_{\un{E}}^1$ is defined.   Note that Schoof had introduced {\bf E2} for his categories $\un{D}$, as a consequence of semistability.  

\begin{Rem} \label{E2E3}
 If $V^{ss} = \oplus \, n_i E_i$, the conductor of $V$ satisfies the lower bound $\gf_N(V) \ge \sum n_i \gf_N(E_i)$, while {\bf E3} imposes equality.  Remark \ref{obstrep} indicates the need for {\bf E3} when $\dim E_i > 2$ and shows that it is not needed when $\dim E_i=2$.   Moreover, {\bf E2} implies {\bf E3} if $\dim E_i = 2 \, \gf_N(E_i)$ for all $i$.  Indeed, $V/V^{\langle\sigma_v\rangle}\simeq (\sigma_v-1)V \subseteq V^{\langle\sigma_v\rangle}$ by {\bf E2}.  Write $\ell(V) = \len_{\Z_p} V$.  Then
$$
\begin{array}{rcl}
2 \, \sum n_i \gf_N(E_i) &=& \sum n_i \dim_{\F_p} E_i \, = \, \ell(V) \, = \, \ell( (\sigma_v-1)V ) + \ell( V^{\lr{\sigma_v}}) \vspace{4 pt}\\
                                          &\ge&  2\, \ell((\sigma_v-1)V) \, = \, 2\, \gf_N(V) \, \ge \, 2 \, \sum n_i \gf_N(E_i).
\end{array}
$$
Hence $f_N(V) = \sum n_i f_N(E_i)$.
\end{Rem}

\begin{Ex}
In Theorem \ref{mypdiv} below, $\gf_N(B)$ is as small as possible, given the structure of $B[p]^{ss}$.  But minimality of conductor does not guarantee that $B$ is semistable.  For example, \cite{Set} gives an elliptic curve over $K = \Q(\sqrt{37})$ with everywhere good reduction:  
$$
C\!: \, y^2 - \epsilon y = x^3 + \frac{1}{2}\left(3\epsilon+1\right) x^2 + \frac{1}{2}\left(11\epsilon+1\right) x,   \hspace{20 pt} \epsilon = 6 + \sqrt{37}.
$$
If $B$ is its Weil restriction to $\Q$, then $B$ has good reduction outside $N = 37$ and $\gf_N(B) = 2$ by Milne's conductor formula \cite[Prop.\! 1]{Mil}.  Let $A$ be any of the elliptic curves over $\Q$ of conductor 37.  These curves share the same group scheme $\cE = A[2]$ and $\gf_N(E) = 1$.   Let $\un{E}$ be the $\Sigma$-category with $\Sigma = \{\cE\}$.  Then $B[2]^{ss} = \cE \oplus \cE$ and so {\bf E3} holds.  But $B$ has potential good reduction at $N$ and inertia at $v \vert N$ acts on $\T_2(B)$ through the finite quotient $\Gal(\Q_N(\sqrt{37})/\Q_N)$, so {\bf E2} fails.  Note that $B$ was considered earlier in \cite{Shi}.
\end{Ex}

We recall the following elegant theorem of Schoof on $p$-divisible groups.

\begin{theo}[{\cite[Theorem 8.3]{Sch2}}]   \label{pdiv} 
Let $\un{C}$ be a full subcategory of the category of $p$-primary group schemes over $O=\Z[\frac{1}{N}]$, closed under taking products, closed flat subgroup schemes and quotients by closed flat subgroup schemes.  Let $G=\{G_n\}$ and $H=\{H_n\}$ be $p$-divisible groups over $O,$ with $G_n$ and $H_n$ in $\un{C}$.  Suppose that
\begin{enumerate}[{\rm i)}]
\item $R=\End(G)$ is a discrete valuation ring with uniformizer $\pi$ and residue field $k=R/\pi R;$
\item the map $\displaystyle{\Hom_O(G[\pi],G[\pi])\xrightarrow{\delta}  \Ext^1_{\un{C}}(G[\pi],G[\pi])},$ induced by the cohomology sequence of $0 \to G[\pi]\to G[\pi^2]\to G[\pi]\to 0,$ is an isomorphism of one-dimensional $k$-vector spaces;
\item each $H_n$ admits a filtration by flat closed subgroup schemes whose successive subquotients are isomorphic to $G[\pi]$. \end{enumerate}
Then $H$ is isomorphic to $G^r$ for some $r$.
\end{theo}

\begin{Not}
If $\cV$ and $\cW$ in $\un{E}$ are annihilated by $p$, write $\Ext_{[p], \un{E}}^1(\cV,\cW)$ for the subgroup of $\Ext_{\un{E}}^1(\cV,\cW)$ whose classes are represented by extensions killed by $p$. 
\end{Not}

\begin{theo} \label{mypdiv}
Let $\{A_i  \, \vert \, 1 \le i \le s \}$ be a $\Sigma$-favorable collection of abelian varieties and let  $\un{E}$ be the $\Sigma$-category with $\Sigma = \{\cE_i = A_i[p] \, \vert \, 1 \le i \le s\}$.  If $B$ is isogenous to $\prod_i A_i^{n_i},$ then subquotients of $B[p^r]$ are in $\un{E}$.   Conversely, suppose that $B$ is semistable and $\gf_N (B)=  \sum n_i \gf_N(E_i)$, where $B[p]^{ss}=\oplus \, n_i\cE_i$.  If
\vspace{2  pt}

\centerline{  {\bf E4:} $\Ext^1_{[p],\un{E}}(\cE_i,\cE_j)=0$ for all $1 \le i \le j \le s,$}

\vspace{2  pt}
\noindent then $B$ is isogenous to $\prod A_i^{n_i}.$ 
\end{theo}

\begin{proof} 
Lemmas \ref{Level1} and \ref{E3properties} imply the first claim.   For the converse, it suffices by Lemma \ref{E3properties}, to show that $B[p^r]$ belongs to $\un{E}$.   Property {\bf E1} is clear and {\bf E2} follows from semistability.  By super-additivity of conductors, 
$$
\sum n_i \gf_N(E_i) = \gf_N(B[p]^{ss}) \le \gf_N(B[p]) \le \gf_N(B) = \sum n_i \gf_N(E_i).
$$
Thus each weak inequality above is an equality and so  
$$
\gf_N(B[p^r]) = r \, \gf_N(B[p]) = \sum r \, n_i \gf(E_i)
$$ 
by Lemma \ref{Level1}.  Hence {\bf E3} holds and $B[p^r]$ is in $\un{E}$.

Assuming {\bf E4}, the Lemma below enables us to define isotypic decompositions of the finite flat group schemes in $\un{E}$.  Thus the $p$-divisible group of $B$ is the product of its isotypic $p$-divisible subgroups $H^{(i)}$.   If $G^{(i)}$ is the $p$-divisible group of $A_i$, then $\End(G^{(i)})=\Z_p$ by the theorem of Faltings proving Tate's conjecture.   Vanishing of $\Ext^1_{[p], \un{E}}(\cE_i,\cE_i)$ and Proposition \ref{Ext} imply that  $\Ext^1_{\un{E}}(\cE_i,\cE_i) = \F_p$ thanks to the existence of the extension $0 \to \cE_i \to A_i[p^2] \to \cE_i \to 0$.  Theorem \ref{pdiv} now gives $H_i\simeq G_i^{n_i}$ and so the $p$-divisible group of $B$ is isomorphic to that of $\prod A_i^{n_i}$.  Conclude by Faltings' theorem on isogenies \cite[\S5]{Falt1}.
\end{proof}

\begin{lem}
Let $M$ be a finite length module over the ring $R$ and $E_1,\dots, E_s$ its non-isomorphic simple constituents. Let $M_i$ be the maximal $R$-submodule all of whose composition factors are isomorphic to $E_i.$   If $\Ext^1_R(E_i,E_j)=0$ for $i\ne j,$  then $M=\bigoplus M_i,$ i.e. $M$ is the sum of its isotypic components. 
\end{lem}

\begin{proof} 
If all composition factors of the $R$-modules $N$ and $N'$ are isomorphic to $E_i$, the same is true of $N+N'$ as a quotient of $N\oplus N'$, so the definition of $M_i$ makes sense. The sum of the $M_i$ is direct, since no simple module occurs in the intersection of $M_j$ with the sum of the other isotypics. By the long exact sequence of $\Ext$ and induction, $\Ext^1_R(E_i,P)=0$ if $P$ does not involve $E_i.$  Let $M'=\bigoplus_{i=1}^s M_i \subsetneq M$ and let $N$ be a minimal submodule of $M$ containing $M'$.   Then, after relabeling, we have $N/M'\simeq E_s$.  The exact sequence
$
0\to M'/M_s\to N/M_s\to E_s\to 0
$
splits, so there is a submodule $N'$ of $N$ with $N'/M_s\simeq E_s,$  contradicting  maximality of $M_s.$
\end{proof}

\begin{Rem}  
In his work on deformations, Ploner \cite{Plon} considered conditions {\bf E1, E2} and {\bf E4} for two-dimensional group schemes.
\end{Rem}

\section{Some Honda systems} \label{Honda}  \numberwithin{equation}{section}
Recall that $\W$ is the ring of Witt vectors over a finite field  $k$ of characteristic $p$ and let $K$ be the quotient field of $\W$.   Suppose that $\cE = A[p]$ is an  absolutely simple finite flat group scheme of order $p^4$ where $A$ is an abelian surface over $K$ with biconnected good reduction.  In this section, we classify the Honda systems of such $\cE$'s and those of extensions of $\cE$ by itself annihilated by $p$.

\begin{prop} \label{simple} 
Let $(\rM,\rL)$ be the Honda system for a group scheme $\cE$ as above.  Then there is a $k$-basis $x_1,x_2,x_3,x_4$ for $\rM$ such that $\rL = \spn\{x_1,x_2\}$,
\begin{equation} \label{VFsimple}
\rV=\left[\begin{smallmatrix}0&0&0&0\\ 1&0&0&0\\0&\lambda&0&0\\ 0&0&0&0\end{smallmatrix}\right] \quad \text{and} \hspace{12 pt} \rF=\left[\begin{smallmatrix}0&0&0&0\\ 0&0&0&0\\0&0&0&1\\ 1&0&0&0\end{smallmatrix}\right]
\end{equation}
for some $\lambda$ in $k^{\times}$.  Furthermore $x_1', \dots, x_4'$ is another such basis if and only if $x_1' = r^{p^2} x_1$ and $\lambda' = r^{1-p^4} \lambda$ with $r$ in $k^\times$.
\end{prop}

\begin{proof} 
Let $\gE = (\rM,\rL)$ be the Honda system for $\cE.$   Refer to Lemma \ref{pM} as needed.  Theorem \ref{FoHon}, applied to the $p$-divisible group of $A$ implies that $\dim L = 2$.  By absolute simplicity, $\cE$ becomes a Raynaud $\F_{p^4}$-module scheme over the Witt vectors $\W(\ov{k})$ \cite{Ray1}, \cite[\S4]{Tat2}.  Berthelot \cite[Lemme 2.5]{Ber} shows that $\rM'=\rM\otimes_{k} \ov{k}$ admits a basis  $\{\xi_i\,|\, i\text{ in }\Z/4\Z\}$ such that $\rF(\xi_i) = \xi_{i+1}$ or $\rV(\xi_{i+1}) = \xi_i,$ with $\rL'$ spanned by a subset of that basis.  

Suppose that $\rL'$ does not contain two successive basis vectors.  Then we may assume that $\rL' = \spn\{\xi_1, \xi_3\}.$   By injectivity of $\rV$ on $\rL$, we have $\rV\xi_1 = \xi_0$ and $\rV\xi_3 = \xi_2$.  Since $\rF(\rM') = \spn\{\rF(\xi_1),\rF(\xi_3)\}$ is 2-dimensional, $\rV(\xi_2) \ne \xi_1$, so $\rF(\xi_1) = \xi_2$ and similarly $\rF(\xi_3) = \xi_0$.  If $\eta = \xi_1 + \xi_3$, then $\rF\eta = \rV\eta = \xi_2+\xi_0$. Thus  there is a sub-Honda system $(\rM'', \rL'')$ of $\gE$ with $\rM'' = \spn\{\eta,\rF\eta\}$ and $\rL'' = \spn\{\eta\}$, contradicting absolute simplicity of $\cE$.

Therefore, we may assume that $\rL' = \spn \{\xi_1, \xi_2\}$.  Since $\rV$ is injective on $\rL'$, we cannot have $\rF(\xi_1) = \xi_2$, so $\xi_1 = \rV\xi_2 \in \rL' \cap \rV\rL'$ and $\dim_k(\rL\cap \rV\rL)=1$ over the original ground field $k$.  Write $x_2 = \rV x_1 \ne 0$ in $\rL \cap \rV\rL$ with $x_1$ in $\rL$ and so $\rL = \spn \{x_1,x_2 \}$.  Set $x_4 = \rF x_1$ and $x_3 = \rF^2 x_1$.  Since $\dim_k \ker \rF = 2$ and $\rF$ is nilpotent, $\rF^3 = 0$.  By iterating $\rF$ on $\rM = \rL + \rF\rM$ to find that $\rF\rM = \rF\rL + \rF^2\rL = \spn\{x_3,x_4\}$.  Thus $x_1, x_2, x_3, x_4$ is a basis for $\rM$.   Injectivity of $\rV$ on $\rL$ implies that $\rV x_2 \ne 0$.  But $\rV x_2$ is in $\ker \rF = \rV\rL = \spn \{x_2, x_3\}$ and $\rV$ is nilpotent.  Hence $\rV x_2 = \lambda x_3$ for some $\lambda \in k^{\times}$, resulting in matrix representations of the form \eqref{VFsimple}.

For another such basis, $x_2'$ generates $\rL \cap \rV\rL$, so $x_2' = r^p x_2$ with $r \in k^\times$.  Then $x_1' = r^{p^2} x_1$ and $x_3' = \rF^2 x_1' = r^{p^4} x_3$.  Thus
$
\lambda' x_3' = \rV x_2' = r \rV x_2 = r \lambda x_3 = r^{1-p^4} \lambda x_3' 
$
and so $\lambda' = r^{1-p^4} \lambda$ in $k^\times$. 
\end{proof}

\begin{Not}   \label{Ebasis}
For $\lambda \in k^\times$, let $\gE_\lambda  = (\rM_0,\rL_0)$ be the Honda system in the Proposition and call $x_1,x_2,x_3,x_4$ a {\em standard basis} for $\gE_\lambda$.   Denote the corresponding group scheme, Galois module and representation by $\cE_\lambda$, $E_\lambda$ and $\rho_{E_{\lambda}}$ respectively.  
\end{Not}

Let $\Ext^1(\gE_\lambda, \gE_\lambda)$ be the group of classes of extensions of Honda systems:
\begin{equation} \label{ExtPSeq}
0 \to \gE_\lambda \xrightarrow{\iota} (\rM,\rL) \xrightarrow{\pi} \gE_\lambda \to 0
\end{equation}
under Baer sum \cite[Ch.III,Thm.2.1]{Mac} and let $\Ext^1_{[p]}(\gE_\lambda, \gE_\lambda)$ be the subgroup such that $p\rM = 0$.

\begin{prop} \label{Hexpp}
If $(\rM,\rL)$ represents a class in $\Ext^1_{[p]}(\gE_\lambda, \gE_\lambda)$, there is a $k$-basis $e_1, \dots, e_8$ for $\rM$ such that $\iota(x_1) = e_1$, $\pi(e_5) = x_1$, $\rL = \spn\{ e_1,e_2,e_5,e_6\}$,   
\vspace{2 pt}
{\small 
$$
\rV=\left[\begin{array}{cccc|cccc}
              0&0&0&0&0&\lambda s_2&0&0 \\
              1&0&0&0&0&\lambda s_3&0&0 \\
              0&\lambda&0&0&0&\lambda s_4 &0&0  \\
              0&0&0&0&s_1&\lambda s_5&0&0 \\
                \hline
              0&0&0&0&0&0&0 &0   \\ 
              0&0&0&0&1&0&0&0        \\ 
              0&0&0&0&0&\lambda&0&0    \\ 
              0&0&0&0&0&0&0&0                                                                                
                                             \end{array}\right] \text{and } \,
\rF= \left[\begin{array}{cccc|cccc}0&0&0&0&0&0&0&0
                                            \\ 0&0&0&0&0&0&0&0
                                             \\ 0&0&0&1&0&-s_1^p&-s_5^p&0
                                             \\ 1&0&0&0&0&0&-s_2^p&0
                                             \\ \hline 0&0&0&0&0&0&0&0
                                             \\ 0&0&0&0&0&0&0&0
                                             \\ 0&0&0&0&0&0&0&1
                                             \\ 0&0&0&0&1&0&0&0   \end{array}\right] 
$$}

\noindent    with $s_1, s_2, s_3, s_4, s_5$ in $k$.   For $\tilde{k} = k/(\sigma^4-1)(k)$, the map $(\rM,\rL) \leadsto (s_1, \dots, s_5)$ induces an isomorphism of additive groups
$
\bfs\!: \, \Ext^1_{[p]}(\gE_\lambda, \gE_\lambda) \xrightarrow{\sim} k \oplus k  \oplus k \oplus  \tilde{k}  \oplus k.
$ 
\end{prop}

\begin{proof} 
Let $\{x_j \, \vert \, 1 \le j \le 4\}$ be a standard basis for $\gE_\lambda$ and define $e_j = \iota(x_j)$ in \eqref{ExtPSeq}.  Since $0 \to \rL_0  \xrightarrow{\iota} \rL  \xrightarrow{\pi} \rL_0 \to 0$ is exact, we can extend $e_1, e_2$ to a basis for $\rL$ by adjoining elements $\tilde{e}_5, \tilde{e}_6$ of $\rL$ such that $\pi(\tilde{e}_5) = x_1$ and $\pi(\tilde{e}_6) = x_2$.   

From $\rV(\pi(\tilde{e}_5)) = \pi(\tilde{e}_6)$, we have $\rV\tilde{e}_5 = \tilde{e}_6 + r_1e_1+r_2e_2+r_3e_3+s_1e_4$ with $s_1$ and all $r_i$ in $k$.  Replace $\tilde{e}_5$ by $e_5 = \tilde{e}_5 + \sigma^2(a_1) e_1 +\sigma(a_2) e_2$ and $\tilde{e}_6$ by $e_6 = \tilde{e}_6+b_1 e_1 + b_2 e_2$ with $a_i, b_i$ in $k$. Then 
\begin{eqnarray*}
\rV e_5 &=& \rV\tilde{e}_5 + \sigma(a_1)e_2 + \lambda a_2 e_3  \\
          &=& \tilde{e}_6 + r_1 e_1+(r_2+\sigma(a_1))e_2+(r_3+\lambda a_2) e_3+s_1e_4 \\
          &=& e_6 + (r_1-b_1) e_1+(r_2+\sigma(a_1)-b_2)e_2+(r_3+\lambda a_2) e_3+s_1e_4.
\end{eqnarray*}
Now choose $a_i,b_i$ so that $\rV(e_5) - e_6 = s_1 e_4$.  Finally, let $e_8 = \rF e_5$ and $e_7 = \rF e_8$. Since $\rV(\pi(e_6)) = \lambda \pi(e_7)$, there we may choose elements $s_i$ of $k$ such that 
\begin{equation} \label{Ve6}
\rV e_6 = \lambda (e_7 + s_2e_1+s_3e_2+s_4e_3+s_5e_4).
\end{equation}
This verifies the matrix representation of $\rV$. From $0 = \rF\rV e_5 = \rF e_6 + \sigma(s_1) e_3$, we get $\rF e_6 = -\sigma(s_1)e_3$.  Apply $\rF$ to \eqref{Ve6} to find $\rF e_7$ and obtain the matrix of $\rF$.  

The only ambiguity left is that $e_5$ might be replaced by $e_5 + \sigma^2(a_1)e_1$, in which case $s_4$ becomes $s_4+a_1-\sigma^4(a_1)$ while $s_1, s_2, s_3, s_5$ remain unchanged.

Another extension $(\rM',\rL')$ is equivalent to $(\rM,\rL)$ if and only if there is an isomorphism $h$ in the commutative diagram
 \begin{equation} \label{IsomExtP}
\begin{CD} 
0 @>>> \gE_\lambda @>\iota'>> (\rM',\rL') @>\pi'>> \gE_\lambda @>>> 0  \\
  &  & @VV{\rm ident} V        @VV h V         @VV{\rm ident} V \\
0 @>>> \gE_\lambda @>\iota>> (\rM,\rL) @>\pi>> \gE_\lambda @>>> 0 \, . 
\end{CD}  \vspace{2 pt}
\end{equation}
Let $e_1', \dots e'_8$ be a basis for $(\rM',\rL')$ constructed as above.  Since $h(e_1'), \dots, h(e'_8)$ must be another such basis, the isomorphism $h$ exists if and only if $h(e_1') = e_1$ and $h(e_5') = e_5 + \sigma^2(a_1)e_1$ with $a_1$ in $k$.  It follows that $\bfs$ is a well-defined bijection.

To verify the additivity of $\bfs$, let $(\rM,\rL)$ and $(\rM',\rL')$ represent two classes in $\Ext^1_{[p]}(\gE_\lambda, \gE_\lambda)$ and let $0 \to \gE_\lambda \xrightarrow{\iota''} (\rM'',\rL'') \xrightarrow{\pi''} \gE_\lambda \to 0$ represent their Baer sum. To obtain a $k$-basis for $\rM''$ let $\gamma_i = (e_i, 0)$ in $\rM \times \rM'$ for $1 \le i \le 4$ and $\gamma_i = (e_i,e_i')$ for $5 \le i \le 8$, each of which satisfies the fiber product condition that $\pi''(\gamma_i) = \pi(e_i) = \pi'(e_i')$.  The relations are given by $\iota''(a) = (\iota(a),0) = (0,\iota'(a))$ for all $a$ in $\gE_\lambda$.  We have
\begin{eqnarray*}
\rV\gamma_5&=& (\rV e_5,\rV e_5') = (e_6 + s_1 e_4,e_6' + s_1' e_4') = \gamma_6 + (s_1e_4,0)+(0,s_1'e_4')  \\
                   &=& \gamma_6  + (s_1e_4,0) + (s_1'e_4,0) = \gamma_6 + (s_1+s_1')\gamma_4,\\
\rV\gamma_6 &= &(\rV e_6, \rV e_6') = \lambda (e_7,e_7') +  \sum_{1 \le i \le 4} \lambda(s_i e_i, s_i'e_i') = \lambda \gamma_7 + \sum_{1 \le i \le 4} \lambda(s_i+s_i') \gamma_i, \\
\rF\gamma_6 &=& (\rF e_6, \rF e_6') = -(s_1^p e_3, (s_1')^p e_3') = -(s_1^p e_3,0) - (0,(s_1')^p e_3') \\ 
                    &=& -(s_1^p e_3,0)-((s_1')^p e_3,0) = -(s_1+s_1')^p \gamma_3, \\
\rF\gamma_7 &=& (\rF e_7, \rF e_7') = -(s_5^p e_3, (s_5')^p e_3') -(s_2^p e_4,(s_2')^p e_4')  \\
&=& -(s_5 + s_5')^p\gamma_3 - (s_2+s_2')^p \gamma_4.
\end{eqnarray*} 
By completing the matrices for $\rV$ and $\rF$, we find that $s_i'' = s_i + s_i'$ for $1 \le i \le 5$. 
\end{proof}

\section{The local theory} \label{FP}

\numberwithin{equation}{subsection}
In this section, we study the fields of points of extensions of exponent $p$ whose Honda systems were described above.  In particular, we obtain good conductor bounds. We use freely the notation of \S\ref{GpScheme}.  Let $K$ be the quotient field of $\W$ and let $\dot{a}$ be the Teichm\"uller lift to $\W$ of $a$ in $k$, with $\dot{0} = 0$.   Assume that $w$ is in $\cO_{\ov{K}}$ and $\ord_p(w) > 0$.  For $a$ in $\ov{K}/w \cO_{\ov{K}}$, let $\tilde{a}$ be an arbitrary lift to $\ov{K}$.  Assertions requiring lifts are made only when the result is independent of the choices, as in the following examples.  If $a$ is  not in $w \cO_{\ov{K}}$, let $\ord_p(a) = \ord_p(\tilde{a})$.  For $w'$ in $\cO_{\ov{K}}$ such that $0 < \ord_p(w') \le \ord_p(w)$, let $a \equiv b \pmod{w'}$ mean that $\tilde{a} - \tilde{b}$ is in $w' \, \cO_{\ov{K}}$.   If $f(X)$ is in $\ov{K}[X]$, we write $f(a) \equiv 0 \pmod{w'' \, \cO_{\ov{K}}}$ only if $f(\tilde{a})$ is in $w'' \, \cO_{\ov{K}}$, for all lifts $\tilde{a}$ of $a$.  For this section, we write $x \sim y$ when  $\ord_p(\frac{x}{y}-1) > 0$ and $x=y + O(w)$ if $\ord_p(x-y) \ge \ord_p(w).$

\subsection{The irreducible case} 

Let $\cE_\lambda$ be the group scheme and $x_1, \dots, x_4$ a standard basis for the corresponding  Honda system $\gE_\lambda = (\rM_0,\rL_0)$ from Notation \ref{Ebasis}.  The Galois module structure of $E_\lambda$ is well-known, but a description of $E_\lambda$ by Witt covectors is required for our analysis of extensions of $\cE_\lambda$ by $\cE_\lambda$.   Let $F = K(E_\lambda)$, reserving Roman F and V for the Honda system Frobenius and Verschiebung operators in this section.  Recall that points of the Galois module $E_\lambda$ correspond to $D_k$-homomorphisms $\psi\!: \, \rM_0 \to  \widehat{CW}_k(\cO_{\ov{K}}/p\cO_{\ov{K}})$ such that $\xi(\psi(\rL_0)) = 0$, cf.\! \S\ref{GpScheme}.
 
\begin{prop}  \label{FieldForE} 
Let $\gR_\lambda = \{ a \in \cO_{\ov{K}}/p \cO_{\ov{K}} \hspace{2 pt} \vert \hspace{2 pt} \lambda^{p^2} a^{p^4} \equiv  (-p)^{p+1}a \pmod{p^{p+2} \cO_{\ov{K}}} \}$.  Given $a$ in $\gR_\lambda$, define $b$ and $c$ in  $\cO_{\ov{K}}/p \cO_{\ov{K}}$ by  
$$
 b \equiv - \textstyle{\frac{1}{p}}\lambda^p a^{p^3} \pmod{p\cO_{\ov{K}}} \qquad \text{and} \qquad  c \equiv \lambda a^{p^2}   \pmod{p\cO_{\ov{K}}}.  
$$
\begin{enumerate}[{\rm i)}]
\item A $D_k$-map $\psi = \psi_a$ belongs to a point $P_a$ of $E_\lambda$ if and only if $\psi(x_1) =  (\0,c,b,a)$ with $a$ in $\gR_\lambda$.  If so, $\psi(x_2)=(\0,c,b)$, $\psi(x_3)=(\0,\lambda^{-1} c)$ and $\psi(x_4)=(\0,a^p)$.   \vspace{2 pt}

\item $F = K(E_\lambda)$ is the splitting field of $f_\lambda(x) = \dot{\lambda}^{p^2}x^{{p^4}-1} - (-p)^{p+1}$ over $K.$  The maximal subfield of $F$ unramified over $K$ is $F_0 = K(\Mu_{{p^4}-1}, \eta)$, where $\eta$ is any root of $x^{p+1} - \dot{\lambda}$.  Moreover $F/F_0$ is tamely ramified of degree $t = (p^2+1)(p-1)$.  For $a \ne 0$ we have
\begin{equation} \label{ordabc}
\textstyle{\ord_p(a) = \frac{1}{t}, \quad \ord_p(b) = \frac{p^2-p+1}{t}, \quad \ord_p(c) = \frac{p^2}{t}.}
\end{equation}

\item   $\gR_\lambda$ is an $\F_{p^4}$-vector space under the usual operations in $\cO_{\ov{K}}/p \cO_{\ov{K}}$ and $a \mapsto P_a$ defines an $\F_p[G_K]$-isomorphism $\gR_\lambda \xrightarrow{\sim} E_\lambda$.     
\end{enumerate}
\end{prop}

\begin{proof}   
i)  If $\psi$ belongs to a point in $E_\lambda$, then $\psi(x_1) = (\0,c,b,a)$, since $V^3 = 0$.  We obtain $\psi(x_2)$ and $\psi(x_3)$ by applying $V$, while $\psi(x_4)= \psi({\rm F} x_1) = (\0,c^p,b^p,a^p)$.  Use $0 = {\rm VF}(x_1) = {\rm V}x_4$ to find that $c^p = b^p = 0$, so $\ord_p(b), \ord_p(c) \ge 1/p$. In addition, ${\rm F}(x_4)=x_3$ implies that $c = \lambda a^{p^2}.$   Let $\tilde{a}, \tilde{b}, \tilde{c}$ denote lifts to $\cO_{\ov{K}}$.  Vanishing of $\xi(\psi(L))$ provides the additional congruences modulo $p \cO_{\ov{K}}$:
\begin{equation} \label{e1e2}
  \tilde{a}+ \frac{1}{p} \, \tilde{b}^p+ \frac{1}{p^2} \, \tilde{c}^{p^2} \equiv 0  \quad \text{ and } \quad   \tilde{b} + \frac{1}{p} \, \tilde{c}^p \equiv 0.
\end{equation}
Thus $p \ord_p(\tilde{c}) = \ord_p(p\tilde{b}) \ge 1+\frac{1}{p}$ and so $\frac{1}{p^2} \, \tilde{c}^{p^2} \equiv 0$.  With this simplification, the required congruences follow from (\ref{e1e2}).   Furthermore, these congruences are sufficient to imply that $\psi$ belongs to $\cE_\lambda$ when $\psi(x_1) = (\0,c,b,a)$.    \vspace{2 pt}

 ii) If $f_\lambda(\theta) = 0$ and $\zeta$ generates $\Mu_{p^4-1}$, then the roots of $f_\lambda$ have the form $\theta_j = \zeta^j \theta$ while their reductions modulo $p$ give all non-zero elements of $\gR_\lambda$.  For the converse, let $\tilde{a}$ be a lift of $a \in \gR_\lambda$ and $g(x) = x^{p^4}-x$.  Then $g(\tilde{a}/\theta) \equiv 0 \pmod{\frac{p}{\theta} \, \cO_{\ov{K}}}$ and so $\tilde{a} \equiv 0$ or $\tilde{a} \equiv \theta_j \pmod{p \cO_{\ov{K}}}$ for some $j$ by Hensel's Lemma.    Hence $F = K(\Mu_{{p^4}-1},\theta)$ is the splitting field of $f_\lambda$.   Let $F_0$ be the maximal subfield of $F$ unramified over $K$.  Since $\dot{\lambda}^{p^2}$ and therefore also $\dot{\lambda}$ is a $(p+1)$ power in $K(\theta)$, each root $\eta$ of $x^{p+1} - \dot{\lambda}$ is in $F_0$.  Furthermore, $\theta$ satisfies an Eisenstein polynomial  of the form $\eta^{p^2}x^t + \omega p = 0$ over $F_0$ for some $\omega$ in $\Mu_{p+1}$.  Hence $K/F_0$ is tamely ramified of degree $t$ and we obtain the desired ordinals of $a,b,c$.
 
\vspace{2 pt} 

iii) The embedding $\F_{p^4} = \W(\F_{p^4})/p \hspace{.5 pt} \W(\F_{p^4}) \hookrightarrow \cO_{\ov{K}}/p \cO_{\ov{K}}$ defines the scalar multiplication by $\F_{p^4}$.  Closure of $\gR_\lambda$ under this operation and under the usual addition in $ \cO_{\ov{K}}/p \cO_{\ov{K}}$  is clear.  The asserted Galois isomorphism follows from the correspondence between $D_k$-homomorphisms belonging to $\cE_\lambda$ and points of $E_\lambda$ once we check that $a \mapsto P_a$ is additive.  If $a_1$ and $a_2$ are in $\gR_\lambda$, then there is some $a$ in $\gR_\lambda$ such that $\psi_{a_1}(x_1) \CWplus \psi_{a_2}(x_1) = \psi_{a}(x_1)$.  Denote this equation of Witt covectors by
$
(\0,c_1,b_1,a_1) \CWplus (\0,c_2,b_2,a_2) = (\0,c,b,a).
$
Then $c = c_1+c_2$, so $a^{p^2} = a_1^{p^2} + a_2^{p^2}$ in $\cO_{\ov{K}}/p\cO_{\ov{K}}$.  By using lifts of $a$, $a_1$ and $a_2$ of the form $\omega_0 \theta$, $\omega_1 \theta$ and $\omega_2 \theta$, with each $\omega_j$ in $\Mu_{p^4-1} \cup \{0\}$, we find that $$\omega_0^{p^2} \equiv \omega_1^{p^2} + \omega_2^{p^2} \equiv (\omega_1+\omega_2)^{p^2} \pmod{\frac{p}{\theta^{p^2}} \cO_{\ov{K}}}.$$  Since the $\omega$'s lie in the absolutely unramified field $\Q_p(\Mu_{p^4-1})$ and $\ord_p(p/\theta^{p^2}) > 0$, we obtain $\omega_0 \equiv \omega_1 + \omega_2 \pmod{p}$ and thus $a = a_1+a_2$ in $\cO_{\ov{K}}/p\cO_{\ov{K}}$.  Alternatively, $\ord_p(a-a_1-a_2) \ge 1$ by the covector addition formulas in Lemma \ref{Wittadd}.
\end{proof}

\begin{Rem} \label{Pa}  \hfill
\begin{enumerate}[i)]
\item By (ii) above, the lifts of all $a \ne 0$ in $\gR_\lambda$ to $\cO_{\ov{K}}$ comprise the cosets $\zeta^j \theta + p\cO_{\ov{K}}$.  Thus $\gR_\lambda$ descends to an $\F_{p^4}$-vector subspace  of $\cO_F/p \cO_F$ and we write 
$$
\gR_\lambda(F) = \{ a \in \cO_F/p \cO_F \hspace{3 pt} \vert \hspace{3 pt} \dot{\lambda}^{p^2} a^{p^4} \equiv  (-p)^{p+1}a\pmod{p^{p+2} \cO_F} \}.
$$
For $\alpha$ in $\F_{p^4}$ and $a$ in $\gR_\lambda$, we write $\alpha P_a = P_{\alpha a}$, in agreement with multiplication on Witt covectors.  In fact, $\alpha \psi_a  = \psi_{\alpha a}$, since evaluating on $x_1$ gives
$$
[\alpha] (\0,c_a,b_a, a) = (0, \alpha^{\frac{1}{p^2}} c_a,  \alpha^{\frac{1}{p}} b_a, \alpha a)= (\0,c_{\alpha a},b_{\alpha a}, \alpha a).
$$   

\item If $h$ is in the ramification subgroup of $\Gal(F/K)$, then $h$ acts on $\gR_\lambda(F)$ by $h(a) = \alpha a$, where $\alpha \in \Mu_t$ depends on $h$.  The structure of $\cE_\lambda$ as a Raynaud $\F_{p^4}$-module scheme is reflected by $h(P_a) = P_{h(a)} = P_{\alpha a} = \alpha P_a$.  However, Frobenius in $\Gal(F/K)$ acts on the scalars.   \vspace{3 pt}

\item By Proposition \ref{FieldForE}, we have $b^p \equiv -pa \, (\text{mod } p^2)$,  $c^p  \equiv \lambda^p a^{p^3}  \equiv -pb \, (\text{mod } p^2)$ and $c^{p^2}  \equiv (-p)^{p+1}a \, (\text{mod } p^{p+2})$.  These congruences are independent of the choices of lifts to $\cO_{\ov{K}}$.
\end{enumerate}
\end{Rem}

Using the local structure above, we next obtain a group scheme $\cE$ over $\Z[\frac{1}{N}]$ fulfilling the hypotheses of Definition \ref{ourcat} for a $\Sigma$-category $\un{E}$ with $\Sigma = \{\cE\}$. We also determine the image of the Galois representation provided by $E$.  

\begin{cor} \label{EScheme} 
Let $E$ be a four-dimensional symplectic module over $\F_p$ and let $\rho\!: \, G_\Q \to \GS(E)$ be unramified outside $\{p,N,\infty\}$ and tamely ramified at the prime $N\ne p$.  Suppose that: 
\begin{enumerate}[{\rm \hspace{3 pt} i)}]
\item $\rho$ restricted to a decomposition group at $p$ is isomorphic to a local representation of the form $\rho_{E_\lambda}$ as in {\rm Notation \ref{Ebasis}};  \vspace{2 pt}
\item     inertia at $v \vert N$ acts on $E$ via a cyclic quotient $\langle \sigma_v \rangle$ with $(\sigma_v-1)^2 = 0$ and $\rk(\sigma_v - 1) = 1$ as a matrix;

 \vspace{2 pt}

\item the fixed field of $\rho^{-1}(\SP(E))$ is $\Q(\mu_p)$ when $p$ is odd.
\end{enumerate} 
Then there is a unique finite flat group scheme $\cE$ over $\Z[\frac{1}{N}]$ whose associated Galois representation is $\rho$.  Moreover, the Galois image $G=\rho(G_\Q)$ is $\GS_4(\F_p)$ for $p\ge 2$ or possibly  ${\rm O}^-_4(\F_2) \simeq \cS_5$ when $p=2$.
\end{cor}

\begin{proof} 
By (i), the local representation is irreducible and so is $E.$ We  patch  as described before \eqref{MV} to get the uniqueness. 

Since $\sigma_v$ is a transvection by (ii),  the normal subgroup $P$  generated by transvections is non-trivial. Follow the proof of \cite[Proposition 2.8]{BK3}, using $\dim E=4$ and the fact that $N$ is square-free, to conclude that $E$  is irreducible for the group $P$ generated by transvections. If $p=2$, we find that $G$ is isomorphic to $\SP_4(\F_2) \simeq \cS_6$ or ${\rm O}^-_4(\F_2) \simeq \cS_5$. Since 5 must divide $|G|,$ we rule out $\cS_3\wr \cS_2.$     When $p$ is odd, $G$ contains $\SP_4(\F_p)$ by \cite{KM} and thus is isomorphic to  $\GS_4(\F_p)$ by (iii).  
\end{proof}

When $p=2$ and $A$ is a favorable abelian surface, $\cE = A[2]$ provides a representation $\rho$ as in the Corollary.

\subsection{Extensions of exponent $p$} \label{ptsEEp}
Let $0 \to \cE_\lambda \xrightarrow{\iota} \cW \xrightarrow{\pi} \cE_\lambda \to 0$ be an extension of $\cE_\lambda$ by $\cE_\lambda$ killed by $p$ with parameters $\bfs(\cW)=[s_1 \cdots s_5]$ from Proposition \ref{Hexpp}.  Let $P_a$ denote the point of $E_\lambda$ corresponding to $a$ in $\gR_\lambda(F)$, cf. Proposition \ref{FieldForE}(iii) and Remark \ref{Pa}.  Then the fiber over $P_a$ has the form $Q+\iota(E_\lambda)$ for any fixed $Q$ in $W$ such that $\pi(Q) = P_a$.  We write $F_a = F(Q)$ for the {\em fiber field} generated over $F$ by the coordinates of $Q$. 

\begin{Not}
Write $R_u = \ov{K}/\frac{p}{u} \, \cO_{\ov{K}}$, provided that $u$ is in $\cO_{\ov{K}}$ and $\ord_p(u) < 1$. 
\end{Not}

\begin{prop}  \label{EqsExpP}
For $\varphi$ to correspond to a point of $W$ in the fiber over $P_a \ne 0$, it is necessary and sufficient that $\varphi(e_1) = (\0,c,b,a)$ as in {\rm Proposition \ref{FieldForE}} and 
$$
\varphi(e_5) = (\0, (\lambda s_2)^\frac{1}{p^2}c, \, (\lambda s_2)^\frac{1}{p}b+(\lambda s_3)^\frac{1}{p}c, \, cz, \, by, \, ax)
$$ 
where $x,y,z$ in $\ov{K}$ satisfy all the following congruences:
\begin{equation}
\begin{array}{r c  r  c  l}
{\rm i)} & \quad &  x - y^p + p^{p-1} z^{p^2} &=& 0 \, \text{ in } \, R_a, \\  
 \label{Newxs} {\rm ii)} &  & y-z^p +  p  \lambda^{-p} \epsilon_p a^{p-p^3} &=& 0  \, \text{ in } \,  R_b,  \\
{\rm iii)} & & x^{p^2} - z +  w a^{-p^2}&=& 0  \, \text{ in } \, R_c,    
\end{array}
\end{equation}
with $w=s_2a+s_3b+s_4\lambda^{-1}c+s_5a^p$,  $\epsilon_p = s_1$ if $p \ge 3$ and $\epsilon_2 = s_1-(\lambda s_2)^2$  if $p =2$.  Equivalently, $z$ in $\ov{K}$ satisfies $f_a(z) = 0$ in $R_c$, where  
\begin{equation} \label{firstfa}
f_a(Z) = \left[\left(Z^p - p  \lambda^{-p} \epsilon_p a^{p-p^3}\right)^p - p^{p-1} Z^{p^2}\right]^{p^2}-Z+  w a^{-p^2}
\end{equation} 
and the classes of $x$ in $R_a$ and $y$ in $R_b$ are determined by {\rm \eqref{Newxs}(i)} and {\rm (ii)}.  When $\epsilon_p = 0$, we may instead use $f_a(Z) =  Z^{p^4} - Z+  w a^{-p^2}$. 
\end{prop}

\begin{proof}
Let $\varphi$ in 
$
\Hom_{D_k}({\rm M},\widehat{CW}_k(\cO_{\ov{K}}/p\cO_{\ov{K}}))
$
be an element of $\cW.$  Since $M$ is generated by $e_1$ and $e_5$ as a $D_k$-module, $\varphi$ is determined by $\varphi(e_1)$ and $\varphi(e_5).$  The injection of $\gE_\lambda$ to $M$ yields $\varphi(e_j) = \psi(x_j)$ for $1 \le j \le 4$, as in Proposition \ref{FieldForE}(i).  Set $\varphi(e_5)=(\0,d_4,d_3,d_2,d_1,d_0)$, with only the five rightmost coordinates significant,  since ${\rm V}^5 = 0$.   Applying ${\rm FV} = 0$ to $e_5$ gives $d_4^p=d_3^p=d_2^p=d_1^p = 0$.

From the matrix representation of ${\rm V}$, we have
$$
\varphi(e_6) = {\rm V}(\varphi(e_5)) \CWplus [-s_1]\varphi(e_4) = (\0,d_4,d_3,d_2,d_1-s_1a^p)
$$
and so $\varphi(\lambda^{-1}{\rm V}e_6) = [\lambda^{-1}](\0,d_4,d_3,d_2)$.  We also have 
$$
\begin{array}{l}
\varphi(\lambda^{-1}{\rm V}e_6) = \varphi(e_7) \CWplus\varphi(s_2e_1) \CWplus\varphi(s_3e_2) \CWplus \varphi(s_4e_3) \CWplus \varphi(s_5e_4) \vspace{5 pt} \\
= {\rm F}^2\varphi(e_5) \CWplus (\0,\sigma^{-2}(s_2)c,\sigma^{-1}(s_2)b,s_2a) \CWplus   (\0, \sigma^{-1}(s_3)c, s_3b+s_4\lambda^{-1}c+s_5a^p) \vspace{3 pt} \\
= (\0,d_0^{p^2}) \CWplus (\0,s_2^\frac{1}{p^2}c,s_2^\frac{1}{p}b+s_3^\frac{1}{p}c,s_2a+s_3b+s_4\lambda^{-1}c+s_5a^p-\Phi_p(s_2^\frac{1}{p}b,s_3^\frac{1}{p}c)) \vspace{3 pt}\\
= (\0,s_2^\frac{1}{p^2}c,s_2^\frac{1}{p}b+s_3^\frac{1}{p}c,s_2a+s_3b+s_4\lambda^{-1}c+s_5a^p+d_0^{p^2}).
\end{array}
$$
since $\Phi_p(s_2^\frac{1}{p}b,s_3^\frac{1}{p}c) = 0$ by \eqref{ordabc} and Lemma \ref{ppower}.  Modulo $p\cO_{\ov{K}}$, this gives: 
\begin{equation}  \label{d2d3d4} 
\begin{array}{lll}
d_4 \equiv (\lambda s_2)^\frac{1}{p^2}c, \quad &d_3 \equiv (\lambda s_2)^\frac{1}{p}b+(\lambda s_3)^\frac{1}{p}c, \quad &d_2 \equiv \lambda( d_0^{p^2} + w).
\end{array}
\end{equation} 
 
Vanishing of the Hasse-Witt map on $\varphi(\rL)$ gives the following additional relations:
\begin{equation} \label{deqs}
\begin{array}{llll}
 \xi(\varphi(e_5)) &=&\frac{d_4^{p^4}}{p^4}+\frac{d_3^{p^3}}{p^3}+\frac{d_2^{p^2}}{p^2}+\frac{d_1^{p}}{p}+d_0 \equiv 0 &\pmod{p \cO_{\ov{K}}},  \vspace{3 pt}\\
 \xi(\varphi(e_6))&=&\frac{d_4^{p^3}}{p^3}+\frac{d_3^{p^2}}{p^2}+\frac{d_2^{p}}{p}+d_1-s_1a^p \equiv 0 &\pmod{p\cO_{\ov{K}}}. 
\end{array}
\end{equation}
Since 
$
p^2 \ord_p(d_4) \ge p^2 \ord_p(c) > p+1,
$ 
we have $p^{-3}d_4^{p^3} \equiv 0$ and $p^{-4}d_4^{p^4} \equiv 0 \pmod{p}$.  Thus the $d_4$-terms drop out of (\ref{deqs}).  By \eqref{d2d3d4}, we have
$$
\begin{array}{ll}
d_3^p \equiv \lambda s_2 b^p + \lambda s_3 c^p  \hspace{3 pt} (\text{mod }p^2),  &\hspace{10 pt} d_3^{p^2} \equiv (\lambda s_2)^p b^{p^2} + (\lambda s_3)^p c^{p^2} \hspace{3 pt} (\text{mod }p^3),  \vspace{3 pt}  \\
\multicolumn{2}{c}{d_3^{p^3} \equiv (\lambda s_2 )^{p^2}b^{p^3} +  (\lambda s_3 )^{p^2}c^{p^3} \hspace{3 pt} (\text{mod }p^4).} 
\end{array}
$$
In addition, 
{\small $$
\ord_p\left(\frac{c^{p^j}}{p^j}\right) > \ord_p\left(\frac{b^{p^j}}{p^j}\right) = \frac{p^j(p^2-p+1)}{(p-1)(p^2+1)} - j = p^{j-1}\left(1+  \frac{1}{(p-1)(p^2+1)}\right) - j  
$$}

\noindent is greater than 1 if:  (i) $j=3$ and all $p$ \, or \, (ii) $j=2$ and $p \ge 3$.  If $j=2$ and $p=2$, we also have $\ord_2(c^4/4) > 1$ and so (\ref{deqs}) simplifies to 
\begin{equation} \label{d0d1d2}
p^{-2}d_2^{p^2} +  p^{-1}d_1^p + d_0 \equiv 0\quad \text{and} \quad 
p^{-1}d_2^p + d_1 - \epsilon_p a^p \equiv 0.
\end{equation}

Let $x = d_0/a$ in $R_a$, $y = d_1/b$ in $R_b$ and $z = d_2/c$ in $R_c$.  Then \eqref{d2d3d4} and \eqref{d0d1d2} give \eqref{Newxs}, using the equations for $a,b,c$ in Remark \ref{Pa}(iii).  It follows that $f_a(z) = 0$ in $R_c$ for $f_a$ given by \eqref{firstfa}.  When $\epsilon_p = 0$, we have 
$$
\ord_p(z) = \frac{1}{p^4}\ord_p(w a^{-p^2}) \ge -\frac{(p^2-1)}{p^4}\ord_p(a) = -\frac{p+1}{p^4(p^2+1)}
$$   
and thus
$$
\ord_p \left(\bino{p^2}{j} p^{(p-1)j} z^{p^4} \right) = (p-1)j + 2 - \ord_p(j) + p^4 \ord_p(z) \ge 1,
$$
i.e., the middle terms  of the binomial expansion for $f_a(z)$ drop out.   

Conversely, if $f_a(z) = 0$ in $R_c$ and $x$ and $y$ are defined by \eqref{Newxs}(i) and (ii), then \eqref{Newxs}(iii) holds and we obtain a $D_k$-homomorphism belonging to a point of $W$ in the fiber over $P$.
\end{proof}

\begin{Not}  \label{zNot}
If $\lambda$ in $k^\times$ is fixed, then $a$ in 
$$
\gR_\lambda(F) = \{ a \in \cO_F/p \cO_F \hspace{3 pt} \vert \hspace{3 pt} \dot{\lambda}^{p^2} a^{p^4} \equiv  (-p)^{p+1}a\pmod{p^{p+2} \cO_F} \}
$$
determines $b$ and $c$ in $\cO_F/p \cO_F$ by the congruences in Proposition \ref{FieldForE}.   If $z$ in $R_c$ satisfies the resulting congruence $f_a(z) = 0$ in $R_c$, then $z$ determines $x$ in $R_a$ and $y$ in $R_b$ by \eqref{Newxs}.  Using the congruences in \eqref{d2d3d4}, set $\bfd_a(z) = (\0,d_4,d_3, c z, b y, ax)$.  Let $\varphi_z$ be the $D_k$-homomorphism such that $\varphi_z(e_1) = (\0,c,b,a)$ and $\varphi_z(e_5) = \bfd_a(z)$ and let $Q_z$ be the corresponding point in $W$.  The {\em fiber field} generated by the point of $W$ lying over the point $P_a$ of $E$ is $F_a = F(Q_z)$. 
\end{Not}

We next examine the effect of various choices of lifts on constructing a generator for the extension $F_a/F$.  Under the assumptions of Notation \ref{zNot}, choose lifts to $\cO_K$ of $\lambda$ and the entries in $\bfs$.   By Remark \ref{Pa}(i), $a$ has a lift $\tilde{a}$ in $\cO_F$.  Using the congruences in Proposition \ref{FieldForE} as equations, $\tilde{a}$ determines lifts $\tilde{b}$ and $\tilde{c}$ in $\cO_F$ of $b$ and $c$.  Let $\tilde{f}$ be the polynomial with coefficients in $F$ obtained by using the respective lifts to replace the corresponding coefficients of $f_a(Z)$.   

\begin{cor}  \label{theta}
Construct $\tilde{f}(Z)$ in $F[Z]$ by choosing the lifts described above.   If $\theta$ is any root of $\tilde{f}$ in $\ov{K}$, then $F_a = F(\theta)$.  If $\epsilon_p = 0$ then $h(X) = X^p - X + \tilde{w} \tilde{a}^{-p^2}$ splits completely in $F_a$.
\end{cor}

\begin{proof}
Let $M$ be the splitting field of $\tilde{f}$ over $F$.  Since the $p^4$ solutions to the congruence $f_a(Z) = 0$ in $R_c$ correspond to the distinct points of $W$ in the fiber over $P_a$, the roots of $\tilde{f}$ in $\ov{K}$ remain distinct when reduced to $R_c$.   If $\theta$ is any root of $\tilde{f}$, its reduction $z$ in $R_c$ determines the point $Q_z$.  Thus $F_a$ is contained in $M$.   If $\gamma$ is in $\Gal(M/F_a)$, then $Q_z = \gamma(Q_z) = Q_{\gamma(z)}$, so $\gamma(z) = z$ in $R_c$.  But then  $\gamma(\theta) = \theta$, since the roots of $\tilde{f}$ are distinct modulo $\frac{p}{c} \, \cO_{\ov{K}}$.   Hence $F_a = F(\theta) = M$ is independent of the various choices of lifts.

When $\epsilon_p = 0$, we have 
$
p^4 \ord_p(\theta) = \ord_p(w/a^{p^2}) \ge (1-p^2) \ord_p(a)$.  We find that
$\alpha = \theta^{p^3} + \theta^{p^2} + \theta^p+\theta$ satisfies
$$
\alpha^p - \alpha \equiv \theta^{p^4} - \theta \equiv - w a^{-p^2}\pmod{ \frac{p}{a^{p^2}} \cO_L},
$$ 
since the worst case middle term in the binomial expansion of $\alpha^p$ leads to
$$
\ord_p \left( p \theta^{p^3(p-1)} \theta^{p^2} \right) = 1 + (p^4-p^3+p^2) \ord_p(\theta) \ge 1 - p^2 \ord_p(a).
$$
Hence $h(\alpha) \equiv 0 \pmod{ pa^{-p^2} \cO_L}.$  Upon clearing denominators, Hensel's Lemma \cite[\!II,\S2]{SL} implies that $h$ has a root in $F_a$ and  the other roots come by refining $\alpha + j$ with $1 \le j \le p-1$. 
\end{proof}

A polynomial $g_a$ of degree $p^4$, analogous to $f_a$, but such that $y$ in $R_b$ satisfies $g_a(y) = 0$ in $R_b$, can also be derived from Proposition \ref{EqsExpP} as in the Corollary below.  Then $y$ determines $x$ in $R_a$ and $z$ in $R_b$ and thus $Q_z$.  Choosing appropriate lifts leads to $\tilde{g}(Y)$ in $F[Y]$, such that a root of $\tilde{g}$ also generates the extension $F_a/F$.  Similar considerations apply to $x$.

\begin{cor} \label{vartheta}
Let $\bfs = [s_10000]$ and choose lifts $\tilde{\lambda}$, $\tilde{s}_1$ in $\cO_K$ and $\tilde{a}$ in $\cO_F.$   Then $F_a = F(\vartheta)$ for any root $\vartheta$ in $\ov{K}$ of $\tilde{g}(Y) = Y^{p^4} - Y - p \tilde{\lambda}^{-p} \tilde{s}_1 \tilde{a}^{p-p^3}.$  \, In addition, $h(X) = X^p - X - p \tilde{\lambda}^{-p} \tilde{s}_1 \tilde{a}^{p-p^3}$ splits completely in $F_a$.   \end{cor}

\begin{proof}
By assumption, $w = 0$ and $\epsilon_p = s_1$.  It suffices to treat $s_1 \ne  0$.  In the proof of Proposition \ref{EqsExpP}, we showed that $d_1^p = 0$ in $\cO_{\ov{K}}/p\cO_{\ov{K}}$.  Hence
$$
{ \ord_p(y) = \ord_p\left(\frac{d_1}{b}\right) \ge \frac{1}{p} - (p^2-p+1) \ord_p(a) = -\frac{1}{p} \ord_p(a)}.
$$
Then $\ord_p(y) >  \ord_p(pa^{p-p^3})$ and so $\ord_p(z^p) = \ord_p(pa^{p-p^3}) = \frac{1-p}{p^2+1}$ by \eqref{Newxs}(ii).  It follows that 
\begin{equation} \label{ErrorTerm}
\ord_p(p^{p-1} z^{p^2}) = p-2+\frac{p+1}{p^2+1}.
\end{equation}
Since \eqref{ErrorTerm} is positive, \eqref{Newxs}(i) and \eqref{Newxs}(iii) imply that 
\begin{equation} \label{yval}
\ord_p(y) = \frac{1}{p} \ord_p(x) = \frac{1}{p^3} \ord_p(z) = -\frac{1}{p^4}\left( \frac{p-1}{p^2+1} \right).  
\end{equation}
By \eqref{ErrorTerm}, if $p \ge 3$, the term $p^{p-1} z^{p^2}$ drops out of \eqref{Newxs}(i) and then we deduce from \eqref{Newxs} that $g_a(y) = y^{p^4} - y + p \lambda^{-p} s_1 a^{p-p^3}$ is 0 in $R_b$.  If $p=2$, apply Lemma \ref{ppower} to \eqref{Newxs}(ii) to obtain $x^4 = y^8$ in $R_c$.   Thus $z = y^{16}$ in $R_c$ and it again follows that $g_a(y) = 0$ in $R_b$.  Conversely, from $y$ satisfying $g_a(y) = 0$ in $R_b$, we can find $x$ and $z$ such that \eqref{Newxs} holds.  The concluding arguments are analogous to those in the proof of Corollary \ref{theta}.
\end{proof}

\vspace{5 pt}
 
We have focused on $x,y,z$ in Proposition \ref{EqsExpP} because, as we show next, distinct solutions  to $f_a(Z) = 0$ in $R_c$ differ by elements of $\Mu_{p^4-1}$.  

\begin{lem}  \label{zz'}
Let $Q_z$ lie in the fiber over $P_a \ne 0.$ Then every other point in the same fiber has the form $Q_{z'}$ with $z' = z + \omega$ in $R_c$ as $\omega$ ranges over $\Mu_{p^4-1}$.  If so, 
\begin{equation} \label{y'x'}
y' = y +\omega^p \text{ in } R_b, \quad x' = x +\omega^{p^2} \text{ in } R_a
\end{equation}
and $Q_{z'} = Q_z + \iota(P_{a'})$ with $a' = \omega^{p^2} a$ in $\gR_\lambda$.
\end{lem}

\begin{proof}
We have $f_a(z) = 0$ in $R_c$ and we use \eqref{Newxs}(i) and (ii) to find $y$ and $x$.   Putting $z' = z+\omega$ and using \eqref{y'x'} to define $y'$ and $x'$ gives another solution to the congruences \eqref{Newxs}, thereby accounting for the additional $p^4-1$ points $Q_{z'}$ in the fiber over $P_a$.  

Let $Q_{z'} = Q_z + \iota(P_{a'})$ and evaluate the corresponding $D_k$-homomorphisms at $e_5$ to find the equation of Witt covectors
$\bfd_a(z') = \bfd_a(z) \CWplus (\0,c',b',a')$.  This sum reduces to ordinary addition on coordinates in $k$.  Indeed, apply Verschiebung twice and use Lemma \ref{Wittadd} to get $cz' = cz+c'$ and so $c' = \omega c$ in $k$.  By Remark \ref{Pa}(iii), $c'$ determines $b'$ and $a'$.  In particular, the various lifts satisfy
$$
(-p)^{p+1} a' \equiv (c')^{p^2} \equiv \omega^{p^2} c^{p^2} \equiv (-p)^{p+1} \omega^{p^2} a \pmod{p^{p+2} \cO_{\ov{K}}}. 
$$
Hence $a' = \omega^{p^2} a$ in $k$ and  similarly  $b' = \omega^{p}b$ in $k.$
\end{proof}

The next lemmas treats special cases used in the following subsection to describe Kummer generators  when $p=2$.

\begin{lem}  \label{fiberchange}
If $P_a \ne 0$, then the field $F_a$ of points of the fiber over $P_a$ equals the full field of points $K(W)$ for the Honda parameters in {\eqref{roote}}.
\end{lem}

\begin{proof}
If $\bfs = [s_1  s_2  0  0  0]$, use the first form of $f_a(Z)$ in Proposition \ref{EqsExpP} with $w = s_2 a$.  In the remaining cases below, $\epsilon_p = 0$ and the simpler equation for $f_a(Z)$ holds.  Note that   $f_{\eta a}(\eta^e Z)  = \eta^e f_a(Z)$ for all $\eta$ in $\Mu_{p^4-1}$, with $e$ given by:  \begin{equation} \label{roote}
\begin{array}{ c || c | c | c | c }
\bfs & [s_1s_2000] & [0  0  s_3  0  0] & [0  0  0  0  s_5]
 & [0  0  0  s_4  0]  \\
\hline
e & 1-p^2 & p^3-p^2 &  p-p^2 & 0
\end{array}
\hspace{3 pt} .
\end{equation}
The correspondence between the roots of $f_a(Z)$ and those of $f_{\eta a}(Z)$ induced by $z \leftrightarrow \eta^e z$ shows that $F_{\eta a}= F_a$ and so each of these fields equals $K(W)$. 
\end{proof}

\begin{prop}  \label{CondExpP}
If $\cW$ is an extension of $\cE_\lambda$ by $\cE_\lambda$ killed by $p$ and  $L = K(W)$, then its abelian conductor exponent satisfies $\gf(L/F) \le p^2$.  Moreover, $\gf(F'/F) \le p^2$ for every intermediate field $F'$ of $L/F$.
\end{prop}

\begin{proof}
Let $\bfs(\cW) = [s_1s_2s_3s_4s_5]$ and write $s_1 = \epsilon_p+\delta_p$, with $\delta_p = 0$ for odd primes $p$ and $\delta_2 = (\lambda s_2)^2$.  Then $\cW = \cW_1 + \dotsc + \cW_5$ is a Baer sum of group schemes corresponding to the sum of Honda parameters:
\begin{equation}\label{special}  
[\epsilon_p0000]+[\delta_ps_2000]  + [00 s_3 0 0 ] + [000s_40] + [0000s_5],
\end{equation}
some of which may be trivial.   For the fiber fields  $F_a^{(j)}$ of each of these $W_j$, we show that $\gf(F_a^{(j)}/F) \le p^2$ in the next lemmas.  Since $F_a$ is contained in the compositum of all $F_a^{(j)}$, we then have $\gf(F_a/F) \le p^2$ by Lemma \ref{cond3}.  Furthermore, $L$ is the compositum of all $F_a$ as $P_a$ varies over $E_\lambda$, so $\gf(L/F) \le p^2$.  Finally $\gf(F'/F) \le p^2$ because the upper ramification numbering behaves well for quotients.
\end{proof}

\begin{Rem} \label{FB}
In contrast to the Proposition, Fontaine's higher ramification bound leads to $\gf(L/F) \le p^2+2$ by Proposition \ref{FontCond}, since Proposition \ref{FieldForE}(ii) gives $e_{F/K} = e_F = (p^2+1)(p-1)$.  In particular, when $p=2$, the sharper bound is essential for our applications.  
\end{Rem}

We next verify the lemmas needed for the proof of the Proposition.  For $P_a \ne 0$ and $f_a$ as in Proposition \ref{EqsExpP}, recall that $F_a = F(Q_z)$, where $f_a(z) = 0$ in $R_c$.  Let $\pi_a$ be a uniformizer of $F_a$.

\begin{lem} \label{s4cond}
If $\bfs = [0 0 0 \hspace{.4 pt} s_4 \hspace{.4 pt}  0]$, then $F_a/F$ is unramified of degree $1$ or $p$.
\end{lem}

\begin{proof}
The claim follows from separability of $f_a(Z) = Z^{p^4}-Z+s_4$ over $k$.  
\end{proof}

\begin{lem} \label{BasicConds}
For the parameters $\bfs$ below, $F_a/F$ is totally ramified of degree $p^4$.
\begin{enumerate}[{\rm i)}]        
\item If $\bfs = [s_10000]$ with $s_1\ne 0$, then $\gf(F_a/F) = p^2-2p+2$.  \vspace{2 pt}
\item Let $\bfs = [s_1 \hspace{.4 pt} s_2 \hspace{.4 pt} s_3 \hspace{.4 pt}  s_4 \hspace{.4 pt} s_5]$, with $s_2 \ne 0$.  Set $s_1 = 0$ for odd $p$ and $s_1 = (\lambda s_2)^2$ for $p = 2$.  Then $\epsilon_p = 0$ for all $p$ and $\gf(F_a/F)= p^2$.   \vspace{2 pt}
     \item If $\bfs = [0 0  s_3  s_4 0]$ and $s_3 \ne 0$, then $\gf(F_a/F) = p$.   \vspace{2 pt}
     \item If $\bfs = [0 0 0  s_4  s_5]$ and $s_5 \ne 0$, then $\gf(F_a/F) = p$.
\end{enumerate}
\end{lem}

\begin{proof}
To find the conductors, we determine $t$ in $F_a$ to which Proposition \ref{cond1} applies.  In all cases below, $g(t) - t$ is in $\Mu_{p^4-1}$ for all $g \ne 1$ in $\Gal(F_a/F)$ by Lemma \ref{zz'} and $F_a = F(t)$.  

In case (i),  let $F_a = F(\vartheta)$ as in Corollary \ref{vartheta} and let $y$ be the image of $\vartheta$ in $R_b$.  Observe that by \eqref{yval}, $F_a/F$ is totally ramified of degree $p^4$ and we have
$
\ord_{\pi_a}(y) = \ord_p(y) \, \ord_{\pi_a}(p) = -(p-1)^2.
$ 
Using $t = y$ gives $\gf(F_a/F) = p^2-2p+2$.

In the remaining cases, $\epsilon_p = 0$ and $F_a = F(\theta)$ as in Corollary \ref{theta}, with $\theta$ a root of $\tilde{f}(Z) = 0$ in $\cO_{\ov{K}}$ and $\tilde{f}$ a lift of the simpler version of $f_a$ in Proposition \ref{EqsExpP}.  If $z$ is the image of $\theta$ in $R_c$, then $p^4 \ord_p(z) = \ord_p(w) - p^2 \ord_p(a)$ and so we have:
$$
\begin{array}{ c || c | c | c }
\text{case}  &     {\rm (ii)} & {\rm (iii)}  & {\rm (iv)}      \\
\hline  
\ord_p(z) & -\frac{p+1}{p^4(p^2+1)} & -\frac{1}{p^4(p^2+1)} &- \frac{1}{p^3(p^2+1)}
\end{array} \hspace{4 pt}.
$$
In cases (ii) and (iii), observe that $F_a$ is totally ramified of degree $p^4$ over $F$, with $\ord_{\pi_a}(z) = 1-p^2$ and $1-p$ respectively.  We use $t = z$ to determine $\gf(F_a/F)$.

In case (iv), $w = s_5 a^p + s_4 a^{p^2}$.  Choose $\beta \in \W^\times$ such that $\beta^p \equiv s_5  \! \pmod{p \W}$ and let $t = \theta^{p^3} +  \beta a^{1-p}$.  By Lemma \ref{ppower}, with $O$-notation from the start of \S\ref{FP}, 
$$
t^p = \theta^{p^4} + s_5 a^{p-p^2} + O(\pi_a) = \theta - s_4 + O(\pi_a), 
$$
so $\ord_p(t) = \frac{1}{p} \ord_p(\theta) = \frac{1}{p^4(p^2+1)}$.  Hence the ramification index of $F(t)/F$ is at least $p^4$.  Since $F(t) \subseteq F_a$ and $\fdeg{F_a}{F} \le p^4$, we have $F_a = F(t)$, totally ramified over $F$.  If $g(z) = z + \omega$ as in Lemma \ref{zz'}, then 
$$
g(t) - t = g(\theta)^{p^3} - \theta^{p^3} \in (\theta+\omega + \pi_a\cO_{\ov{K}})^{p^3} - \theta^{p^3} \subseteq \omega^{p^3} + \pi_a\cO_{\ov{K}}.
$$
Proposition \ref{cond1} therefore applies with $\ord_{\pi_a}(t) = 1-p$ to give $\gf(F_a/F) = p$.  
\end{proof}

\subsection{Local corners}  \label{CornerSection}
For this subsection, $p = 2$ and $K = \Q_2$.  Let $\cE$ be the simple group scheme $\cE_\lambda$ of Notation \ref{Ebasis}, with $\lambda = 1$ necessarily.  Let $E$ be the Galois module of $\cE$, $F = \Q_2(E)$ and $\Delta = \Gal(F/\Q_2)$.   By Proposition \ref{FieldForE}, $F = \Q_2(\Mu_{15}, \varpi)$, with uniformizer $\varpi$ satisfying $\varpi^5 = 2$.  Fix a generator $\sigma$ of the inertia subgroup of $\Delta$ and a Frobenius $\tau$ generating $\Gal(F/\Q(\varpi))$ with $\tau \sigma \tau^{-1} = \sigma^2.$ Then $\Delta = \langle \sigma, \tau \rangle$ is isomorphic to the Frobenius group of order 20 and  $E$ is the unique non-trivial irreducible  module over $\rR = \F_2[\Delta]$.

Let $W$ represent a class in $\Ext^1_{[2],\Q_2}(E,E)$, $L = \Q_2(W)$ and $\gh = \Hom_{\F_2}(E,E)$. Then $[W]$ corresponds to a cohomology class $[\psi]$ in $H^1(\Gal(L/\Q_2),\gh)$ such that 
\begin{equation}  \label{RepViaCocycle}
\rho_W(g) = \left[\begin{smallmatrix} \rho_E(g) & \psi(g) \, \rho_E(g)  \\ 0 & \rho_E(g) \end{smallmatrix} \right]  \quad \text{for all } g \in \Gal(L/\Q_2), 
\end{equation}
as in \eqref{RepFromPsi}.  We introduce {\em corners} to rigidify $\psi$ and facilitate comparison with the cocycles arising from global extensions.

Suppose that $V$ is any finitely generated $\rR$-module and let $T_\sigma = \sigma^4 + \sigma^3 + \sigma^2 + \sigma + 1$ in $\rR$ be the trace with respect to $\sigma$.   Since $\sigma$ has odd order, $V = V_0 \oplus V'$, where $V_0$ is the submodule on which $\sigma$ acts trivially and $V' = \ker T_\sigma = (\sigma-1)(V)$.   The {\em corner subgroup} of $V,$  which depends on the choice of $\tau$, is defined as
$$
\Cor(V) = \{ v \in V \, \vert \, \tau(v) = v \text{ and } T_\sigma(v) = 0 \}.
$$
If $v_1, \dots, v_n$ is an $\F_2$-basis for $\Cor(V)$, then $\rR v_i \simeq E$ and $V' = \bigoplus_{i=1}^n \rR v_i$.

We consistently write $P$ for the unique non-zero element of $\Cor(E)$, so $P = P_\varpi$ as in Proposition \ref{FieldForE}(iii) and $P$, $\sigma(P)$, $\sigma^2(P)$, $\sigma^3(P)$ is an $\F_2$-basis for $E$ affording the matrix representations 
\begin{equation} \label{stEnd}
s = \rho_E(\sigma) = \left[\begin{smallmatrix}0&0&0&1\\ 1&0&0&1\\0&1&0&1\\0&0&1&1\end{smallmatrix}\right] \quad \text{and} \quad t = \rho_E(\tau) = \left[\begin{smallmatrix} 1&0&1&0\\ 0 &0&1&1\\0&1&1&0\\0&0&1&0\end{smallmatrix}\right].
\end{equation}
We will also use the twisted action of $\F_{16}$ on $E$ described in Remark \ref{Pa}.  If  a primitive fifth root of unity $\zeta$ in $\cO_F$ is defined by $\sigma(\varpi) = \zeta \varpi$, then $\sigma(\alpha P) = \alpha \zeta P$ and $\tau(\alpha P) = \tau(\alpha) P$ for all $\alpha$ in $\F_{16}$. 

The endomorphisms $s$ and $t$ belong to $\gh$, with respective minimal polynomials $s^4+s^3 +s^2+s+1 = 0$ and $t^4-1 = 0$.  We next describe $\gh$ as an $\rR$-module.

\begin{lem} \label{gh}
An $\F_2$-basis for $\gh_0 = \ker((\sigma-1) \, \vert \, \gh)$ is $1, s, s^2, s^3,$ with $\tau$ acting on $\gh_0$ as one Jordan block.   An $\F_2$-basis for $\Cor(\gh)$ is  $t, \,t^2, \, t^3.$   We have $\gh \simeq \gh_0 \oplus_{j=1}^3 \rR t^j$, with each $\rR t^j \simeq E$.  The cohomology group $H^1(\Delta,\gh)$ vanishes.  
\end{lem}

\begin{proof}
The elements of $\gh_0$ are precisely the $\F_2[s]$-endomorphisms of $E$.  Since $E$ is a cyclic $\F_2[s]$-module, $\End_{\F_2[s]}(E) = \F_2[s] \simeq \F_{16}$.  The action of $\tau$ on $\gh_0$ is the action of Frobenius on $\F_{16}$ and thus has one Jordan block.  Similarly, the elements of $\Cor(\gh)$ are $\F_2[t]$-endomorphisms of $E$, so contained in $\F_2[t]$.  But only the linear combinations of $t, \, t^2, \, t^3$ are annihilated by the action of $T_\sigma$ on $\gh$. 

We have $H^1(\langle \tau \rangle, \gh_0) =  H^1(\langle \tau \rangle, \F_{16}) = 0$ by the additive Hilbert Theorem 90 and $H^1(\langle \sigma \rangle, \gh) = 0$ because $\sigma$ has odd order.  Applying inflation-restriction with respect to the exact sequence $1 \to \langle \sigma \rangle \to \Delta \to \langle \tau \rangle \to 1$ shows that $H^1(\Delta, \gh) = 0$.
\end{proof} 

\begin{Not}\label{cors}
For $t$ as in \eqref{stEnd}, the following elements comprise $\Cor(\gh)$.  Their labels are consistent with Notation \ref{Ca}.
\begin{equation} \label{corners}
\begin{array}{l l l l}
\gamma_0 = 0, & \gamma_4 = t+t^2+t^3, & \gamma_5 = t + t^3, & \gamma_9 = t^2,   \vspace{4 pt} \\
\gamma_{11} = t+t^2, & \gamma_{11}'  = t^2 + t^3, & 
\gamma_{15} = t^3,  & \gamma_{15}'  = t. 
\end{array}
\end{equation}
All occur as values of extension cocycles for $E$ by $E$ when we range over Honda parameters, cf.\! Proposition \ref{dim48} below. 
\end{Not}

Motivated by the conductor bound in Proposition \ref{CondExpP}, we assume from now on that $\gf_\gp(L/F) \le 4$.  If $T$ is the maximal elementary 2-extension of $F$ with ray class conductor exponent 4, then $T$  is Galois over $\Q_2$ and we denote the action of $\delta $ in $\Delta$ on elements $h$ of $\Gamma = \Gal(T/F)$ by  
$
\lexp{\delta}{h} = \tilde{\delta} h \tilde{\delta}^{-1} 
$
independent of the choice of lift $\tilde{\delta}$ of $\delta$ to $\Gal(T/\Q_2)$.  We also write $\sigma$ for an element of order 5 in $\Gal(T/\Q_2)$ projecting to $\sigma$ in $\Delta$.   We have the following diagram of fields and Galois groups, 
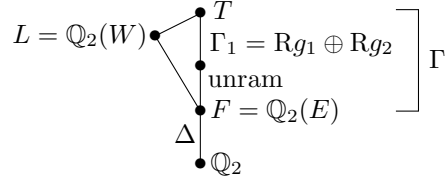
\begin{center}
\begin{tikzpicture}[scale=1, thin, baseline=(current  bounding  box.center)]	
    \draw (0,0) -- (0,2);
    \draw (0,0) node {$\bullet$};   \draw (.35,0) node {$\Q_2$};
    \draw (-.2,.4) node {$\Delta$};
    \draw (0,.7) node {$\bullet$};  \draw (1,.7) node {$F = \Q_2(E)$};  
    \draw(.6,1.1) node {unram};
    \draw (0,1.3) node {$\bullet$};     
    \draw (1.35,1.59) node {$\Gamma_1 = \rR g_1 \oplus \rR g_2$};
    \draw (0,2) node {$\bullet$};   \draw (.3,2.05) node {$T$};
    \draw (0,.7) -- (-.6,1.7) -- (0,2);
    \draw (-.6,1.7) node {$\bullet$};  \draw (-1.6,1.7) node {$L = \Q_2(W)$};  
    \draw (2.6,2.05) -- (2.9,2.05) -- (2.9,.7) -- (2.6,.7); \draw  (3.15,1.4) node {$\Gamma$};
\end{tikzpicture}
\end{center}
where $\Gamma_1$ is the wild ramification subgroup (see Appendix \ref{condapp}) of $\Gamma$.  We next describe the complete lower ramification filtration on $\Gamma$ and its structure as a module for $\rR = \F_2[\Delta]$.

\begin{prop} \label{Max2}
Let  $g_0 = \Artin (\varpi, T/F)$, $g_1 = \Artin (1+\varpi+\varpi^3, T/F)$ and $g_2 = \Artin(1+\varpi^3, T/F)$.  Then 
$
\Gamma = \rR g_0 \oplus \rR g_1 \oplus \rR g_2 \simeq \F_2 \oplus E \oplus E
$
and
$$
\Gamma_1 \rhd \Gamma_2 = \Gamma_3  \rhd  \Gamma_4 = \{1\},
$$
with $\Gamma_1 = \rR g_1 \oplus \rR g_2$ and $\Gamma_3 = \rR g_2$.  There is a Frobenius $\Phi$ of order $8$ in $\Gal(T/\Q_2)$ projecting to $\tau$ in $\Delta$ and satisfying $\Phi \sigma \Phi^{-1} = \sigma^2$.  In addition, $\Gal(T/\Q_2)= \Gamma_1 \rtimes H$ with $H = \langle \sigma, \Phi \rangle$. 
\end{prop}

\begin{proof}
We use the standard filtration $U_F^{(n)}$ on local units, see \eqref{nUnits}.  The $\rR$-module structure of $\Gamma$ follows from the class field theory isomorphism   
$$
\Artin( - , T/F)\!: \,  F^\times/U_F^{(4)} F^{\times 2} \, \xrightarrow{\sim} \, \Gamma.
$$
In particular, $\rR$ acts trivially on the Frobenius $g_0$ of $\Gamma$, while $\rR g_1$ and $\rR g_2$ are isomorphic to $E$ as $\rR$-modules.  Since $\Gamma_1 = \Artin(U_F,T/F)$, we have $\Gamma_1 =  \rR g_1 \oplus \rR g_2$ and similarly for $\Gamma_2$, using $U_F^{(2)} \subset U_F^{(3)} F^{\times 2}$.   Note that
$$
\Gamma_1 = \ker (T_\sigma \vert \Gamma) = \Image ((\sigma-1) \vert \Gamma).  
$$
There is a residue extension of degree 2 for $T/F$, so Frobenius $\Phi$ projecting to $\tau$ has order 8.  Set $\Phi \sigma^3 \Phi^{-1} = h \sigma$ for some $h$ in $\Gamma_1$.  By direct computation, $T_\sigma(h) = (h \sigma)^5 = (\Phi \sigma^3 \Phi^{-1})^5 = 1$.  Hence $h = \lexp{\sigma}{x}/x$ for some $x$ in $\Gamma_1$ and so $(x \Phi) \sigma^3 (x \Phi)^{-1} = \sigma$.    Replace $\Phi$ by $x \Phi$ to guarantee that $\Phi \sigma \Phi^{-1} = \sigma^2$.   Then $\Delta$ acts trivially on $\Phi^4$, so $\Phi^4 = g_0$.  Since $H = \langle \sigma, \Phi \rangle$ is isomorphic to the Galois group of the maximal tame extension of $F$ in $T$, we find that $\Gal(T/\Q_2)$ is a semi-direct product of $H$ by the normal subgroup $\Gamma_1$.
\end{proof} 

Let $r_{T/L}: \Gal(T/\Q_2) \twoheadrightarrow \Gal(L/\Q_2)$ be the natural projection.  Note that the inertia group $\Gal(L/F)_1$ of $\Gal(L/F)$ is the wild ramification subgroup $\Gal(L/\Q_2)_1$ of $\Gal(L/\Q_2)$. 

\vspace{5 pt}

\begin{cor} \label{classfield}
The subgroup $\ov{H} = r_{T/L}(\langle \sigma, \Phi \rangle)$ of $\Gal(L/\Q_2)$ projects onto $\Delta$ in $\Gal(F/\Q_2)$. As $\rR$-modules, $\Gal(L/F)_1 \simeq E^b$, with $0 \le b \le 2$.   
\begin{enumerate}[{\rm i)}]
\item If $L/F$ is totally ramified, then $\Gal(L/F) = \Gal(L/F)_1$ and $\vert \ov{H} \vert = 20$.   

\vspace{3 pt}

\item Otherwise, $L/F$ has residue degree $2$, $\Gal(L/F) \simeq \Gal(L/F)_1 \oplus \F_2$ and $ \ov{H} $ has order and exponent $40.$  
\end{enumerate}
\end{cor}

\begin{proof}
That $\ov{H}$ projects onto $\Delta$ and that $\Gal(L/F)_1 = r_{T/L}(\Gamma_1)$ is the direct sum of at most 2 copies of $E$ is immediate.  Moreover, $L/F$ is totally ramified if and only if $g_0 = \Phi^4$ is in $\ker r_{T/L}$.  Thus $\vert \ov{H} \vert = 20$ in case (i) and 40 in case (ii).
 \end{proof}

Since $T$ contains $L = \Q_2(W)$, the cocycle $\psi$ in \eqref{RepViaCocycle} inflates to $\Gal(T/\Q_2)$.   We may arrange for $\psi(\sigma) = 0$, since $\sigma$ has odd order.   Lemma \ref{gh} and \eqref{InfRes} give injectivity of the restriction map:
\begin{equation} \label{ToToverF}
0 \to H^1(\Gal(T/\Q_2),\gh) \xrightarrow{\res} H^1(\Gamma,\gh)^\Delta = \Hom_\rR(\Gamma,\gh)
\end{equation} 
and we say that $\chi = \res([\psi])$ in $\Hom_\rR(\Gamma,\gh)$ {\em belongs to} $W$.   Note that $\chi$ is determined by its values on $g_0$, $g_1$, $g_2$, as defined in Proposition \ref{Max2}.   

\begin{lem} \label{chiLem}
The field $L = \Q_2(W)$ is the fixed field of $\ker \chi$.  Moreover:
\begin{enumerate}[{\rm i)}] 
\item  $\chi(g_i)$ is in $\Cor(\gh)$ for $i = 1,2$ and $\chi(g_0)$ is in $\{0,I_4\}$.  \vspace{2 pt}
\item $L/F$ is unramified if and only if $\chi(g_1) = \chi(g_2) = 0$. \vspace{2 pt}
\item $\gf(L/F) = 4$ if and only if $\chi(g_2) \ne 0$.  If $\chi(g_2) = 0$, then $\gf(L/F) = 0$ or $2$.   \vspace{2 pt}
\item The residue degree of $L/F$ is $1$ or $2$, according to whether $\chi(g_0) = 0$ or $I_4$. 
\end{enumerate}
\end{lem}

\begin{proof}
The matrix representation \eqref{RepViaCocycle} shows that $g$ in $\Gal(T/\Q_2)$ acts trivially on $W$ if and only if $g$ is in $\Gamma = \Gal(T/F)$ and $\chi(g) = 0$.  Then items (i)--(iv) immediately follow from Proposition \ref{Max2}.  In particular, (i) holds by considering the action of $\Delta$ on $g_0$, $g_1$ and $g_2$.
\end{proof}

Write $\cW_{\bfs}$ for the extension of $\cE$ by $\cE$ of exponent 2 with Honda parameter $\bfs$ and $W_{\bfs}$ for its Galois module.  Belonging to $W_\bfs$ are the cohomology class $[\psi_\bfs]$ in $H^1(\Gal(T/\Q_2),\gh)$ and its restriction $\chi_\bfs$ in $\Hom_{\F_2[\Delta]}(\Gamma, \gh)$, as described above.  The rest of this section is devoted to evaluating $\chi_{\bfs}$ as $\bfs$ varies. 

If $h$ is in $\Gamma = \Gal(T/F)$ and $Q_{z_j}$ is any point in the fiber over $\sigma^{j}(P)$, cf.\! Notation \ref{zNot}, any basis of the form
\begin{equation} \label{BetterBasis}
P, \,  \sigma(P), \, \sigma^2(P), \,  \sigma^3(P), \, Q_{z_0}, \,  Q_{z_1},  \, Q_{z_2}, \,  Q_{z_3}
\end{equation}
yields the same matrix $\rho_{W_{\bfs}}(h)$ in \eqref{RepViaCocycle}.  Moreover,  $h(Q_{z_j}) = Q_{z_j} + \chi_{\bfs}(h) \, \sigma^j(P)$.

Let $M/F$ be a finite elementary 2-extension.  Define its {\em Kummer group} by
$$
\Kappa(M/F) =  F^\times \cap M^{\times 2} \quad \text{and let} \quad \ov{\Kappa}(M/F) = \Kappa(M/F)/F^{\times 2}.
$$   
By definition, $F^{\times 2} \subseteq\Kappa(M/F)$ and we have
$
M = F(\{\sqrt{\theta} \, \vert \, \theta \in \Kappa(M/F) \}).
$  
Kummer theory gives a perfect pairing: 
$$
\Gal(M/F) \times \ov{\Kappa}(M/F) \to \Mu_2 \quad \text{by} \quad (g,\theta) \mapsto g(\sqrt{\theta})/\theta.
$$
 
\begin{lem}  \label{FirstKummerCor}
Let $P = P_\varpi$ and let $F_\varpi$ be the subfield of $L$ generated by the points of $W_{\bfs}$ in the fiber over $P$.   If $\bfs = [10000]$, then $\Kappa(F_\varpi/F)$ contains $1+2\varpi^4$.  If $s_1 = s_2$, then $\Kappa(F_\varpi/F)$ contains $1+ 2s_2 \varpi^2 + 2(s_3+s_5)\varpi^4$.  
\end{lem}

\begin{proof}
Refer to Proposition \ref{EqsExpP}.  Since $p=2$ and $\lambda = 1$, we have $\epsilon_2 = 0$ when $s_1 = s_2$.  Then take the square class of the discriminant of the polynomial $h(X)$ in Corollary \ref{theta}.    Similarly, use Corollary \ref{vartheta} when $s = [10000]$.
\end{proof}

We first determine $\chi_{\bfs}$ when $L/F$ is a non-trivial totally ramified extension.   For compatibilty with the notation for decomposition groups in \S \ref{Punchline}, where we consider global Galois module extensions of $E$ by $E$, set $\cD_\gp(L/F) = \Gal(L/F)$.

\begin{prop}  \label{dim48}
If $L/F$ is totally ramified, then $\chi_{\bfs}(g_0) = 0$.  Depending on the conductor exponent $\gf(L/F)$, we have:
\begin{enumerate}[{\rm i)}]
\item \fbox{$\gf(L/F) = 2$.}  Then $\vert \cD_\gp(L/F) \vert = 16$,  $\chi_s(g_2) = 0$ and 
$$
 \begin{array}{ c || c | c | c | c | c | c | c }
\bfs & [00001] & [00100] & [10000] & [10101] & [00101] & [10001] & [10100]  \\
\hline
\chi_{\bfs}(g_1) &  \gamma_{15} &  \gamma_{15}'  &  \gamma_9 & \gamma_4 &  \gamma_5 & \gamma_{11}' &  \gamma_{11}      \\
\end{array}
\hspace{2 pt}  .
$$

\vspace{3 pt}

\item \fbox{$\gf(L/F) = 4$ and $\vert \cD_\gp(L/F) \vert = 16$}  Then $\chi_{\bfs}(g_2) = \gamma_9$ and $\chi_\bfs(g_1) = 0$ or $\gamma_9$ according to whether $\bfs = [11000]$ or $[01000]$.

\vspace{8 pt}

\item \fbox{$\gf(L/F) = 4$ and $\vert \cD_\gp(L/F) \vert = 256$.}  Then $\chi_s(g_2) = \gamma_9$ and 
$$
\begin{array}{ c || c | c | c | c | c | c  } 
\bfs & [11001] & [11100] & [01101] & [11101] & [01001] & [01100]    \\
\hline
\chi_{\bfs}(g_1)  & \gamma_{15} & \gamma_{15}' & \gamma_4  & \gamma_5 & \gamma_{11}'   & \gamma_{11} \\
\end{array}
\hspace{2 pt} .
$$
\end{enumerate}
\end{prop}

\begin{proof}
We begin with some basic Honda parameters, from which the others can be generated by Baer sum.  Recall that $F_a$ denotes the extension of $F$ obtained by adjoining the coordinates of the points in the fiber of $W_{\bfs}$ above one point $P_a$ of order 2 in $E$. \vspace{3 pt}

\noindent {\bf Basic Cases}: (1) $\bfs = [00001]$, $[00100]$ or $[10000]$.   By Lemma \ref{BasicConds}, $F_a/F$ is totally ramified of degree 16 and $\gf(F_a/F) = 2$.  Thus $\chi_{\bfs}(g_0) = \chi_{\bfs}(g_2) = 0$ by Lemma \ref{chiLem} and so $L =  F_a$ is the subfield of $T$ fixed by $\rR g_0 \oplus \rR g_2$ independent of $a$.

\hspace{48 pt} (2)  $\bfs = [11000]$.  Lemma \ref{fiberchange} indicates that $L = F_a$ does not depend on $a$.  Now $L/F$ is totally ramified of degree 16 and $\gf(L/F) = 4$ by Lemma \ref{BasicConds}, so $\chi_{\bfs}(g_0) = 0$ but $\chi_{\bfs}(g_2) \ne 0$.  By Lemma \ref{FirstKummerCor}, the Kummer group $\ov{\Kappa}(L/F)$ contains the coset $\kappa = (1+2\varpi^2)F^{\times 2}$ and therefore equals $\rm R \, \kappa$.  By evaluating the pairing of Kummer theory and class field theory given by Hilbert symbols, we find that $g_1$ acts trivially on the square roots of elements of $\ov{\Kappa}(L/F)$, so $\chi_{\bfs}(g_1) = 0$.

\vspace{2 pt}

Set $h = g_1$ in the basic case (1) and $h = g_2$ in (2).  Recall that the primitive fifth root of unity $\zeta$ is defined by $\sigma(\varpi) = \zeta \varpi$.  To find the matrix $\chi_{\bfs}(h)$, we use a basis for $W_{\bfs}$ of the form
$$
P, \, \sigma(P), \, \sigma^2(P), \, \sigma^3(P), \, Q_{z_0}, \,  Q_{z_1},  \, Q_{z_2}, \,  Q_{z_3},
$$
where $z_j$ is a root of the Honda polynomial $f_{\zeta^j \varpi}$, cf. Notation \ref{zNot}.  The action of $\Delta = \Gal(F/\Q_2)$ puts $h$ in the corner group of $\cD_\gp(L/F)$, so $\chi_{\bfs}(h)$ is in $\Cor(\gh)$ and therefore equals one of the matrices in \eqref{corners}.   In particular, $\chi_{\bfs}(h)(P) = \alpha_0 P$, with $\alpha_0 = 0$ or 1.  Write $h(Q_{z_j}) = Q_{z_j} + \alpha_j P$, where $\alpha_0 = 0$ or 1 and
$$
\alpha_j = c_{0j} + c_{1j} \zeta + c_{2j} \zeta^2 + c_{3j} \zeta^3 \, \text{ in } \, \Z[\zeta] \, \text{ for } 1 \le j \le 3.
$$ 
Then the $(j+1)$-column of the matrix $\chi_{\bfs}(h)$ is $[c_{0j},c_{1j},c_{2j},c_{3j}]^T \bmod 2$ by \eqref{BetterBasis}. 

From $h(Q_{z_0}) = Q_{z_0} + \alpha_0 P$, we get $h(z_0) = z_0+\alpha_0$ by Lemma \ref{zz'}.  In the proof of Lemma \ref{fiberchange}, we showed that there is a correspondence between roots of $f_\varpi$ and $f_{\zeta^j \varpi}$, allowing us to choose $z_j = \zeta^{je}z_0$, with $e$ given by \eqref{roote} and $j=1,2,3$.  Then $h(z_j) = z_j + \alpha_0 \zeta^{je}$ in $R_c$.  Since $h$ is not trivial on $L$, we have $\alpha_0 = 1$.  Further use of Lemma \ref{zz'} gives
$$
h(Q_{z_j}) = Q_{z_j} + \zeta^{4je} P_{\zeta^j \varpi} =  Q_{z_j} + \zeta^{(1-e)j} P.
$$    
This determines $\chi_{\bfs}(h)$ for all $\bfs$ in the Basic Cases.  

\vspace{2 pt}

\noindent {\bf Remaining Cases}.  Write $\bfs = \bft + \bfu$, choosing Honda parameters $\bft$ and $\bfu$ already treated above.  Then $W_{\bfs}$ is the Baer sum of $W_\bft$ and $W_\bfu$ and $\chi_{\bfs} = \chi_\bft + \chi_\bfu$.  

In (ii), use $ [01000] = [11000] + [10000]$.  In (i), the last three entries follow by varying  $\bft$ and $\bfu$ among first three entries. Use $[10101] = [10000]+[00101]$ to complete (i).  For (iii), let $\bft = [11000]$ and let $\bfu$ run over the Honda parameters in (i), omitting $[10000]$.  Since $g_1$ and $g_2$ are independent and non-trivial on $L$, we have $\Gal(L/F) = Rg_1 \oplus Rg_2$ of order 256.
\end{proof}

We briefly treat the  remaining 16 non-trivial Honda parameters, even though Lemma \ref{GlobalAt2} shows that they are not needed for our  global  applications.

\begin{prop}  \label{dim59}
If $L/F$ is not totally ramified, then $\bfs = \bft + \bfu$, where $\bft$ ranges over $[00000]$ and the $15$ Honda parameters in {\rm Proposition \ref{dim48}}, while $\bfu = [00010]$.   Then $\chi_{\bfs}(g_0) = I_4$, $\chi_{\bfs}(g_j) = \chi_\bft(g_j)$ for $j=1,2$ and $\Q_2(W_\bfs)$ is the compositum of $\Q_2(W_\bft)$ and the unramified quadratic extension of $F$. 
\end{prop}

\begin{proof}
By Lemma \ref{s4cond}, $F(W_{\bfu})$ is the splitting field of $Z^{16}-Z-1$, namely the unramified quadratic extension of $F$.  Thus $\chi_\bfu(g_0) = I_4$ and $\chi_\bfu(g_1) = \chi_\bfu(g_2) = 0$ by Lemma \ref{chiLem}.  The rest follows from $\chi_{\bfs} = \chi_\bft+\chi_\bfu$.
\end{proof}

\section{Global conclusions}  \label{Punchline}
\numberwithin{equation}{subsection}

\subsection{Favorable abelian surfaces}
There are two irreducible $\cS_5$-representations of dimension 4 over $\F_2$.  Denote the one taking transpositions to transvections by $\iota\!: \, \cS_5 \to \SL_4(\F_2)$ and fix it by sending $(12) \mapsto r$ and $(12345) \mapsto s$, where
\begin{equation}  \label{stMatrices}
\begin{array}{l l l l}  
r = \left[\begin{smallmatrix}0&1&0&0\\ 1&0&0&0\\0&0&1&0\\0&0&0&1\end{smallmatrix}\right]  \text{ and } \,  s = \left[\begin{smallmatrix}0&0&0&1\\ 1&0&0&1\\0&1&0&1\\0&0&1&1\end{smallmatrix}\right].  \quad \text{Let } t =  \left[\begin{smallmatrix}1&0&1&0\\ 0&0&1&1\\0&1&1&0\\0&0&1&0\end{smallmatrix}\right]. 
\end{array}
\end{equation}
The image of $\iota$ is isomorphic to the odd orthogonal group ${\rm O}^-_4(\F_2) \subset \SP_4(\F_2)$.  In addition, $\iota((2354)) = t$ and $\Delta = \langle s, t \rangle$ is the Frobenius group of order 20.

Fix a favorable quintic field $F_0$ with discriminant $d_{F_0/\Q} = \pm 16 N$ and Galois closure $F$.  By Proposition \ref{Fproperties}(i), the inertia group $\cI_v(F/\Q)$ at each place $v \vert N$ is generated by a transposition $\sigma_v$ when we identify $\Gal(F/\Q)$ with $\cS_5$.   In this section, $E$ is the Galois module giving $\rho_E\!: \, \Gal(F/\Q) = \cS_5 \xrightarrow{\iota} \SL_4(\F_2)$.  Using the matrices $r,s,t$ in \eqref{stMatrices}, $\sigma_v$ is conjugate to $\rho_E^{-1}(r)$, inertia at a some $\gp \vert 2$ is generated by $\sigma = \rho_E^{-1}(s)$ and $\tau = \rho_E^{-1}(t)$ is a Frobenius in the decomposition group $\cD_\gp(F/\Q) = \langle \sigma, \tau \rangle$.   Hence the restriction of $\rho_E$ to $\cD_{\gp}(F/\Q)$ agrees with the representation $\rho_{E_\lambda}$ of Definiton \ref{simple}, as normalized in \eqref{stEnd}.  By Corollary \ref{EScheme}, $E$ extends to a group scheme $\cE$ over $\Z[\frac{1}{N}]$.  Let $\un{E}$ be the $\Sigma$-category introduced in Definition \ref{ourcat} with $\Sigma = \{\cE\}$.  This subsection is devoted to criteria for the validity of axiom {\bf E4} in Theorem \ref{mypdiv}, needed to prove Theorem \ref{main}.

To treat extensions $W$ of $E$ by $E$ of exponent 2, let $\cP = \cP_{E,E}$ be the parabolic group as in \eqref{defP}.  We describe subgroups of $\cP$ in which the relevant representations $\rho_W$ take their values.

\begin{Not} \label{Ca} Let $c\!: \, \Mat_4(\F_2) \to \cP$ by $c(m)=\left[\begin{smallmatrix}1&m\\0&1\end{smallmatrix}\right]$ and  $d\!: \, \cS_5 \to \cP$ by $d(g)=\left[\begin{smallmatrix}\iota(g)&0\\0&\iota(g)\end{smallmatrix}\right]$.
Let $G_0$ be the image of $d$.  With $\gamma_a$ as in \eqref{cors} and  ${\rm S} = \F_2[\cS_5],$ we define ${\rm S}$-submodules of $\Mat_4(\F_2) = \End(E)$ with adjoint action of $\cS_5$:
\begin{equation} \label{Gammas}
\Gamma_4 = {\rm S} \, \gamma_4 \quad \Gamma_5 = {\rm S} \, \gamma_5, \quad \Gamma_9 = {\rm S} \, \gamma_9, \quad \Gamma_{11} = {\rm S} \, \gamma_{11}, \quad \Gamma_{15} = {\rm S} \,\gamma_{15}.
\end{equation} 
Let $G_a = \,  <G_0,c(\gamma_a)> \, =\,  c(\Gamma_a) \rtimes G_0$.     
\end{Not}

The radical of $G_a$ equals $ c(\Gamma_a)$ and has size   $ 2^a.$  The abelianization of $G_a$ is cyclic of order 2 and so defines the character $\epsilon_0\!: \, G_a \to \F_2$, generalizing the additive signature on $\cS_5$.   If $a = 0$ or $4$, all automorphisms of $G_a$ are inner. The center of the other $G_a$'s  is generated by $c(1)$ and there is an automorphism
\begin{equation} \label{EpRep}
\epsilon\!: \, G_a \to G_a \quad \text{ by } \quad \epsilon(g) = g \, c(1)^{\epsilon_0(g)}.
\end{equation}
When $a = 5$ or 9, $\Aut(G_a)$ is generated by $\epsilon$, modulo automorphisms induced from conjugation by elements of the normalizer of $G_a$ in $\cP$.   

The corner group of an $\F_2[\Delta]$-module consists of the elements fixed by $t$ and annihilated by the trace $T_s$.  Using Magma, we find the non-zero corners of $\Gamma_a$:
\begin{equation} \label{CorGa}
\begin{array}{| c || c | c | c | c | c |}
\hline
a & 4 & 5 & 9 & 11 & 15 \\
\hline
\Cor(\Gamma_a) - \{\gamma_0\} &  \{\gamma_4\} & \{\gamma_5\} & \{\gamma_4, \gamma_5, \gamma_9\} & \{\gamma_5, \gamma_{11}, \gamma_{11}'\} & \{\text{all}\,\gamma_i\} \\
\hline
\end{array}
\end{equation}
Inclusions among the groups $G_a$ follow from this table and are indicated  in the Hasse diagram by ascending lines: 
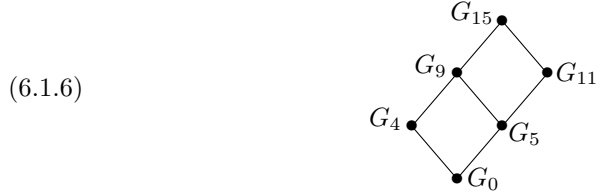
\begin{equation}  \label{GaDiagram}
\begin{tikzpicture}[scale=1, thin, baseline=(current  bounding  box.center)]	
       \draw (0,1.4) node {$\bullet$};  \draw (-.6,.7) node {$\bullet$}; \draw (.6,.7) node {$\bullet$};
       \draw (-1.2,0) node {$\bullet$}; \draw (0,0) node {$\bullet$};
       \draw (-.6,-.7) node {$\bullet$};
       \draw (-1.2,0) -- (-.6,.7)  -- (0,1.4);  \draw (.6,.7) -- (0,0)  -- (-.6,.7);  \draw (.6,.7)  -- (0,1.4);
       \draw  (-1.2,0) -- (-.6,-.7) -- (0,0); 
       \draw (-0.4,1.5) node {$G_{15}$};  \draw (-.95,.8) node {$G_9$}; \draw (1,.7) node {$G_{11}$}; \draw (0.3,-.1) node {$G_{5}$};  \draw (-1.55,.1) node {$G_4$};
       \draw (-.25,-.7) node {$G_0$};
\end{tikzpicture}
\end{equation}
Moreover, $G_9$ is isomorphic to the fiber product of $G_4$ and $G_5$ over $G_0$ and similarly for the other parallelograms.  When an inclusion $G_b \subset G_a$ exists,  Magma extends the identity  on $G_0$ to a surjection $f_{a,b}\!: G_a \twoheadrightarrow G_b$ sending $\gamma_a$  to  $\gamma_b$.   

\begin{Def} 
An involution $g$ in a group $H$ is {\em good} if its conjugates generate $H$.  If $g$ is good in $H \subseteq \cP$ and  $\rk \, (g-1) = 2$, then $g$ is {\em very good}.  
\end{Def}

\begin{Rem} \label{OneVG}
A Magma verification shows that  each $G_a$  has a unique conjugacy class of very good involutions, represented by $d(r)$ with $r$ as in  \eqref{stMatrices}. 
\end{Rem}

\begin{prop}  \label{LProperties}
Let $L$ be an elementary $2$-extension of $F = \Q(E)$, Galois over $\Q$, with $L/F$ unramified outside $\{2,\infty\}$ and $\gf_\gp(L/F) \le 4$ for all $\gp \vert 2$.  Then: 
\begin{enumerate}[{\rm i)}]
\item  The maximal subfield of $L$ abelian over $\Q$ is $\Q(\sqrt{N^*})$, with $N^* = \pm N \equiv 5 \, (8)$.

\vspace{2 pt}

\item For $v \vert N$, inertia $\cI_v(L/\Q)$ is generated by a good involution in $\Gal(L/\Q)$. 
\end{enumerate}
\end{prop}

\begin{proof}  
By Proposition \ref{Fproperties}, $F$ contains $\sqrt{N^*}$.  For $v \vert N$, the inertia group $\cI_v(F/\Q)$ has order 2.  Since $L/F$ is unramified, $\cI_v(L/\Q)$ is generated by an involution $\sigma_v$.  Intermediate fields $L\supseteq F' \supseteq F$ satisfy $\gf_\gp(F'/F) \le \gf_\gp(L/F) \le 4$.  But Lemma \ref{cond2} implies that $\gf_\gp(F(i)/F) = 6$ and $\gf_\gp(F(\sqrt{\pm 2})/F) = 11$, so $L \cap F(i, \sqrt{2}) = F$.  Since $L/\Q$ is unramified outside $\{2,N,\infty\}$, item (i) follows from Kronecker-Weber.  The subfield of $L$ fixed by the normal closure of $\sigma_v$ is unramified outside $\{2, \infty\}$ and is contained in $\Q(i)$ by \cite{BK1}, so equals $\Q$.  Thus (ii) holds.  
\end{proof}

\begin{cor}  \label{LPropCor} 
For $[W]$ in $\Ext^1_{[2],\Q}(E,E)$, assume that $L = \Q(W)$ satisfies the hypotheses in the Proposition and $\rk\rho_W(\sigma_v-1) = 2$.  Then $\rho_W(\Gal(L/\Q))$ is one of the groups $G_a$, up to conjugation in $\cP$.  If $[\cW]$ is in $\Ext^1_{[2],\un{E}}(\cE,\cE)$, then $\Gal(\Q(\cW)/\Q)$ is conjugate to some $G_a$. 
\end{cor}

\begin{proof}
By the Proposition $\rho_W(\sigma_v)$ is good and so is very good by assumption.  Magma  verifies that the $G_a$ represent the six conjugacy classes of subgroups of $\cP$ that project onto $\cS_5$ and admit very good involutions. If $[\cW]$ is a class in $\Ext^1_{[2],\un{E}}(\cE,\cE)$, then the Proposition applies to $L = \Q(W)$, since $\gf_\gp(L/F) \le 4$ by Proposition \ref{CondExpP} and  $\rk\rho_W(\sigma_v-1) = 2$ by {\bf E3} of Definition \ref{ourcat}.
\end{proof}

\begin{Def} \label{GaClass}
A class $[W]$ in $\Ext^1_{[2],\Q}(E,E)$ with $L = \Q(W)$ is a $G_a$-{\em class} if $L/F$ is unramified outside $\{2,\infty\}$, $\gf_\gp(L/F) \le 4$ for $\gp \vert 2$ and $\rk\rho_W(\sigma_v-1) = 2$, so that $\rho_W(\Gal(L/\Q)) = G_a$ for some $a$ by the Corollary. 
\end{Def}

\begin{lem} \label{ab}
Let $[W]$ be a $G_a$-class with $L = \Q(W)$.
\begin{enumerate}[{\rm i)}]
\item If $[W']$ is a $G_{a'}$-class, with $L' = \Q(W')$, then the Baer sum  $[W''] = [W] + [W']$ is a $G_b$-class for some $b$.    \vspace{2 pt}
\item If $f_{a,b} \!: G_a \twoheadrightarrow G_b$ exists in \eqref{GaDiagram}, then the Galois module for $f_{a,b} \, \rho_W$ represents a $G_b$-class.
\end{enumerate}
\end{lem}

\begin{proof} 
In (i), $[W]$ and $[W']$ correspond to classes $[\psi]$ and $[\psi']$ in $H^1(G_\Q,\gh)$ as in \eqref{RepFromPsi} and $[W'']$ belongs to the class of $\psi'' = \psi + \psi'$.  Since $L'' = \Q(W'')$ is a subfield of the compositum $LL'$, the ramification properties required of $L''$ in Definition \ref{GaClass} hold.   Proposition \ref{LProperties} shows that  $\rho(\sigma_v)$ is a good involution in $G_a$ and so is very good, conjugate to $d(r)$ by Remark \ref{OneVG}.  Similarly for $\rho_{W'}(\sigma_v)$ in $G_{a'}$.  Hence the representatives $\psi$ and $\psi'$ can be chosen to satisfy $\psi(\sigma_v) = \psi'(\sigma_v) = 0$.  We now have $\psi''(\sigma_v) = 0$ and so $\rk \rho_{W''} (\sigma_v-1) = 2$.  By Corollary \ref{LPropCor}, $[W'']$ is a $G_b$-class for some $b$.

For (ii), let $L'$ be the subfield of $L$ fixed by $\rho_W^{-1}(\ker f_{a,b})$.  Then $f_{a,b} \, \rho_W$ induces an isomorphism $\rho'\!: \, \Gal(L'/\Q) \to G_b$.  The required ramification conditions hold for the subfield $L'$ of $L$.  As above, $\rho'(\sigma_v \vert L')$ is a good involution in $G_b$.   Since $f_{a,b}$ is the identity on $G_0$ and $\rho_W(\sigma_v)$  is conjugate to $d(r)$ in $G_0$ so is $\rho'(\sigma_v \vert L' ).$ 
\end{proof}

Let $K = \Q(r_1+r_2)$ be a pair-resolvent field for $F = \Q(E)$, as defined before Theorem \ref{IntroThm}, namely the fixed field of $\Sym\{1,2\} \times \Sym\{3,4,5\}$.  Let $\Omega_K = \Omega_K^{(4)}$ be the maximal elementary 2-extension of $K$ of modulus $\gp^4 \, \infty$, where $\gp$ is the unique prime over 2 in $K$ and $\infty$ allows ramification at all archimedean places.   Refer to the following diagram of fields and Galois groups.   

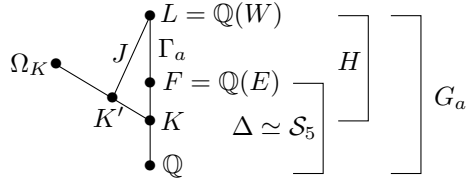
\vspace{2 pt}

\begin{center}
\begin{tikzpicture}[scale=1, thin, baseline=(current  bounding  box.center)]	
    \draw (0,0) -- (0,2);
    \draw (0,0) node {$\bullet$};   \draw (.3,0) node {$\Q$};
    \draw (0,.6) node {$\bullet$};     \draw (.3,.6) node {$K$};
    \draw (0,1.1) node {$\bullet$};   \draw (.95,1.1) node {$F = \Q(E)$};
    \draw (.3,1.55) node {$\Gamma_a$};
    \draw (0,2) node {$\bullet$};   \draw (.95,2) node {$L = \Q(W)$};
    \draw (0,.6) -- (-.5,.9) -- (0,2);    \draw (-.5,.9) node {$\bullet$};     \draw (-.55,.65) node {$K'$}; 
    \draw (-.5,.9) -- (-1.25,1.35); \draw (-1.25,1.35) node {$\bullet$};   \draw (-1.6,1.35) node {$\Omega_K$};
    \draw (2.5,2) -- (2.9,2) -- (2.9,.6) -- (2.5,.6);   \draw (2.65,1.4) node {$H$}; 
    \draw (1.9,1.1) -- (2.3,1.1) -- (2.3,-.1) -- (1.9,-.1);   \draw (1.65,.5) node {$\Delta \simeq \cS_5$};
   \draw (-.4,1.5) node {$J$};
     \draw (3.2,2) -- (3.6,2) -- (3.6,-.1) -- (3.2,-.1);   \draw (4,.9) node {$G_a$}; 
\end{tikzpicture}
\end{center}

To simplify notation, also write $\gp$ for a place over 2 in $L$ and for the restrictions of $\gp$ to subfields of $L$.  Note that primes over 2 are unramified in $F/K$.  Suppose that $L$ is the Galois closure of $K'/\Q$.  By Lemma \ref{condK'L} with $M$, $F$, $K'$, $K_1$ and $K$ there equal to the respective $\gp$-adic completions of $L$, $F$, $K'$, $K$ and $\Q$ here, $\gf_\gp(K'/K) = \gf_\gp(L/F)$.

\begin{prop}   \label{MakeGaFields} 
Let $K$ be a pair-resolvent of $F$.  There is a bijection 

\vspace{2 pt}

\centerline{\rm \{$G_a$-classes $[W]$ with $a \in \{4, 5, 9\}\} \longleftrightarrow$ \{subfields $K' \subseteq \Omega_K$ quadratic over $K$\} } 

\vspace{2 pt}
 
\noindent such that $\Q(W)$ is the Galois closure of $K'/\Q$.
\end{prop}

\begin{proof}

For $v \vert N$, $\cI_v(F/K)$ acts on the left cosets of $\Gal(F/K)$ in $\Gal(F/\Q)$ with four fixed points and three orbits of size 2.  Thus $(N)\cO_K = \ga \gb^2$ where $\ga$ and $\gb$ are square-free, relatively prime ideals of $\cO_K$ of absolute norms $N^4$ and $N^3$ respectively. 

Let $[W]$ be a $G_a$-class with $a$ in $\{4,5,9\}$, $L = \Q(W)$ and $\rho_W\!: \, \Gal(L/\Q) \xrightarrow{\sim} G_a$.  Then $H = \Gal(L/K)$ is the inverse image under $\pi\!: G_a \twoheadrightarrow \cS_5$ of $\Gal(F/K)$.  Choose  $v \vert N$ so that if $\sigma_v$ generates $\cI_v(L/\Q)$, then $\pi(\rho_W(\sigma_v)) = (12)$.  By assumption $g = \rho_W(\sigma_v)$ is very good in $G_a$.  Magma shows that among the subgroups of index 2 in $H$,  exactly one, say $J,$ has the property that the action of $G_a$ on $G_a/J$ is faithful and $g$ has exactly 8 fixed points in this action.   Hence $K' = L^J$ is a stem field for $L$ and in view of the factorization of $(N)\cO_K$, no prime over $N$ ramifies in $K'/K$.  If $v' \vert N$ is any other choice such that $\pi(\rho_W(\sigma_{v'})) =  (12)$, then $\sigma_{v'}$ is conjugate to $\sigma_v$ in $H$ and therefore gives the same $J$, so also the same $K'$.  Since $\gf_\gp(K'/K) = \gf_\gp(L/F) \le 4$ by definition of a $G_a$-class, $K'$ is contained in $\Omega_K$.  

Conversely, let $K'$ be a subfield of $\Omega_K$ quadratic over $K$, $L$ the Galois closure of $K'/\Q$, $G = \Gal(L/\Q)$, $H = \Gal(L/K)$ and $J = \Gal(L/K')$.  Then $L$ properly contains $F$, since each quadratic extension of $K$ in $F$ ramifies at some prime over $N$.   By Proposition \ref{LProperties}(ii), $\sigma_v$ is a good involution in $G$.  Since no prime over $N$ ramifies in $K'/K$, the action of $\sigma_v$ on $G/J$ has eight fixed points.   The following group-theoretic properties of $G$ have been established:
\begin{enumerate}[\hspace{3 pt} i)]
\item There is  a surjection $\pi\!: G \to \cS_5$ whose kernel  has exponent 2 and is the radical of $G$.  \vspace{2 pt}

\item The abelianization of $G$ has order 2.  \vspace{2 pt}

\item If $H$ is the inverse image under $\pi$ of the centralizer of a transposition in $\cS_5$, then there is a subgroup $J$ of index 2 in $H$ such that the action of $G$ on $G/J$ is faithful.    \vspace{2 pt} 
\item There is a good involution $g$ in $G$ whose action on $G/J$ has 8 fixed points. 
\end{enumerate}
We have (i) since the radical of $\Gal(L/\Q)$ is $\Gal(L/F)$ and (ii) by Lemma \ref{LProperties}(i).  

In the Magma database of 1117 transitive groups of degree 20 only three satisfy (i)--(iv), namely $G_a$ with $a$ in $\{4,5,9\}$.   Furthermore, if $J$ is the stabilizer in $\cS_{20}$ of any letter, then there is a unique conjugacy class of good involutions $g$ in $G$ such that $g$ acts on $G/J$ with exactly 8 fixed points. 
By applying this construction to $G = \Gal(L/\Q)$, there is an isomorphism $\rho\!:  \Gal(L/\Q) \to G_a$   such that $\rho(\sigma_v)$ is conjugate to $g$ and has 8 fixed points when acting on $G_a/\rho(J)$.
Computation now shows the  following.  If $a = 4$, then $\rho(\sigma_v)$ is conjugate to $d(r)$.  If $a$ is in $\{5,9\}$, then  $\rho(\sigma_v)$ is conjugate to $d(r)$ or $d(r) \epsilon(r)$ where $\epsilon$ is the automorphism of \eqref{EpRep}.  In the latter case,  replace $\rho$ by $\epsilon \circ \rho$. 
If $W$ is the associated Galois module, then its class is a $G_a$-class. Because any automorphism of $G_a$ preserving the conjugacy class of $d(r)$ is conjugation by an element of $\cP$, the class $[W]$ is unique.  
\end{proof}

Unless otherwise stated, $[W]$ now denotes a $G_a$-class and $L = \Q(W)$. Thus $W$ represents a class in $\Ext^1_{R'}(\cE,\cE)$, where $R' = \Z[\frac{1}{2N}]$.  By the Mayer-Vietoris sequence \eqref{MV}, $W$ prolongs to a group scheme $\cW$ over $R = \Z[\frac{1}{N}]$ if and only if the image of $[W]$ in $\Ext^1_{\Q_2}(\cE,\cE)$ agrees with that of a class from $\Ext^1_{\Z_2}(\cE,\cE)$.  If so, the other conditions in Definition \ref{GaClass} guarantee that $[\cW]$ is in $\Ext^1_{\un{E}}(\cE,\cE)$.  Recall that $\gh = \Hom_{\F_2}(E,E)$ and let $\psi\!: \, G_\Q \to \gh$ represent the class in $H^1(G_\Q,\gh)$ associated to $[W]$, as in \eqref{RepFromPsi}.   Recall that at $\gp \vert 2$, the decomposition group $\cD_\gp(F/\Q)$ is isomorphic to $\Delta = \langle s,t \rangle$ .  

\begin{lem} \label{GlobalAt2}  
As a $\Delta$-module, $\cD_\gp(L/F)$ is isomorphic to $E^b$ with $b \le 2$.
\end{lem}

\begin{proof}
We may  assume $\cD_\gp(L/F) \neq 1$.  Computation shows that $G_a$ contains no subgroup of order {\em and} exponent $40$ whose projection to $\cS_5$ has order $20$.  Conclude by using Proposition \ref{Max2} and its Corollary \ref{classfield}.
\end{proof}

\begin{Rem} \label{Compat}
Let $[\psi]$ in $H^1(G_\Q,\gh)$ correspond to the $G_a$-class $[W]$ and write $\psi_{\vert \cD_\gp}$ for the restriction to the decomposition group $\cD_\gp$ in $G_\Q$ at a fixed place $\gp$ over $2$.  The classes $[\cW_{\bfs}]$ in $\Ext^1_{\Z_2}(\cE,\cE)$ are classified by  their Honda parameters $\bfs$ in $(\F_2)^5$.  Let $[\psi_{\bfs}]$ in $H^1(G_{\Q_2},\gh)$ correspond to $[W_{\bfs}]$.   Then $[W]$ is compatible with $[W_\bfs]$ if and only if:
\begin{equation} \label{compateq}
[\psi_{\vert \cD_\gp}] = [\psi_{\bfs}] \text{ in } H^1(G_{\Q_2},\gh) \text{ for some Honda parameter } \bfs.
\end{equation}
Let $F_\gp$ be the completion of $F$ at $\gp$ and $T$ the maximal elementary 2-extension of $F_\gp$ having conductor exponent 4.  By Proposition \ref{CondExpP}, $\Q_2(W_{\bfs})$ is contained in $T$, while the completion $L_\gp$ is contained in $T$ by definition of a $G_a$-class.  In the diagram below, inflation is injective   and restriction is injective by \eqref{ToToverF}: 
 \begin{equation} 
\begin{CD} 
&  & H^1(\cD_\gp, \gh) \\
  &   &   @VV \inf V         \\
0 @>>> H^1(\Gal(T/\Q_2),\gh) @>\res>> \Hom_{\F_2[\Delta]}(\Gal(T/F_\gp,\gh) \, .
\end{CD}  \vspace{2 pt}
\end{equation}
Hence, it suffices to compare the image $\chi$ of $[\psi_{\vert \cD_\gp}]$ with the image $\chi_{\bfs}$ of $[\psi_{\bfs}]$ in $\Hom_{\F_2[\Delta]}(\Gal(T/F_\gp),\gh)$. Note that the values of $\chi$ and $\chi_{\bfs}$ are corners in $\gh$.  See Proposition \ref{Max2} for specific generators $g_0,g_1,g_2$ of $\Gamma$ as an $\F_2[\Delta]$-module.   In particular, $\chi(g_0) = 0$ by Lemmas \ref{GlobalAt2} and \ref{chiLem}(iii).   Thus $W$ prolongs to a group scheme over $R = \Z[\frac{1}{N}]$ exactly if there is a Honda parameter $\bfs$ in Proposition \ref{dim48} satisfying $\chi(g_j) = \chi_{\bfs}(g_j)$ for $j =1,2$.   
\end{Rem}

\begin{lem}  \label{ProCon}
Let $[W]$ be a $G_a$-class and $L = \Q(W)$. 
\begin{enumerate}[{\rm i)}]
\item If $\gf_\gp(L/F) \le 2$ for $\gp \vert 2$, then $W$ prolongs to a group scheme $\cW$ over $R$.  \vspace{2 pt}
\item If $a \in \{4, 5, 11\}$ and  $W$ prolongs to a group scheme over $R$, then $\gf_\gp(L/F) \le 2$.\end{enumerate}
\end{lem}

\begin{proof}
Refer to Remark \ref{Compat} for notation.   In item (i), we have $\chi(g_2) = 0$ by Lemma \ref{chiLem}(ii).  To match $\chi$ with $\chi_{\bfs}$ for some local Honda parameter $\bfs$, we therefore consider $\bfs$ in Proposition \ref{dim48}(i), also allowing $\bfs = 0$.  As $\bfs$ varies,  $\chi_{\bfs}(g_1)$ ranges over all possible corners of $\gh$ and we can find a unique $\bfs$ such that $\chi_{\bfs}(g_1) = \chi(g_1)$.  Hence $W$ prolongs to a group scheme $\cW$ over $R$. 

In item (ii), $G_a$ does not contain $\gamma_9$ by \eqref{CorGa}.  Then $\chi(g_2) = 0$, to match $\chi_{\bfs}(g_2)$ for some Honda parameter $\bfs$ in Proposition \ref{dim48}.  Hence $\gf_\gp(L/F) \le 2$.
\end{proof}

\begin{Def}\label{amiable}
Let $K$ be a pair-resolvent of $F$ and $\Omega_K$ the maximal elementary 2-extension of $K$ unramified outside $\{2,\infty\}$ such that $\gf_\gp(\Omega_K/K) \le 4$ for $\gp\, \vert\, 2$.  We say $F$ is {\em amiable} if either:  i) $\Omega_K = K$ \, or \, ii) $\fdeg{\Omega_K}{K}=2$ and $\gf_\gp(\Omega_K/K) =4.$
\end{Def}

\begin{Rem}  \label{amRem} 
For $F$ to be amiable, all the following conditions are necessary: (i) The narrow class number of $K$ is odd.  (ii) If $a \in (1+ \gp^9) K_\gp^{\times 2}$, then $a \in K^{\times 2}$, since $f_\gp(K(\sqrt{a})/K) \le 2$ by Lemma \ref{cond2}. (iii) $K$ is not totally real; otherwise $\rk U_K/U_K^2 = 10$, but $\rk U_\gp/(1+ \gp^9) U_\gp^2 = 8$.  
\end{Rem}

\begin{prop}\label{ext2} 
Let $\cE$ be the group scheme introduced at the beginning of this section. Then $\Ext^1_{[2],\un{E}}(\cE,\cE)=0$ if and only if  $F=\Q(E)$ is amiable.
\end{prop}

\begin{proof} 
Suppose that $F$ is amiable and let $[\cW]$ be a non-trivial class in $\Ext^1_{[2],\un{E}}(\cE,\cE)$.  By Corollary \ref{LPropCor}, $[W]$ is $G_a$-class with $a \ne 0$.  If $a=11$, then $\gf_\gp(L/F) \le 2$ by Lemma \ref{ProCon}(ii).  By diagram \eqref{GaDiagram} and Lemma \ref{ab}(ii), there is a $G_5$-class $[W']$ with $L' = \Q(W')$ contained in $L$. Lemma \ref{MakeGaFields} provides a quadratic extension $K'$ of $K$ contained in $\Omega_K$ with $\gf_\gp(K'/K) = \gf_\gp(L'/F) \le \gf_\gp(L/F) \le 2$, contradicting the amiability of $F$.   The same argument applies when $a = 4$ or $5$.  If $a = 15$ or 9, then $[W]$ gives rise to both a $G_4$-class and a $G_5$-class.  Then Lemma \ref{MakeGaFields} provides two distinct quadratic extensions of $K$ contained in $\Omega_K$, again contradicting the amiability of $F$. 

Suppose that $F$ is not amiable.  Assume first that $\fdeg{\Omega_K}{K} = 2$ and let $[W]$ be the $G_a$-class corresponding to $\Omega_K/K$ by Proposition \ref{MakeGaFields}.  By amiability, $\gf_\gp(\Omega_K/K) \le 2$ and so $\gf_\gp(L/F) = \gf_\gp(\Omega_K/K) \le 2$.  Then Lemma \ref{ProCon}(i) implies that $W$ prolongs to a non-trivial class in $\Ext^1_{[2],\un{E}}(\cE,\cE)$.  Next, assume that there is a $G_a$-class $[W]$ with $L = \Q(W)$ and a $G_{a'}$-class $[W']$ with $L' = \Q(W')$, coming from distinct quadratic extensions of $K$ in $\Omega_K$ and satisfying $a, a' \in \{4,5,9\}$.  Since a $G_9$-class gives rise to a $G_4$-class and a $G_5$-class, we need only consider the pairs $(a,a')$ in $\{(4,4), (5,5), (4,5)\}$.  In the notation of Remark \ref{Compat}, let $\chi$ and $\chi'$ in $ \Hom_{\F_2[\Delta]}(\Gal(T/F_\gp),\gh)$ belong to $W$ and $W'$ respectively.   Then the Baer sum $W'' = W + W'$ represents a $G_b$-class by Lemma \ref{ab} and  $\chi'' = \chi + \chi'$ belongs to $W''$.  By Lemma \ref{ProCon}(i) and Lemma \ref{condK'L}, we may assume that  $\gf_\gp(L/F) = \gf_\gp(L'/F) = 4$ and so $\chi(g_2)$ and $\chi'(g_2)$ are non-trivial, by Proposition \ref{chiLem}(ii).  In all these cases, only one non-trivial corner is available in \eqref{CorGa}, namely $\chi(g_2) = \gamma_a$ and $\chi'(g_2) = \gamma_{a'}$.  If $a = a' = 4$ or 5, then $\chi''(g_2) = 0$ and so $\gf(L''/F) \le 2$.  Thus $W''$ prolongs to a group scheme over $\Z[\frac{1}{N}]$.  If $(a,a') = (4,5)$, then $\chi''(g_2) = \gamma_4 + \gamma_5 = \gamma_9$, so $\chi''$ is compatible with $\chi_{\bfs}$ for some $\bfs$ in Proposition \ref{dim48}(i) or (ii) and the corresponding group scheme exists.
\end{proof}

\begin{theo}  \label{main}
Let $A$ be a favorable abelian surface of prime conductor $N$ such that $F=\Q(A[2])$ is amiable. If $B$ is a semistable abelian variety of dimension $2d$ and conductor $N^d,$ with $B[2]$  filtered by $A[2]$, then $B$ is isogenous to $A^d.$ 
\end{theo} 

\begin{proof} 
By Proposition \ref{Fproperties}, $\cE =  A[2]$ satisfies the conditions in Definition \ref{ourcat} for a $\Sigma$-category $\un{E}$ with $\Sigma = \{\cE\}$.  Then Theorem \ref{mypdiv} applies, since $\Ext_{[2],\un{E}}(\cE,\cE)=0$ by Proposition \ref{ext2} and $\End(A)=\Z$ because $A$ has prime conductor \cite{BK4}.
\end{proof}

\subsection{Elliptic curves of prime conductor, supersingular at 2} 
We briefly note how Theorem \ref{mypdiv} applies to elliptic curves.  Let $A$ be an elliptic curve of prime conductor $N$ with supersingular reduction at 2 and $\cE= A[2].$ Then $F=\Q(E)$ is an $\cS_3$-extension and $E$ is an irreducible Galois module even locally over $\Q_2$.  The only two irreducible $\F_2[\cS_3]$ modules are the trivial one and $E.$ 

\begin{prop}\label{ell} 
Let $K$ be a cubic subfield of $F=\Q(E)$ and let $\gp$ be the prime in $K$ above $2.$ A necessary and sufficient condition for $\Ext^1_{[2],\un{E}}(\cE,\cE)=0$ is that there be no quadratic extension of $K$ of dividing conductor $\gp^2\!\cdot\!\infty.$ 
\end{prop}

\begin{proof}
Only two subgroups of the parabolic group $\cP_{E,E}$ admit good involutions.  One is  isomorphic to $\cS_3$ and corresponds to the split extension of $\cE$ by itself because $H^1(\cS_3,\End(E))=0$ while the second is isomorphic to $\cS_4.$  If $M$ is the field of points of an extension of $\cE$ by $\cE$ annihilated by 2 and $\Gal(M/\Q) \simeq \cS_4$, then $M$ is the Galois closure of a quadratic extension of $K$ unramified at primes over $p$.   The bound for the local conductor over 2 is  given in \cite[Proposition 6.4]{Sch1} and Theorem \ref{mypdiv} applies.  A related proof is in \cite{Sch2} for $A=J_0(N)$ with $N=11$ and $19.$
\end{proof}

In the Cremona Database, we find 2037 isogeny classes of elliptic curves supersingular at 2 and of prime conductor $N< 350000.$ From the Brumer-McGuinness Database \cite{BM}, we extract an additional 2422 isogeny classes for a total of 4459 such classes with $N \le 10^8.$   Applying the Proposition above, we find 847 elliptic curves $A$ to which Theorem \ref{mypdiv} applies. 

Let $A_1$ and $A_2$ be elliptic curves of prime conductor $N$ with each  $\cE_i=A_i[2]$ biconnected over $\Z_2$ and satisfying $\Ext^1_{[2],\un{E}}(\cE_i,\cE_i)=0$. Suppose that the cubic subfields $K_i$ of $\Q(E_i)$ are non-isomorphic.  Then $2 \cO_{K_1K_2}$ has the prime factorization $(\gp_1 \gp_2 \gp_3)^3$.  If $K_1 K_2$ admits no quadratic extension of conductor dividing $(\gp_1\gp_2\gp_3)^2 \infty$, then $\Ext^1_{\un{E}}(\cE_1,\cE_2)=0$.  We found 42 conductors $N$ with multiple $A_i$ to which our results apply.  

As  an entertaining example, Cremona's Database lists four elliptic curves of conductor 307, with $A_1=307A1$, $A_2=307C1$ and $A_3=307D1$ supersingular at 2. Their 2-division fields  correspond to the three subfields of the ray class field of $k=\Q(\sqrt{-307})$ of modulus $2 \hspace{1 pt} O_k.$

Theorem \ref{mypdiv} implies the following.  Let $B$ be a semistable  abelian variety, good outside $N = 307$, with $B[2]^{ss}=A_1[2]^{n_1}\oplus A_2[2]^{n_2}\oplus A_3[2]^{n_3}$ for some $n_i$. Then $B$ is isogenous to $A_1^{n_1}\times A_2^{n_2}\times A_3^{n_3}$. Note that we need not impose the conductor $f_N(B) = \sum n_i f_N(A_i) = \sum n_i$, thanks to Remark  \ref{E2E3}.

\appendix \numberwithin{equation}{section}

\section{A cohomology computation in the old style}  \label{coda}  Let $T = \Lambda[G]$ be the group ring of a finite group $G$ over a discrete valuation ring $\Lambda$ with prime element $\pi$ and finite residue field $k$ of characteristic $p$.  We consider a cocycle approach to $\Ext_{\Lambda[G]}^1(E,E)$.   Let $V$ and $W$ be finitely generated $T$-modules such that  $\pi V=\pi W=0.$ A {\it symmetric cocycle} is a function $f:V\times V\to W$ satisfying
$$ f(v_1,v_2)=f(v_2,v_1) \text{ and } f(v_1,v_2)+ f(v_1+v_2,v_3)= f(v_1,v_2+v_3)+ f(v_2,v_3)$$ for  $v$'s in $V,$ as in \cite[Theorem\! 7.1]{EiMa}. 
Coboundaries are symmetric cocycles such that $$f(v_1,v_2)=g(v_1)+g(v_2)-g(v_1+v_2)$$ for some function $g:V\to W$. The  symmetric cocycle $f$ is {\it enhanced} if there is a function $h:T\times V\to W$ satisfying the following for $v$'s in $V$ and $r,s$ in $T:$
\begin{enumerate}[{\rm i)}]\item $rf(v_1,v_2)-f(rv_1,rv_2)= h(r,v_1)+h(r,v_2)-h(r,v_1+v_2)$;
\item $h(rs,v)=rh(s,v)+h(r,sv)$;
\item $f(rv,sv)=h(r+s,v)-h(r,v)-h(s,v)$.
\end{enumerate}
The cohomology classes of enhanced cocycles form a $k$-vector space $\mathcal{D}(V,W).$

\begin{lem} \label{sc1} 
The functor  from $T$-modules to abelian groups induces an exact sequence
$$0\to\Ext^1_{[\pi],T}(V,W)\to \Ext^1_T(V,W)=\mathcal{D}(V,W)\to \Hom_T(V,W),$$
where $\Ext^1_{[\pi],T}(V,W)$ consists of classes of extensions annihilated by $\pi$. 
\end{lem}

\begin{proof}  Let $0 \to W \xrightarrow{i} M \xrightarrow{j} V \to 0$
 be an exact sequence of $T$-modules with $\pi V=\pi W =0.$ Let $\sigma\!: \, V\to M$ be a section of $j$ such that $\sigma(0)=0.$ The associated cocycle is defined by $f(v_1,v_2)=\sigma(v_1)+\sigma(v_2)- \sigma(v_1+v_2) .$ If  $r$ is in $T$, then  $h(r,v)=r\sigma(v)-\sigma(rv)$ turns $f$ into an enhanced cocycle. For the   converse, give $W\times V$ the structure of a $T$-module by setting
$$
(w_1,v_1)+(w_2,v_2)=(w_1+w_2+f(v_1,v_2),v_1+v_2), \hspace{5 pt} r(w,v)=(rw+h(r,v),rv).
$$
Hence $\Ext^1_T(V,W)=\mathcal{D}(V,W)$.  Given $f$ as above, let $\iota:V\to W$ be defined by $\iota(a)=h(\pi,a).$ Since $\pi(w,v)=(\iota(v),0)$ and $\pi$ is in the center of $T$, we conclude that $\iota$ is a $T$-homomorphism and that the  sequence is exact.  
\end{proof} 

Using the Lemma, we give a refined variant  of \cite[Lemma 2.1]{Sch3}. 
Let $F$ be a number field and  $R$ its ring of $S$-integers for a finite set $S$ of primes.

\begin{prop} \label{Ext}
Let $\mathcal{V}$ and $\mathcal{W}$  be finite flat $\Lambda$-module schemes over $R$ killed by $\pi$, with associated Galois modules $V$ and $W.$ Let $\Ext^1_{[\pi],R}(\mathcal{V},\mathcal{W})$ denote the subgroup of  $\Ext^1_{R}(\mathcal{V},\,\mathcal{W})$ consisting of those extensions killed by $\pi.$  Then there is a natural exact sequence 
$$
0\to \Ext^1_{[\pi],R}(\mathcal{V},\,\mathcal{W}) \to \Ext^1_{R}(\mathcal{V},\,\mathcal{W})\to \Hom_{\Gal}(V,W).
$$
If $V$ is absolutely irreducible over $k,$ then $\End_{\Gal}(V)=k.$
\end{prop}

\begin{proof} Apply  Lemma \ref{sc1} with $G$ the Galois group of a suitable finite extension of $F$. 
Then the  passage from    Galois modules to the associated  group schemes  is as in Schoof and so is left to the reader.
\end{proof} 

\section{Parabolic subgroups and an obstreperous cocycle}
For any group $G$, consider representations $\rho_{E_i}$ afforded by $\F_p[G]$-modules $E_i$ for $i = 1,2$.  If $g$ is in $G$ and $\delta_i = \rho_{E_i}(g)$, then $g$ acts on $m$ in $\gh = \Hom_{\F_p}(E_2,E_1)$ by $g(m) = \delta_1 m \delta_2^{-1}$.  In the category of $\F_p[G]$-modules, the extension classes of $E_2$ by $E_1$ under Baer sum form a group isomorphic to $H^1(G,\gh)$.  The exact sequence of $\F_p[G]$-modules $0 \to E_1 \to W \to E_2 \to 0$ gives rise to a cocycle $\psi\!: G \to \gh$ such that 
\begin{equation} \label{RepFromPsi}
\rho_W(g) = \left[\begin{smallmatrix} \delta_1  \, & \,  \psi(g) \, \delta_2 \\ 0 \, & \, \delta_2 \end{smallmatrix} \right]
\end{equation}
and the class $[W]$ in $\Ext^1_{\F_p[G]}(E_2,E_1)$ corresponds to that of $[\psi]$ in $H^1(G,\gh)$.  
If $N$ is a normal subgroup of $G$ contained in $\ker \rho_W$, then $[\psi]$ comes by inflation from a unique class in $H^1(G/N,\gh)$, also denoted by $[\psi]$.  

Note that $\rho_W(G)$ lies in a {\em parabolic} matrix group
\begin{equation} \label{defP}
\cP = \cP_{E_1, E_2}= \left\{g = \left[\begin{smallmatrix} \delta_1 & m \\ 0 & \delta_2 \end{smallmatrix} \right] \, \vert \, \delta_i = \rho_{E_i}(g), \, m \in \Mat_{n_1,n_2}(\F_p) \right\}
\end{equation}
with $n_i = \dim_{\F_p} E_i$.  If $H_i = \{ g \in G \, \vert \, g_{\vert E_i} = 1 \}$ and $\Delta_i = G/H_i$, then $E_i$ is a faithful $\F_p[\Delta_i]$-module.  Any normal subgroup $H$ of $G$ acting trivially on both $E_1$ and $E_2$ satisfies
$$ 
\rho_W(H) \, \subseteq \, \left\{g = \left[\begin{smallmatrix} 1 & m \\ 0 & 1 \end{smallmatrix} \right] \in \cP \, \vert \, m \in \Mat_{n_1,n_2}(\F_p) \right\}.
$$
Since $H^1(H, \gh)^{G/H}  = \Hom_{\F_p[G/H]}(H,\gh)$, the following sequence is exact:
\begin{equation} \label{InfRes}
0 \to H^1(G/H, \gh) \xrightarrow{inf} H^1(G,\gh) \xrightarrow{res} \Hom_{\F_p[G/H]}(H,\gh).
\end{equation}

\begin{Rem} \label{obstrep}
Let $E_1$ and $E_2$ above be $G_\Q$-modules, with $F = \Q(E_1,E_2)$ and $\Delta = \Gal(F/\Q)$.  If the extension $W = W_\psi$ belongs to a cocyle $\psi\!: \Delta \to \gh$ whose class in $H^1(\Delta,\gh)$ is not trivial, then $\Q(W) = F$, even though $W$ does not split as a $\Delta$-module.

For example, let $p=2$ and $E = E_1 = E_2$, with $\dim_{\F_2}(E) = 2n$, so that $\gh$ is isomorphic to $\Mat_{2n}(\F_2)$.  As in \cite[Remark 2.6]{BK3}, equip $E$ with the irreducible symplectic representation of $\Delta \subset \SP_{2n}(\F_2)$ isomorphic to $\cS_m$, with transvections corresponding to transpositions and $m = 2n+1$ or $2n+2$.  If $n=1$, then $\Delta=\cS_3$ and so $H^1(\Delta, \gh) = 0$.  If $n \ge 2$, there is a non-trivial class $[\psi]$ in $H^1(\Delta,\gh)$ such that $\psi(g) = \sign^+(g)I_{2n}$, where $\sign^+$ is the sign of the permutation $g$ with values in $\F_2$.   This situation can occur when $E$ is the kernel of multiplication by 2 on the Jacobian of a hyperelliptic curve of genus at least 2.

Suppose further that $E$ has prime conductor $N$ and let $\sigma_v$ generate inertia in $F/\Q$ at $v \vert N$.  Then $\sigma_v$ is a transposition in $\cS_m$, so $\psi(\sigma_v) = I_{2n}$.  It follows from \eqref{RepFromPsi} that $\rk \rho_W(\sigma_v-1) = 2n$.  Since the minimality assumption {\bf E3} on our category $\un{E}$ requires that this rank be 2, the extension $W$ is not acceptable when $n > 1$.   However, $W$ does prolong to a group scheme over $\Z[\frac{1}{N}]$ satisfying {\bf E1} and {\bf E2} under the hypotheses in Lemma \ref{gh}, since $H^1(\cD_\gp,\gh) = 0$ for the decomposition group $\cD_\gp$ at $\gp \vert 2$ in $F/\Q$.

\end{Rem}

\section{Some technical lemmas on local conductors} \label{condapp}
Let $K$ be a finite extension of $\Q_p$ with uniformizer $\pi_K$, ring of integers $\cO_K$ and absolute ramification index $e_K = \ord_{\pi_K}(p)$.  Set   
\begin{equation} \label{nUnits}
U_K^{(n)} = \{ u \in \cO_K^\times \, | \, \ord_{\pi_K}(u-1) \ge n \}.
\end{equation} 
See \cite[IV]{Ser1} for basic information about ramification groups and conductors. Let $L/K$ be a finite Galois extension.  The {\em index} of elements $g$ in $G = \Gal(L/K)$ is given by $i_{L/K}(g) = \ord_{\pi_L}(g(\theta) - \theta)$ for any choice of $\theta$ in $\cO_L$ such that $\cO_L = \cO_K[\theta]$.  Then $\ord_{\pi_L}(g (a) - a) \ge i_{L/K}(g)$ for all $a$ in $\cO_K$.   In Serre's {\em lower numbering} on ramification groups, $G_j = \{ g \in G  \, \vert \, i_{L/K}(g) \ge j+1\}$.  Thus $G_{-1} = G$, $G_0$ is the inertia group, its fixed field is the maximal unramified extension of $K$ inside $L$ and the $p$-Sylow subgroup $G_1$ is the wild ramification subgroup of $G$.  For $g$ in $G_0$, we have $i_{L/K}(g) = \ord_{\pi_L}(g(\pi_L) - \pi_L)$.  The Herbrand function is defined by 
\begin{equation} \label{HerbDef}
\varphi_{L/K}(x) = \int_0^x \frac{ds}{\fdeg{G_0}{G_s}}
\end{equation}
In Serre's {\em upper numbering}, $G^m = G_n$ with $m = \varphi_{L/K}(n)$.  

\begin{Not}
Let $c_{L/K} = \max\{j \, \vert \, G_j \ne 1\}$ and let $m_{L/K} = \varphi_{L/K}(c_{L/K})$.  Thus $G^{m_{L/K}} \ne 1$ but $G^{m_{L/K}+\epsilon} = 1$ for all $\epsilon > 0$.  When $L/K$ is abelian, the conductor exponent $\gf(L/K)$ is the smallest integer $n \ge 0$ such that $U_K^{(n)}$ is contained in the norm group $N_{L/K}(L^\times)$.   
\end{Not}

We have $\gf(L/K) =m_{L/K}+1$ by \cite[XV, \S2]{Ser1}, with $c_{L/K} = m_{L/K} = -1$ and $\gf(L/K) = 0$ when $L/K$ is unramified.  If $M/K$ is a Galois extension and the intermediate field $L$ also is Galois over $K$, then $m_{L/K} \le m_{M/K}$ because $\Gal(M/K)^\alpha \xrightarrow{\rm res} \Gal(L/K)^\alpha$ is surjective for all $\alpha$.  Translation by an unramified extension of the base does not affect the conductor, as we next recall.

\begin{lem}  \label{LF}
If $F/K$ is unramified, then $m_{LF/F} = m_{L/K}$.  If, in addition, $L/K$ is abelian, then $\gf(LF/F) = \gf(L/K)$.
\end{lem}

\begin{proof}
The restriction map $\Gal(LF/F) \xrightarrow{\rm res} \Gal(L/ L \cap F)$ is an isomorphism.  Since $F/K$ is unramified, $\pi_L$ also is a prime element of $LF$.  For all $s \ge 0$, it follows from the definition of the lower numbering that restriction induces an isomorphism $\Gal(LF/F)_s \xrightarrow{\sim} \Gal(L/ L \cap F)_s = \Gal(L/K)_s$   Thus the Herbrand functions of $LF/F$ and $L/K$ agree and the rest is clear.
\end{proof}

\begin{prop}  \label{cond1}
Let $L = K(t)$ be Galois over $K$, with $\ord_{\pi_L}(t) = -n$ prime to $p$ and negative.  If $g(t) - t$ is a unit for all $g \ne 1$ in $G_0$, then $G_0$ is an elementary abelian $p$-group and $\gf(L/K) = i_{L/K}(g) = n+1$.
\end{prop}

\begin{proof}
By assumption, non-trivial elements $g$ of $G_0$ satisfy $g(t) = t + u$ with $u$ a unit in $\cO_L$ and $g(u) \equiv u \pmod{\pi_L}$.  If $g$ has order $d$, then   
$$
t = g^d(t) = t + u + g(u) + \dots + g^{d-1}(u) \equiv t + du \pmod{\pi_L},
$$
so $p \vert d$.  Hence $G_0 = G_1$ is a $p$-group and so $i = i_{L/K}(g) \ge 2$.   Furthermore, $\ord_\pi(g (a) - a) \ge i$ for all $a$ in $\cO_L$.

Set $\pi = \pi_L$, $\theta =  1/t = \alpha \pi^n$  and $g(\pi)-\pi = \beta \pi^i$, where $\alpha$ and $\beta$ are units in $ \cO_L$.  We have the following congruences modulo $\pi^{n+i} \cO_L$:  
$$
\begin{array}{l l l l l}
g(\theta)-\theta &=& (g-1)(\alpha \pi^n) &=& \alpha \, (g-1)(\pi^n) + g(\pi^n) \, (g-1)(\alpha) \\
 & &  &\equiv& \alpha \, (g-1)(\pi^n)  \\
  & & &\equiv& \alpha \, ((\pi + \beta \, \pi^i)^n - \pi^n)  \\
  & & &\equiv& \alpha \beta n \, \pi^{n-1+i} 
\end{array}
$$
and therefore $\ord_\pi(g(\theta)-\theta) = n-1+i$.  Explicitly, 
$$
g(\theta) - \theta = \frac{t-g(t)}{t \, g(t)} = -\frac{u}{t \, g(t)} = - u \cdot  \theta \, g(\theta),
$$
so $\ord_\pi(g(\theta)-\theta) = 2n$.  Hence $i = n+1$ and the lower ramification sequence has only one gap:
$
G_0 = G_n \, \supsetneq \, G_{n + 1} = \{1\}. 
$
By ramification theory, $G_n$ is an elementary abelian $p$-group and we have $\gf(L/K) = \varphi_{L/K}(n)+1 = n + 1$.
\end{proof}

Next, we recall the conductors of Kummer extensions of degree $p$.  

\begin{lem} \label{cond2}
Let $K$ contain $\Mu_p$ and $L =  K(\kappa^{1/p})$ with $\kappa \in K^\times$.  Then 
$$
\gf(L/K) = \frac{pe_K}{p-1} + 1  \quad \text{if } \ord_{\pi_K}(\kappa) \not \equiv 0 \bmod{p}
$$
and this is maximal for cyclic extensions of $K$ of degree $p$.  If $\ord_{\pi_K}(\kappa-1) = n$ with $1 \le n <  pe_K/(p-1)$ and $n \not\equiv 0 \bmod{p}$, then  
$ 
\displaystyle{\gf(L/K) = \frac{pe_K}{p-1} - n + 1}. 
$
\end{lem}

\begin{proof}
In the first case, assume without loss of generality that $\ord_{\pi_K}(\kappa) = 1$, so $\theta = \kappa^{1/p}$ is a prime element for $L$.  If $g \ne 1$ in $\Gal(L/K)$, then $g(\theta) - \theta = (\zeta-1) \pi_L$ for some a $p$-th root of unity $\zeta$ and the conductor follows by definition.  

In the second case, set $\kappa = 1 + u\pi_K^n$ with $u$ in $U_K$ and $\theta = \kappa^{1/p} - 1$.  Then $g(\theta) = \zeta \kappa^{1/p}-1 = \theta + (\zeta-1)\kappa^{1/p}$, where $\theta$ satisfies
$
x^p + \sum_{j=1}^{p-1} \bino{p}{j} x^j =  u\pi_K^n  .
$
Let $t = \theta/(\zeta-1)$, to find that $g(t) - t = \kappa^{1/p}$ is a unit in $L$ and $t$ satisfies 
\begin{equation} \label{tPoly}
z^p + \sum_{j=1}^{p-1} a_j z^j = \frac{u\pi_K^n}{(\zeta-1)^p} \quad \text{with} \quad a_j = \bino{p}{j} (\zeta-1)^{j-p}.
\end{equation}
For $1 \le j \le p-1$, we have 
\begin{eqnarray*}
\ord_{\pi_K}(a_j) = e_K - (p-j)\frac{e_K}{p-1} = (j-1) \frac{e_K}{p-1} \ge 0.
\end{eqnarray*}
Put $z = t$ in \eqref{tPoly} and compare ordinals on both sides, using $p \nmid n$, to see that  $L/K$ is totally ramified of degree $p$ and
$$ 
\ord_L(t^p) = n \ord_{\pi_L}(\pi_K)  -  p\ord_{\pi_L}(\zeta-1) = np - p \frac{pe_K}{p-1}.
$$
Thus $\ord_p(t) = n - \frac{pe_K}{p-1}$ and $\gf(L/K)$ can be found by using Proposition \ref{cond1}.
\end{proof}

\begin{Rem}
Since the choice of $\kappa$ can be changed by multiplying by a suitable element of $K^{\times p}$, the only remaining cases are $n \ge \frac{pe_K}{p-1}$.  If equality holds, then \eqref{tPoly} gives an integral polynomial satisfied by $t$ whose reduction modulo $\pi_K$ has the form $z^p + \ov{a}_1 z^{p-1} - \ov{b}$ with $b = u\pi^n(\zeta-1)^{-p}$.   Since $a_1$ and $b$ are unit in $\cO_K$, this polynomial is separable and $L/K$ is unramified, but  possibly split.  If $n > \frac{pe_K}{p-1}$, then $\kappa$ is in $K^{\times p}$ and $L = K$.
\end{Rem}

\begin{lem} \label{cond3}
Let $L_i/K$ be Galois and let $m_i = m_{L_i/K}$ be the upper numbering of the last non-trivial ramification subgroup of $\Gal(L_i/K)$.  If $M = L_1L_2$, then $m_{M/K} = \max\{m_1,m_2\}$ and if $L$ is a subfield of $M$ with $L/K$ abelian, then $\gf(L/K) \le m_{M/K} +1$.
\end{lem}

\begin{proof}
If $m = \max\{m_1,m_2\}$, then $m_{M/K} \ge m$.  But if $g$ is in $\Gal(M/K)^\alpha$ with $\alpha > m$, then $g_{\vert L_i} = 1$ for $i=1,2$, so $g = 1$.  Hence $m_{M/K} = m$.  It follows that $m_{L/K} \le m$ and therefore $\gf(L/K) \le m+1$. 
\end{proof}

\begin{lem}\label{cond5}
Assume that $F/K$ is Galois and $L/F$ is abelian.  Let $M$ be the Galois closure of $L/K$.  Then $M/F$ is abelian and $\gf(M/F) = \gf(L/F)$.
\end{lem}

\begin{proof}
Since $m_{L/F} \le m_{M/F}$, we have $\gf(L/F) \le \gf(M/F)$.  If $\tau$ is in $\Gal(M/K)$, then  $\tau(L)/F$ is abelian and $\gf(\tau(L)/F) = \gf(L/F)$.  But $M$ is the compositum of all $\tau(L)$ as $\tau$ varies. Therefore, $M/F$ is abelian and by Lemma \ref{cond3}, $\gf(M/F) \le \gf(L/F)$, giving equality.
\end{proof}

For the next lemma, refer to the following diagram:   \quad
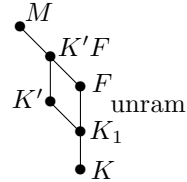
\begin{tikzpicture}[scale=1, thin, baseline=(current  bounding  box.center)]	
    \draw (0,0) -- (0,1.1);
    \draw (0,0) node {$\bullet$};   \draw (.3,0) node {$K$};
    \draw (0,.5) node {$\bullet$};     \draw (.35,.5) node {$K_1$};
    \draw (0,1.1) node {$\bullet$};   \draw (.29,1.15) node {$F$};
    \draw (.9,.85) node {unram};
    \draw (0,.5) -- (-.4,.9) -- (-.4,1.5);
    \draw (-.4,.9) node {$\bullet$};  \draw (-.7,.95) node {$K'$};
    \draw (-.4,1.5) node {$\bullet$};   \draw (.05,1.65) node {$K'F$};
    \draw (0,1.1) -- (-.8,1.9);
    \draw (-.8,1.9) node {$\bullet$};  \draw (-.55,2.1) node {$M$};
\end{tikzpicture}

\begin{lem}  \label{condK'L}
Let $F$ be the Galois closure of $K_1/K$  and assume that $F/K_1$ is unramified.  Let $K'$ be an abelian extension of $K_1$ and let $M$ be the Galois closure of $K'/K$.  Then $M$ is abelian over $F$ and $\gf(M/F) = \gf(K'/K_1)$.
\end{lem}

\begin{proof} 
The field $M$ contains $F$ because $K'$ contains $K_1$.  Moreover, $M$ is the Galois closure of $K'F/K$.  Since $K'$ is abelian over $K_1$, the extension $K'F/F$ is abelian.  By Lemma \ref{cond5}, with $L$ there equal to $K'F$ here, we find that $M/F$ is abelian and $\gf(M/F) = \gf(K'F/F)$.  By Lemma \ref{LF}, translation of the base via an unramified extension does not change the conductor, so $\gf(K'F/F) = \gf(K'/K_1)$.  Hence $\gf(M/F) = \gf(K'/K_1)$. 
\end{proof}

When $L = K(V)$, where $\cV$ is a finite flat group scheme over $\cO_K$ of exponent $p^n$, Fontaine \cite{Fon4} showed that $m_{L/K} \le e_K(n+\frac{1}{p-1})-1$.  Now consider the conductor exponent of an intermediate abelian 
extension.  

\begin{prop} \label{FontCond}
Let $L = K(V)$ and suppose that $K \subseteq F \subseteq F' \subseteq L$, with  $F'/F$ abelian and the relative ramification index $e_{F/K}$ equal to the tame ramification degree $\fdeg{G_0}{G_1}$ of $L/K$.  Then 
$
\gf(F'/F) \le e_F (n+\frac{1}{p-1})-e_{F/K}+1.
$
\end{prop}

\begin{proof} 
The fixed field $L_1$ of $H = G_1$ is the maximal subfield of $L$ tamely ramified over $K$.  Since $H_0 = G_1$ and $H_s = G_s$ for all $s > 0$, \eqref{HerbDef} gives
$$
\varphi_{L/L_1}(x) = \fdeg{G_0}{G_1} \varphi_{L/K}(x) = e_{F/K} \, \varphi_{L/K}(x) \quad \text{for all } x > 0.
$$
We may assume that $L$ properly contains $L_1$.  Using $c_{L/L_1} = c_{L/K}$, we have 
$$
m_{F'L_1/L_1} \le m_{L/L_1} = \varphi_{L/L_1}(c_{L/L_1}) = e_{F/K} \, \varphi_{L/K} (c_{L/K}) = e_{F/K} \, m_{L/K}. 
$$
But $F$ is contained in $L_1$ and $L_1/F$ is unramified.  Hence Lemma \ref{LF} shows that
$\gf(F'/F) = \gf(F'L_1/L_1) \le 1 + e_{F/K} \, m_{L/K}$.  Conclude with Fontaine's bound.
\end{proof}

\section{Some Data}  \label{DataSection}  
The  quintic field $F_0$ is  {\em amiable} if its Galois closure $F$ is amiable as in Definition \ref{amiable}, so that the uniqueness in Theorem \ref{main} applies.  To check amiability, construct the pair-resolvent field $K$ and ask  Magma, under GRH, for the 2-rank of the ray class groups  of $K$ with the desired moduli, as  in Theorem \ref{IntroThm}.  A favorable abelian surface $A$ is of {\em type} $F_0$ if  $\Q(A[2])$ is the Galois closure of $F_0$.  To find representatives for isogeny classes of abelian surfaces of prime conductor $N$, it suffices to search for Jacobians by \cite[Theorem 3.4.11]{BK4}.  If $F$ is amiable, then it is not totally real by Remark \ref{amRem}.  The Magma database of quintic fields contains 1919 favorable quintic fields that are not totally real.  Their absolute discriminants are at most $5 \!\cdot \!10^6$ and 714 of them are amiable.  We know Jacobians for only 82 of the latter, but expect conductors of abelian surfaces to be sparse among integers.  

We tabulate explicit information for favorable fields and curves with $N < 25000$ and summarize some data for $N < 10^{10}$.  In all our tables, $[a_0,a_1,a_2, \dots]$ denotes the polynomial $a_0+a_1x+a_2x^2+ \dots$, as in Magma.  

\vspace{5 pt}

\centerline{\bf Legend for Tables \ref{Fields} and \ref{Curves}}  

\vspace{2 pt}

Table \ref{Fields} gives a defining polynomial $f(x)$ for each of the 172 favorable quintic fields $F_0$ of discriminant $\pm16N$ with $N<25000.$   Table \ref{Curves} consists of 75 curves $y^2=g(x)$ whose Jacobians represent distinct known isogeny classes of favorable abelian surfaces of prime conductor $N<25000$.   If $C$ is curve number 25, 63 or 64 in that table, its leading coefficient has the form $4m^3.$  These curves exhibit {\em mild reduction} \cite[p.~1162]{BK4}, in that $C$ is bad at $p|m$ but the reduction of $J(C)$ at $p$ is the product of two elliptic curves.  

In both tables, the column marked $\epsilon$ contains an $\alpha$ if $F_0$ is amiable.   
For each field $F_0$ in Table \ref{Fields}, the column marked \#C contains one of the following:
\begin{enumerate}[$\bullet$]
\item the line number of a curve in Table \ref{Curves} such that $g$ has a root in $F_0$;
\item 0 if we can prove that no abelian surface of type $F_0$ exists by \cite{BK5};
\item P if no non-lift paramodular form of that level exists, so no such surface is expected to exist;
\item U if there is at most one isogeny class of that type, but it is unknown whether such an abelian surface actually exists; 
\item $\nu$ if $F_0$ is not amiable and we do not know whether or not any surface exists.
\end{enumerate}

\vspace{5 pt}

\centerline{\bf Legend for Tables \ref{data} and \ref{LargeN}}  

\vspace{2 pt}

We know 276109 curves, including 10360 mild curves with $3\le m\le 53,$  whose Jacobians are favorable and non-isogenous of prime conductor $N <10^{10},$ for a total of 275494 non-isomorphic fields.   Table \ref{data} summarizes the statistics.  For $0 \le j \le 9$, the $j$-th column refers to $N$ between $j \cdot 10^9$ and $(j+1) \cdot 10^9$.  The rows A, F and $\alpha$, respectively, give the number of abelian varieties, fields and amiable fields.   It is remarkable that approximately 11.8\% of the favorable fields are amiable, uniformly for each slice of size $10^9.$ For the reader's entertainment,  Table \ref{LargeN} lists the  curves we found with largest  conductors below $10^{10}$ and amiable Jacobians.

\begin{table}
\centering{
\begin{caption}{Favorable quintic fields}\label{Fields}\end{caption}
\begin{tabular}{|c|c|c|c|c||c|c|c|c|c|}
\hline
\#$F_0$&\text{$f(x)$}&$N$&$\epsilon$&\#$C$&\#$F_0$&\text{$f(x)$}&$N$&$\epsilon$&\#$C$\\
\hline 
1&[-1,-1,-2,0,1,1]&277&$\alpha$&1&47&[2,4,-4,-4,2,1]&5867&&27\\
2&[-1,1,0,0,1,1]&349&$\alpha$&2&48&[2,-4,2,-2,2,1]&6277&$\alpha$&U\\
3&[-1,3,0,-2,1,1]&461&$\alpha$&3&49&[-1,-1,-8,-4,1,1]&6317&$\alpha$&0\\
4&[1,3,2,2,1,1]&613&$\alpha$&P&50&[-3,-5,-6,2,1,1]&6373&$\alpha$&U\\
5&[1,1,2,0,1,1]&677&$\alpha$&P&51&[2,4,0,-2,2,1]&6397&$\alpha$&0\\
6&[2,2,2,2,2,1]&797&$\alpha$&4&52&[2,-2,0,-2,0,1]&6491&&28\\
7&[-2,0,0,0,2,1]&971&$\alpha$&5&53&[2,0,4,6,0,1]&6701&&0\\
8&[1,1,0,-2,1,1]&997&$\alpha$&6&54&[-2,2,4,-4,0,1]&6763&&29\\
9&[-1,-3,0,4,1,1]&1051&$\alpha$&7&55&[-1,9,-2,-6,1,1]&6907&$\alpha$&U\\
10&[2,-2,-2,0,2,1]&1061&$\alpha$&U&56&[2,6,4,0,0,1]&7013&$\alpha$&U\\
11&[1,-1,2,-2,1,1]&1109&$\alpha$&9&57&[-2,0,-4,-2,2,1]&7109&&30\\
12&[-1,3,-2,0,1,1]&1109&$\alpha$&8&58&[2,-4,-2,4,2,1]&7541&$\alpha$&U\\
13&[-2,-4,-2,2,2,1]&1277&$\alpha$&0&59&[2,-2,6,0,0,1]&7549&$\alpha$&U\\
14&[2,-4,4,-2,0,1]&1597&$\alpha$&0&60&[-3,7,2,6,1,1]&7589&$\alpha$&U\\
15&[2,-2,2,0,0,1]&1637&$\alpha$ &10&61&[6,2,-8,-4,2,1]&7723&&$\nu$\\
16&[1,-3,0,2,1,1]&1811&$\alpha$&11&62&[2,6,0,-6,0,1]&7877&&31\\
17&[-2,2,2,4,2,1]&2069&$\alpha$&U&&&7877&&32\\
18&[-2,0,2,-2,0,1]&2243&$\alpha$&12&&&7877&&33\\
19&[3,5,4,4,1,1]&2269&$\alpha$ &U&63&[11,-1,-4,-4,1,1]&7963&&$\nu$\\
20&[-3,-1,-2,2,1,1]&2341&$\alpha$&13&64&[-2,4,0,-2,2,1]&8243&$\alpha$&34\\
21&[2,4,2,2,2,1]&2557&$\alpha$&0&65&[2,4,2,2,0,1]&8581&&$\nu$\\
22&[2,4,0,-2,0,1]&2677&$\alpha$&14&66&[-1,-5,-4,6,1,1]&8803&&35\\
23&[-2,0,2,0,0,1]&2693&&15&67&[-3,13,-4,-6,1,1]&9091&$\alpha$&36\\
24&[2,4,2,0,0,1]&2909&$\alpha$&U&68&[5,7,0,0,1,1]&9781&$\alpha$&U\\
25&[6,8,8,6,2,1]&3037&$\alpha$&0&69&[7,3,-6,-4,1,1]&9803&&37\\
26&[2,-2,4,0,0,1]&3109&$\alpha$&U&70&[2,-2,4,0,2,1]&9941&$\alpha$&38\\
27&[-2,4,2,-6,0,1]&3251&$\alpha$&16&71&[7,1,2,-2,1,1]&9949&&0\\
28&[1,5,2,4,1,1]&3461&$\alpha$&U&72&[2,-8,8,0,0,1]&10037&&39\\
29&[-1,-3,-2,-2,1,1]&3499&&17&73&[1,-3,-4,-2,1,1]&10163&$\alpha$&U\\
30&[2,0,2,0,0,1]&3557&&18&74&[2,4,0,6,0,1]&10253&&0\\
31&[2,2,0,0,0,1]&3637&$\alpha$&19&75&[-2,2,2,-8,0,1]&10259&&$\nu$\\
32&[2,6,0,-4,0,1]&3701&$\alpha$ &20&76&[1,3,6,2,1,1]&10453&$\alpha$&U\\
33&[2,0,0,2,2,1]&3853&$\alpha$&0&77&[3,-7,10,-6,1,1]&10789&&40\\
34&[2,0,0,2,0,1]&3989&&21&78&[2,-2,4,-4,0,1]&10837&&41\\
35&[-2,-2,-2,2,2,1]&3989&$\alpha$&U&79&[2,2,6,4,2,1]&10853&&42\\
36&[-1,5,-4,-4,1,1]&4003&&0&80&[6,-4,0,-2,0,1]&10949&$\alpha$&43\\
37&[2,2,-2,-2,2,1]&4157&$\alpha$&22&81&[1,1,6,-6,1,1]&10957&&$\nu$\\
38&[2,-6,4,0,0,1]&4219&$\alpha$&U&82&[-3,-1,0,0,1,1]&11117&&44\\
39&[2,2,0,2,0,1]&4517&$\alpha$&23&83&[-1,-5,-6,-4,1,1]&11131&&$\nu$\\
40&[2,0,-6,-2,2,1]&5059&$\alpha$&24&84&[5,11,0,-4,1,1]&11243&$\alpha$&U\\
41&[-1,1,0,-4,1,1]&5227&&25&85&[-1,5,-6,6,1,1]&11261&&0\\
42&[2,2,2,0,0,1]&5261&$\alpha$&0&86&[-1,3,2,-4,1,1]&11579&&45\\
43&[-2,-2,2,4,2,1]&5309&$\alpha$&U&87&[-3,1,0,2,1,1]&11701&&$\nu$\\
44&[-1,-3,-6,-2,1,1]&5381&&$\nu$&88&[2,-10,14,-4,0,1]&11971&&46\\
45&[3,-1,4,6,1,1]&5437&$\alpha$&0&&&11971&&47\\
46&[-2,-4,0,2,2,1]&5651&$\alpha$&26&89&[13,11,-6,-6,1,1]&12037&&$\nu$\\
\hline
\end{tabular}}
\end{table}

\begin{table}
\centering{
\begin{tabular}{|c|c|c|c|c||c|c|c|c|c|}
\hline
\#$F_0$&\text{$f(x)$}&$N$&$\epsilon$&\#$C$&\#$F_0$&\text{$f(x)$}&$N$&$\epsilon$&\#$C$\\
\hline 
90&[3,-1,-2,0,1,1]&12109&&$\nu$&133&[4,-4,8,-2,0,1]&17341&$\alpha$&U\\
91&[3,11,0,-4,1,1]&12301&&$\nu$&134&[-2,0,4,2,0,1]&17341&&0\\
92&[2,10,6,-2,0,1]&12541&$\alpha$&U&135&[-4,4,4,0,0,1]&17389&$\alpha$&59\\
93&[10,6,-8,-4,2,1]&12757&&$\nu$&136&[3,7,6,4,1,1]&17597&$\alpha$&0\\
94&[2,2,4,2,0,1]&12781&$\alpha$&U&137&[14,24,4,-6,0,1]&17923&&$\nu$\\
95&[-3,5,-2,-4,1,1]&12781&$\alpha$&U&138&[6,-4,6,0,0,1]&18077&$\alpha$&60\\
96&[-3,-5,-10,-6,1,1]&12907&$\alpha$&U&139&[-1,-3,-8,-4,1,1]&18181&$\alpha$&0\\
97&[-3,1,-6,-6,1,1]&12923&$\alpha$&48&140&[-1,-5,-4,2,1,1]&18691&&0\\
98&[-1,-1,2,-4,1,1]&13003&$\alpha$&U&141&[1,7,2,-2,1,1]&18757&&$\nu$\\
99&[-2,2,-2,0,2,1]&13037&$\alpha$&0&142&[10,4,-8,-4,2,1]&18869&&$\nu$\\
100&[-2,4,-2,-4,2,1]&13147&$\alpha$&49&143&[-1,3,-8,-8,1,1]&19051&$\alpha$&U\\
101&[7,-1,-2,-4,1,1]&13147&$\alpha$&50&144&[-2,-2,4,4,2,1]&19211&&61\\
102&[2,-4,0,0,0,1]&13259&&51&    &                     &19211&&62\\
103&[3,-1,4,-4,1,1]&13597&&0&145&[2,0,4,4,2,1]&19429&&63\\
104&[2,8,8,6,2,1]&13597&&$\nu$&146&[-2,-12,-22,-8,2,1]&19469&$\alpha$&U\\
105&[1,5,2,-12,1,1]&13723&&52&147&[-1,-5,-14,-8,1,1]&19531&$\alpha$&64\\
106&[6,4,6,4,2,1]&13829&$\alpha$&U&148&[4,0,-8,2,2,1]&19597&&0\\
107&[1,1,-4,-6,1,1]&13963&&$\nu$&149&[4,4,0,4,2,1]&20389&&$\nu$\\
108&[-2,6,2,-6,0,1]&13997&&53&150&[1,-3,2,4,1,1]&20533&$\alpha$&U\\
109&[4,-4,4,0,2,1]&13997&&$\nu$&151&[-2,6,0,2,2,1]&21061&$\alpha$&U\\
110&[-9,-1,4,0,1,1]&14149&&$\nu$&152&[-2,2,2,-4,2,1]&21211&$\alpha$&65\\
111&[15,13,-6,-6,1,1]&14197&&54&153&[-5,11,2,-12,1,1]&21283&&0\\
112&[2,-2,6,-2,2,1]&14293&&$\nu$&154&[-6,-4,4,-4,0,1]&21563&&66\\
113&[-3,-1,-2,-2,1,1]&14629&$\alpha$&U&155&[-14,-18,-10,-2,2,1]&21739&$\alpha$&U\\
114&[-46,48,6,-14,0,1]&14779&&$\nu$&156&[18,8,-12,-6,2,1]&21787&&67\\
115&[2,4,4,4,0,1]&14821&$\alpha$&U&157&[-3,-1,2,2,1,1]&22277&&68\\
116&[-2,4,2,-2,0,1]&15013&&$\nu$&158&[-2,8,-8,-6,2,1]&22291&&69\\
117&[1,-3,2,-4,1,1]&15227&&$\nu$&159&[-1,-3,-8,4,1,1]&22637&&0\\
118&[-2,0,2,0,2,1]&15307&&55&160&[-3,13,2,10,1,1]&22709&&$\nu$\\
119&[-2,2,4,4,0,1]&15373&$\alpha$&U&161&[2,0,-6,-4,2,1]&22787&$\alpha$&U\\
120&[3,7,0,0,1,1]&15493&$\alpha$&U&162&[1,9,6,2,1,1]&22861&&70\\
121&[-2,4,-2,0,2,1]&15581&&$\nu$&163&[-5,13,-4,-8,1,1]&23003&&71\\
122&[5,9,4,6,1,1]&15749&&56&164&[-3,-1,-4,-4,1,1]&23059&$\alpha$&U\\
123&[4,0,0,-2,2,1]&15749&$\alpha$&U&165&[1,-3,-2,4,1,1]&23131&&72\\
124&[2,-6,2,2,2,1]&15923&$\alpha$&U&&&23131&&73\\
125&[-2,0,10,8,0,1]&16139&&$\nu$&166&[2,-4,-2,0,2,1]&23251&&$\nu$\\
126&[2,-2,-10,-4,2,1]&16451&&$\nu$&167&[6,4,2,4,0,1]&23669&&$\nu$\\
127&[1,5,2,0,1,1]&16901&$\alpha$&U&168&[-6,2,4,-2,0,1]&24109&$\alpha$&0\\
128&[-6,4,2,-4,0,1]&16981&$\alpha$&U &169&[2,8,0,6,0,1]&24469&&74\\
129&[9,5,-6,-4,1,1]&17029&&$\nu$&&&24469&&75\\
130&[-7,5,4,-2,1,1]&17203&&$\nu$&170&[2,-4,2,2,0,1]&24533&&$\nu$\\
131&[-2,10,-12,-2,2,1]&17291&&57&171&[-6,4,6,-6,0,1]&24611&$\alpha$&U\\
132&[-15,13,6,-4,1,1]&17317&&58&172&[-7,-5,-2,-2,1,1]&24763&&$\nu$\\
\hline
\end{tabular}}
\end{table}

\begin{table}
\centering{
\begin{caption}{Curves $y^2 = g(x)$, their 2-division fields and conductors}\label{Curves}\end{caption}
\centering{
\begin{tabular}{|c|c|c|c|c|c|}
\hline
\#$C$&\#$F_0$&$g(x)$&$N$& $\epsilon$ \\
\hline 
1&1&[1,-4,8,-8,0,4]&277  & $\alpha$ \\
2&2&[1,-4,4,4,-8,4]&349  &  $\alpha$\\
3&3&[1,8,20,12,-8,4]&461  &  $\alpha$ \\
4&6&[1,0,0,4,-4,4]&797  &$\alpha$  \\
5&7&[1,4,0,-8,0,4]&971  &  $\alpha$ \\
6&8&[1,0,-4,8,-8,4]&997  &  $\alpha$ \\
7&9&[1,-4,4,0,-4,4]&1051  &  $\alpha$ \\
8&11&[-79,-304,-560,-200,-4,4]&1109  &  $\alpha$ \\
9&12&[1,4,4,-4,-4,4]&1109  &   $\alpha$\\
10&15&[1,0,-4,4,-4,4]&1637  &   \\
11&16&[5,-24,44,-36,8,4]&1811  & $\alpha$  \\
12&18&[1,4,4,4,8,4]&2243  &  $\alpha$ \\
13&20&[-3,-4,0,8,8,4]&2341  &  $\alpha$ \\
14&22&[5,-16,20,-8,-4,4]&2677  &  $\alpha$ \\
15&23&[1,0,0,4,8,4]&2693  &   \\
16&27&[1,4,-8,-4,4,4]&3251  & $\alpha$  \\
17&29&[9,-40,60,-32,0,4]&3499  &   \\
18&30&[1,0,0,4,-8,4]&3557 &   \\
19&31&[1,0,4,0,4,4]&3637  & $\alpha$  \\
20&32&[161,-360,284,-80,-4,4]&3701  & $\alpha$  \\
21&34&[1,-4,4,0,0,4]&3989  &   \\
22&37&[-3,8,-12,12,-8,4]&4157  &  $\alpha$ \\
23&39&[1,-4,8,-8,4,4]&4517  &  $\alpha$ \\
24&40&[-3,8,0,-12,4,4]&5059  &  $\alpha$ \\
25&41&[5,-20,-40,240,-600,500]&5227  &   \\
26&46&[5185,-6384,2664,-396,-4,4]&5651  &  $\alpha$ \\
27&47&[73,-180,152,-40,-8,4]&5867  &   \\
28&52&[1,4,0,-8,4,4]&6491  &   \\
29&54&[-3,4,4,-8,0,4]&6763  &   \\
30&57&[25,28,-12,-16,4,4]&7109  &   \\
31&62&[41,-148,160,-56,-4,4]&7877  &   \\
32&62&[1,8,12,-8,-8,4]&7877  &   \\
33&62&[73,-228,232,-84,0,4]&7877  &   \\
34&64&[-591,-1160,-792,-204,-4,4]&8243  &  $\alpha$ \\
35&66&[1,-8,20,-12,-8,4]&8803  &   \\
36&67&[1,-8,24,-28,4,4]&9091  & $\alpha$  \\
37&69&[1,-8,16,-8,-4,4]&9803  &   \\
38&70&[1,8,20,16,8,4]&9941  & $\alpha$  \\
39&72&[1,0,4,0,0,4]&10037  &   \\
40&77&[1,12,44,52,4,4]&10789  &   \\
41&78&[13,4,-20,-8,8,4]&10837  &   \\
42&79&[5,12,0,-12,0,4]&10853  &   \\
43&80&[-7,12,4,16,4,4]&10949  &  $\alpha$ \\
44&82&[1,-4,4,-4,8,4]&11117  &   \\
45&86&[1,12,44,44,-4,4]&11579  &   \\
46&88&[1,4,0,-4,4,4]&11971  &   \\
47&88&[1461041,-565424,78052,-4092,8,4]&11971  &   \\
48&97&[1,4,0,-8,-4,4]&12923  &  $\alpha$ \\
49&100&[1,12,32,28,8,4]&13147  & $\alpha$  \\
50&101&[1,-4,4,-4,4,4]&13147  &  $\alpha$ \\

\hline
\end{tabular}}}
\end{table}

\begin{table}
\centering{
\begin{tabular}{|c|c|c|c|c|c|}
\hline
\#$C$&\#$F_0$&$g(x)$&$N$ & $\epsilon$  \\
\hline 
51&102&[5,-28,48,-24,-4,4]&13259  &   \\
52&105&[1,-4,0,4,8,4]&13723  &   \\
53&108&[137,-356,328,-116,4,4]&13997  &   \\
54&111&[9,16,-4,-16,0,4]&14197  &   \\
55&118&[1,4,-8,-4,8,4]&15307  &   \\
56&122&[1,4,4,8,8,4]&15749  &   \\
57&131&[1,-4,4,0,-8,4]&17291  &   \\
58&132&[-3,8,-8,8,-8,4]&17317  &   \\
59&135&[1,0,0,-4,4,4]&17389  &  $\alpha$  \\
60&138&[-3,-20,-40,-20,4,4]&18077  &  $\alpha$ \\
61&144&[-247,552,-200,-136,4,4]&19211  &   \\
62&144&[-7,16,4,-16,0,4]&19211  &   \\
63&145&[-3,36,-144,192,-108,108]&19429  &   \\
64&147&[-11,-44,264,440,968,5324]&19531  &  $\alpha$ \\
65&152&[-3,-4,8,4,-8,4]&21211  & $\alpha$  \\
66&154&[-21167,-18908,-5996,-712,0,4]&21563  &   \\
67&156&[-3,-16,-28,-16,4,4]&21787  &   \\
68&157&[9,-32,40,-20,0,4]&22277  &   \\
69&158&[1,-4,8,-12,4,4]&22291  &   \\
70&162&[1,4,8,4,4,4]&22861  &   \\
71&163&[5,-36,76,-40,4,4]&23003  &   \\
72&165&[1909,-2652,1308,-236,-4,4]&23131  &   \\
73&165&[1,8,-12,-8,8,4]&23131  &   \\
74&169&[1,8,20,16,0,4]&24469  &   \\
75&169&[7309,-8208,3292,-504,4,4]&24469  &   \\
\hline
\end{tabular}}
\end{table}

\begin{table}[h]
\begin{caption} {Amiable fields among favorable fields} \label{data}\end{caption}
\centering{
{\tiny \begin{tabular}{|c|c|c|c|c|c|c|c|c|c|c|c|c|}
\hline
$j$ &0&1& 2&3&4& 5&6& 7&8&9& Total\\
\hline
A& 63563& 35507& 29047& 25450& 23684& 22099& 20500& 19505& 18773& 17981& 276109\cr
F& 63212& 35429& 28998& 25417& 23657& 22079& 20479& 19493& 18761& 17969& 275494\cr
$\alpha$& 7632& 4290& 3362& 2948& 2799& 2606& 2375& 2340& 2189& 2127&32668\cr 
\hline
\end{tabular}}}
\end{table}

\begin{table}[h]
\begin{caption}[FC2]
{Curves $y^2 = 1+4P(x)$ of large conductor with amiable fields}\label{LargeN}\end{caption}
\centering{
\small \begin{tabular}{|c|c||c|c|}
\hline
$P(x)$ & $N$ & $P(x)$  & $N$ \\
\hline
    [-9    0, -184, -136, -39, -1, 1]&  9882329341&  [10, 22, 7, -7, 0, 1] & 9891907261 \cr
     [11, 26, -7, -8, 0, 1] &       9893121157&    [11, 17, 3, -4, -2, 1]           &    9897613669\cr
   [-8428, -6910, -2025, -226, -1, 1] &    9898501189&     [-21, 6, 10, -1, 1, 1]     &    9911121709\cr  
    [87, -106, 56, -9, -2, 1]&   9934582709&   [-61, 50, 9, -13, 0, 1]      &    9982174061\cr
   [-33, 20, -1, 10, 1, 1] &    9987633941& [-2, -3, -15, -9, 0, 1]  &    9994370909\cr
\hline
\end{tabular}}
\end{table}

\end{document}